\documentclass[a4paper,notitlepage,twoside,reqno,10pt]{amsart}

\usepackage{bbm,pifont,latexsym,dcolumn,indentfirst}
\usepackage[hypertexnames=false,bookmarksopen=true,linktocpage=true,pdfstartview={XYZ null null 1.25}]{hyperref}
\usepackage{amsthm,amsmath,amssymb,amscd,amsfonts,mathrsfs}
\usepackage{color,graphicx,xcolor,graphics,subfigure,extarrows,caption2}
\usepackage{pdfpages,titletoc}

\usepackage{autonum}

\newtheorem{thm}{Theorem}[section]
\newtheorem{lem}[thm]{Lemma}
\newtheorem{cor}[thm]{Corollary}
\newtheorem{prop}[thm]{Proposition}

\newtheorem*{mainthm}{Main Theorem}

\theoremstyle{definition}
\newtheorem*{defi}{Definition}

\newtheorem*{rmk}{Remark}
\newtheorem*{ques}{Question}
\newtheorem*{nota}{Notations}

\newcommand{\DF}[2]{{\displaystyle\frac{#1}{#2}}}
\newcommand{\EC}{\widehat{\mathbb{C}}}
\newcommand{\C}{\mathbb{C}}
\newcommand{\D}{\mathbb{D}}

\newcommand{\N}{\mathbb{N}}
\newcommand{\Q}{\mathbb{Q}}
\newcommand{\R}{\mathbb{R}}

\newcommand{\Z}{\mathbb{Z}}
\newcommand{\OD}{\overline{\mathbb{D}}}

\newcommand{\MA}{\mathcal{A}}
\newcommand{\MB}{\mathcal{B}}
\newcommand{\MC}{\mathcal{C}}
\newcommand{\MD}{\mathcal{D}}
\newcommand{\MF}{\mathcal{F}}
\newcommand{\MH}{\mathcal{H}}

\newcommand{\MN}{\mathcal{N}}
\newcommand{\MO}{\mathcal{O}}
\newcommand{\MP}{\mathcal{P}}
\newcommand{\MQ}{\mathcal{Q}}
\newcommand{\MMR}{\mathcal{R}}

\newcommand{\IS}{\mathcal{IS}}

\newcommand{\ii}{\textup{i}}
\newcommand{\re}{\textup{Re}\,}
\newcommand{\im}{\textup{Im}\,}
\newcommand{\dd}{\textup{d}}

\newcommand{\Expo}{\mathbb{E}\textup{xp}}
\newcommand{\Log}{\mathbb{L}\textup{og}}

\newcommand{\diam}{\textup{diam}}
\newcommand{\dist}{\textup{dist}}

\newcommand{\cv}{\textup{cv}}
\newcommand{\cp}{\textup{cp}}
\newcommand{\HT}{\textup{HT}}
\newcommand{\Int}{\operatorname{int}}

\newcommand{\len}{\textup{len}}
\newcommand{\Comp}{\textup{Comp}}

\newcommand{\kc}{\textit{\textbf{k}}}

\usepackage{enumerate,enumitem}

\makeatletter\@addtoreset{equation}{section}\makeatother

\begin{document}

\author{MITSUHIRO SHISHIKURA}
\address{Department of Mathematics, Kyoto University, Kyoto 606-8502, Japan}
\email{mitsu@math.kyoto-u.ac.jp}

\author{FEI YANG}
\address{Department of Mathematics, Nanjing University, Nanjing 210093, P. R. China}
\email{yangfei@nju.edu.cn}

\title[High type quadratic Siegel disks]{The high type quadratic Siegel disks are Jordan domains}

\begin{abstract}
Let $\alpha$ be an irrational number of sufficiently high type and suppose $P_\alpha(z)=e^{2\pi\ii\alpha}z+z^2$ has a Siegel disk $\Delta_\alpha$ centered at the origin. We prove that the boundary of $\Delta_\alpha$ is a Jordan curve, and that it contains the critical point $-e^{2\pi\ii\alpha}/2$ if and only if $\alpha$ is a Herman number.
\end{abstract}



\date{\today}



\maketitle

{\setcounter{tocdepth}{1}
\tableofcontents
}

\section{Introduction}\label{introduction}

Let $f$ be a non-linear holomorphic function with $f(0)=0$ and $f'(0)=e^{2\pi\ii\alpha}$, where $0<\alpha<1$ is an irrational number. We say that $f$ is \textit{locally linearizable} at the fixed point 0 if there exists a holomorphic function defined near $0$ which conjugates $f$ to the \textit{rigid rotation} $R_\alpha(z)=e^{2\pi\ii\alpha}z$. The maximal region in which $f$ is conjugate to $R_\alpha$ is a simply connected domain called the \textit{Siegel disk} of $f$ centered at 0.

The existence of the Siegel disk of $f$ is dependent on the arithmetic condition of $\alpha\in(0,1)\setminus\Q$. Let
\begin{equation}
[0;a_1,a_2,\cdots]:=
\DF{1}{a_1
+\DF{1}{a_2
+\DF{1}{\ddots}}}
\end{equation}
be the \textit{continued fraction expansion} of $\alpha$. The rational numbers $p_n/q_n:=[0;$ $a_1$, $\cdots$, $a_n]$, $n\geqslant 1$, are the convergents of $\alpha$, where $p_n$ and $q_n$ are coprime positive integers. If $\alpha$ belongs to the \emph{Brjuno class}
\begin{equation}
\MB:=\{\alpha=[0;a_1,a_2,\cdots]\in (0,1)\setminus\Q~|~\mathop{\Sigma}\nolimits_{n=1}^\infty q_n^{-1}\log q_{n+1}<+\infty\},
\end{equation}
then any holomorphic germ $f$ with $f(0)=0$ and $f'(0)=e^{2\pi\ii\alpha}$ is locally linearizable at $0$ and hence $f$ has a Siegel disk centered at the origin \cite{Sie42, Brj71}. Yoccoz proved that the Brjuno condition is also necessary for the local linearization of the quadratic polynomial
\begin{equation}
P_\alpha(z):=e^{2\pi\ii\alpha}z+z^2:\C\to\C
\end{equation}
at the origin \cite{Yoc95}.

\subsection{Topology and obstructions of Siegel disk boundaries}

The dynamics in the Siegel disks is simple and one mainly concerns the properties on the boundaries. In the 1980s, Douady and Sullivan asked the following question (see \cite{Dou83}, \cite{Rog92a}):
\begin{ques}
Is the boundary of a Siegel disk a Jordan curve?
\end{ques}
This question is still open, even for quadratic polynomials. However, much progress has been made on this problem for various families of functions under preconditions. An irrational number $\alpha=[0;a_1,a_2,\cdots]$ is called \emph{bounded type} if $\sup_{n\geqslant 1}\{a_n\}<+\infty$. Douady-Herman, Zakeri, Yampolsky-Zakeri, Shishikura and Zhang, respectively, proved that the boundaries of bounded type Siegel disks of quadratic polynomials, cubic polynomials, some quadratic rational maps, all polynomials and all rational maps with degree at least two are quasi-circles (hence are Jordan curves) (see \cite{Dou87}, \cite{Her87}, \cite{Zak99}, \cite{YZ01}, \cite{Shi01}, \cite{Zha11}). This is also true for some transcendental entire functions (see \cite{Gey01}, \cite{Zha05}, \cite{Che06}, \cite{KZ09}, \cite{Zak10}, \cite{Yan13}, \cite{CE18}, \cite{ZFS20}).

An important breakthrough was made by Petersen and Zakeri in 2004. They proved that for almost all irrational number $\alpha$, the boundary of the Siegel disk of the quadratic polynomial $P_\alpha$ is a Jordan curve \cite{PZ04}. We refer these irrational numbers the \emph{PZ type}, i.e., $\log a_n=\MO(\sqrt{n})$ as $n\to\infty$, where $a_n$ is the $n$-th digit of the continued fraction expansion of $\alpha$. Recently, Zhang generalized this result to all polynomials \cite{Zha14} and obtained the same result for the sine family \cite{Zha16}.

\medskip
Suppose the closure of the Siegel disk of $f$ is compactly contained in the domain of definition of $f$. One may wonder what phenomena near the boundary of a Siegel disk prevents $f$ from having a larger linearization domain. Obviously, the presence of periodic cycles near the boundary is one of the reasons since any Siegel disk cannot contain periodic points except the center itself. It was proved by Avila and Cheraghi that under some condition on $\alpha$ every neighborhood of the Siegel disk of $P_\alpha$ contains infinitely many cycles \cite{AC18}, which is similar to the small cycle property that prevents linearization (see \cite{Yoc88} and \cite{Per92}).

On the other hand, note that any Siegel disk cannot contain a critical point. Hence the second question on the Siegel disk boundary is:
Does the boundary of a Siegel disk always contain a critical point?
The answer is no. Ghys and Herman gave the first examples of polynomials having a Siegel disk whose boundary does not contain a critical point (see \cite{Ghy84}, \cite{Her86} and \cite{Dou87}).

In relation to the results on the regularity\footnote{The word ``regularity'' here means the topological and geometric properties of the boundaries of the Siegel disks. See \cite{BC07}.} of the boundaries of the Siegel disks mentioned above (for the bounded type or PZ type rotation numbers), they also include the statement that the boundaries of those Siegel disks pass through at least one critical point. In particular, for the bounded type rotation numbers, Graczyk and \'{S}wi\c{a}tek proved a very general result: if an analytic function has a Siegel disk properly contained in the domain of holomorphy and the rotation number is of bounded type, then the boundary of the corresponding Siegel disk contains a critical point \cite{GS03}.

\medskip
Herman was one of the pioneers who studied the analytic diffeomorphisms on the circles \cite{Her79}. He introduced the following subset of irrational numbers.

\begin{defi}[{Herman numbers}]
Let $\MH$ be the set of irrational numbers $\alpha$ such that every orientation-preserving analytic circle diffeomorphism of rotation number $\alpha$ is analytically conjugate to the rigid rotation.
\end{defi}

Herman proved that the set $\MH$ is non-empty and contains a subset of Diophantine numbers \cite{Her79}. Yoccoz proved that $\MH$ contains all Diophantine numbers (and hence contains all bounded type and PZ type numbers), and also gave an arithmetic characterization of the numbers in $\MH$ \cite{Yoc02}.

\medskip
Suppose $f$ is an analytic function which has a Siegel disk properly contained in the domain of holomorphy. Ghys proved that if the rotation number belongs to $\MH$ and the boundary of the Siegel disk is a Jordan curve, then $f$ has a critical point in the boundary of the Siegel disk \cite{Ghy84}. Later, Herman generalized this result by dropping the topological condition on the Siegel disk boundary but requiring that the restriction of $f$ on the Siegel disk boundary is injective \cite{Her85} (see also \cite{Per97}). In particular, he proved that if a unicritical polynomial has a Siegel disk whose rotation number is contained in $\MH$, then the boundary of the Siegel disk contains a critical point. Recently, Ch\'{e}ritat and Roesch, Benini and Fagella, respectively, generalized this result to the polynomials with two critical values \cite{CR16} and to a special class of transcendental entire functions with two singular values \cite{BF18}.

For polynomials, Rogers proved that if the Siegel disk $\Delta$ is fixed and the rotation number is in $\MH$, then either $\partial\Delta$ contains a critical point or $\partial\Delta$ is an indecomposable continuum \cite{Rog98}. For the exponential map $E_\theta(z)=e^{2\pi\ii\theta}(e^z-1)$, it was proved by Herman that, if $E_\theta$ has a bounded Siegel disk $\Delta_\theta$, then $E_\theta$ is injective on $\partial\Delta_\theta$. Hence it follows from Herman's result that $\Delta_\theta$ is unbounded when $\theta\in\MH$ since $E_\theta$ has no critical points \cite{Her85}. Conversely, Herman, Baker and Rippon asked a question: if $\Delta_\theta$ is unbounded, is necessarily the singular value $-e^{2\pi\ii\theta}$ contained in $\partial\Delta_\theta$? Rippon showed that this is true for almost all $\theta$ \cite{Rip94} and the question was fully answered positively by Rempe \cite{Rem04} and independently by Buff and Fagella (unpublished). Moreover, Rempe also studied the Herman type Siegel disks of some other transcendental entire functions \cite{Rem08}.

\subsection{The statement of the main result}

The proofs of the regularity results for the bounded type and PZ type Siegel disks stated previously are all based on surgery: either quasiconformal or trans-quasiconformal. In these proofs, some pre-models, and usually, a single or a family of Blaschke products are needed. By surgery, the regularity and the existence of critical points on the boundaries of Siegel disks were proved at the same time.

\medskip
In this paper, without using surgeries we shall prove that the Siegel disks of some holomorphic maps are Jordan domains and that Herman type rotation number is also necessary for the existence of critical points on the Siegel disk boundaries. To this end, it requires us to restrict the rotation numbers to a special class since we need to use near-parabolic renormalization scheme. In \cite{IS08}, a renormalization operator $\MMR$ and a compact class $\MF$ that is invariant under $\MMR$ were introduced. All the maps in $\MF$ have a special covering structure. They have a neutral fixed point at the origin and a unique simple critical point in their domains of definition. The renormalization operator assigns a new map in $\MF$ to a given map of $\MF$ that is obtained by considering the return map to a sector landing at the origin. As a return map, one iterate of $\MMR f$ corresponds to many iterates of $f\in\MF$. To study very large iterates of $f$ near $0$, one hopes to repeat this process infinitely many times. However, to iterate $\MMR$ infinitely many times, the scheme requires the rotation number $\alpha$, where $f'(0)=e^{2\pi\ii\alpha}$, to be of \emph{high type}, that is, $\alpha$ belongs to
\begin{equation}
\HT_{N}:=\{\alpha=[0;a_1,a_2,\cdots]\in (0,1)\setminus\Q~|~a_n\geqslant N \text{ for all } n\geqslant 1\}
\end{equation}
for some big integer\footnote{The precise value of $N$ is not known. But the value of $N$ is likely to be not less than $20$. It is conjectured that a variation of the invariant class and renormalization may be defined for $N=1$.} $N\in\N$. In this paper we prove the following main result.

\begin{mainthm}\label{thm-main-1}
Let $\alpha$ be an irrational number of sufficiently high type and suppose $P_\alpha(z)=e^{2\pi\ii\alpha}z+z^2$ has a Siegel disk $\Delta_\alpha$ centered at the origin. Then the boundary of $\Delta_\alpha$ is a Jordan curve. Moreover, it contains the critical point $-e^{2\pi\ii\alpha}/2$ if and only if $\alpha$ is a Herman number.
\end{mainthm}

Note that $\HT_{N}$ has measure zero if $N\geqslant 2$. However, all the usual types of irrational numbers have non-empty intersections with $\HT_{N}$: bounded type, PZ type, Herman type and Brjuno type etc. In particular, $\HT_{N}$ contains some irrational numbers such that the Siegel disk boundary of $P_\alpha$ has the regularity studied in \cite{ABC04}, \cite{BC07} and the self-similarity studied in \cite{McM98b}. Rogers proved that the boundary of any bounded irreducible Siegel disk $\Delta$ is either tame: the conformal map from $\Delta$ to the unit disk has a continuous extension to $\partial\Delta$, or wild: $\partial\Delta$ is an indecomposable continuum \cite{Rog92b}. Recently, Ch\'{e}ritat constructed a holomorphic germ such that the corresponding Siegel disk is compactly contained in the domain of definition but the boundary is not locally connected \cite{Che11}. Our main theorem indicates that the boundaries of  quadratic Siegel disks should be tame.

As we have seen, in order to guarantee the existence of critical points on the boundaries of Siegel disks, Herman condition (i.e., the rotation number is of Herman type) appears usually as a requirement of sufficiency in most of the literature. As far as we know, the necessity only appears in \cite{BCR09}, where it proves that Herman condition is equivalent to the existence of a critical point on the boundary of the Siegel disks of a family of toy models.

In fact, besides the quadratic polynomials, the proof of the Main Theorem in this paper is also valid for all the maps in Inou-Shishikura's invariant class. Hence the Main Theorem is also true for some rational maps and transcendental entire functions. We would like to point out that it was proved in \cite{Yam08} and \cite{AL22} that the bounded type Siegel disks of the maps in Inou-Shishikura's class are quasi-disks if the rotation number is of sufficiently high type.

\medskip
By constructing topological models of the post-critical sets of the maps in the Inou-Shishikura's class for all high type numbers, Cheraghi gave an alternative proof of the Main Theorem independently (see \cite{Che22b}). Our proofs are different: we analyze the dynamics and carry out the computations in the renormalization tower directly.

Recently, Dudko and Lyubich made significant progress on the quadratic Siegel polynomials $P_\alpha$ \cite{DL22}. They proved that the restriction of $P_\alpha$ on the boundary of the Siegel disk $\Delta_\alpha$ of $P_\alpha$ is injective, which implies that $\partial\Delta_\alpha$ is not the whole Julia set of $P_\alpha$ (actually they proved a more general result for all $\alpha\in\R\setminus\Q$).

\subsection{Strategy of the proof}

Let $f_0$ be the normalized quadratic polynomial or a map in Inou-Shishikura's class (see \S\ref{subsec-IS-class}) satisfying $f_0(0)=0$ and $f_0'(0)=e^{2\pi\ii\alpha}$, where $\alpha$ is of Brjuno type and of sufficiently high type. For $n\geqslant 0$, let $f_{n+1}=\MMR f_n$ be the sequence of the maps which are generated by the near-parabolic renormalization operator $\MMR$. For each $n\geqslant 0$, we use $\MP_n$ to denote the perturbed petal of $f_n$ and $\Phi_n$ the corresponding perturbed Fatou coordinate (see definitions in \S\ref{subsec-near-para}).

In order to prove that the boundary of the Siegel disk of $f_0$ is a Jordan curve, we construct a sequence of continuous curves $(\gamma_0^n:[0,1]\to\C)_{n\in\N}$ in the perturbed Fatou coordinate plane of $f_0$ by using a renormalization tower.
Each $\gamma^n_0$  is obtained from $\gamma^0_n$ (in the perturbed Fatou coordinate plane of $f_n$) by going up through the renormalization tower, i.e., by lifting and then spreading around.
In Lemma \ref{lema-fixed-disk} we show that the inner radius of the Siegel disk $\Delta_n$ of $f_n$ is estimated by the Brjuno sum up to a multiplicative constant. Then we choose the suitable height of $\gamma_n^0$ such that $\Phi^{-1}_n(\gamma^0_n)$ is contained in the Siegel disk $\Delta_n$ of $f_n$. Consequently, $\Phi_0^{-1}(\gamma_0^n)$ with $n\in\N$ are curves in the Siegel disk of $f_0$.

The key ingredient is Proposition \ref{prop-Cauchy-sequence}: the sequence of the continuous curves $(\gamma_0^n:[0,1]\to\C)_{n\in\N}$ converges uniformly to a limit $\gamma^\infty:[0,1]\to\C$, which is also a continuous curve.
For the proof, we use a family of ``straight'' curves $\eta^0_n$ to encode the difference between $\gamma^0_n$ and $\gamma^1_n$ in the Fatou coordinate plane of $f_n$. The diameters of the $\eta^0_n$ are discussed in Step 2 of the proof. The diameters of the lifts of $\eta^0_n$ are estimated by two kinds of contraction: one is the uniform contraction with respect to the hyperbolic metrics in subdomains of the renormalization tower (see Lemma \ref{lema:exp-conv}) and the other is ``Brjuno-type arithmetic'' -- estimates from \S\ref{subsec-esti-2} (see also Lemma \ref{lema-go-up}). In conclusion, the oscillations of the curves $(\gamma_0^n:[0,1]\to\C)_{n\in\N}$ are bounded in terms of the Brjuno sum, i.e., $(\gamma_0^n:[0,1]\to\C)_{n\in\N}$ form an equicontinuous family.
Because of the contraction by going up the renormalization tower, the sequence $\Phi^{-1}_0(\gamma^n_0)$ converges exponentially fast towards the boundary of $\Delta_0$ (see Proposition \ref{prop-tend-to-bdy}).

\medskip
For the second part of the Main Theorem which concerns Herman condition, we construct a Jordan arc $\Gamma_0$ in the non-escaping set of $f_0$ which connects the unique critical value $\cv$ with the origin, where $\gamma_0:=\Phi_0(\Gamma_0)$ is contained in a half-infinite strip $\mho$ with finite width.
The existence of $\Gamma_0$ is proved in Lemma \ref{lem-Jordan-arc} and the proof is also based on the contraction via going up the renormalization tower. To apply the contraction property successfully, the shape of $\Phi_0^{-1}(\mho)$ needs to be controlled and this is Lemma \ref{lema-height} whose proof is given in the Appendix. The construction of $\Gamma_0$ guarantees that $\Gamma_n=\Expo\circ\Phi_{n-1}(\Gamma_{n-1})$ is also a Jordan arc connecting $\cv$ with the origin and $\gamma_n=\Phi_n(\Gamma_n)$ is contained in $\mho$ for all $n\geqslant 1$.

We study the homeomorphism $s_{\alpha_n}:=\Phi_n\circ \Expo:\gamma_{n-1}\to\gamma_n$ from the simple curve in one level of the renormalization to another.
Lemmas \ref{lema-all-greater} and \ref{lema-eventually-above} estimate the dynamics of the $s_{\alpha_n}$ in terms of the Brjuno sum.
Based on the sequence $(s_{\alpha_n})_{n\in\N}$, we define a new class of irrational numbers $\widetilde{\MH}_N$ which is a subset of Brjuno numbers, where $N$ is a large number. After comparing the properties of $s_{\alpha_n}$ and Yoccoz's arithmetic characterization of $\MH$, we prove that $\widetilde{\MH}_N$ is exactly equal to the set of high type Herman numbers (see Lemmas \ref{lem-subset-alpha} and \ref{lem-subset-alpha-yy}). On the other hand, we prove that the boundary of the Siegel disk of $f_0$ contains the critical value $\cv$ if and only if $\alpha\in\widetilde{\MH}_N$ (see Proposition \ref{prop-equi-alpha}). This implies that the second part of the Main Theorem holds.

\subsection{Some observations}

There are several applications of Inou-Shishikura's invariant class. The first remarkable application is that Buff and Ch\'{e}ritat used it as one of the main tools to prove the existence of Julia sets of quadratic polynomials with positive area \cite{BC12}. Recently, Cheraghi and his collaborators have found several other important applications. In \cite{Che13} and \cite{Che19}, Cheraghi developed several elaborate analytic techniques based on Inou-Shishikura's results. The tools in \cite{Che13} and \cite{Che19} have led to part of the recent major progresses on the dynamics of quadratic polynomials. For examples, the Feigenbaum Julia sets with positive area (which is different from the examples in \cite{BC12}) have been found in \cite{AL22}, the Marmi-Moussa-Yoccoz conjecture for rotation numbers of high type has been proved in \cite{CC15}, the local connectivity of the Mandelbrot set at some infinitely satellite renormalizable points was proved in \cite{CS15}, some statistical properties of the quadratic polynomials was depicted in \cite{AC18}, the topological structure and the Hausdorff dimension of high type irrationally indifferent attractors were characterized in \cite{Che22b} and \cite{CDY20} respectively.

Recently, Ch\'{e}ritat generalized the near-parabolic renormalization theory to the unicritical families of any finite degrees \cite{Che22}. See also \cite{Yan21p} for the corresponding theory of local degree three. Hence there is a hope to generalize the Main Theorem in this paper to all unicritical polynomials.

\medskip
\noindent\textbf{Acknowledgements.} We would like to thank Xavier Buff and Arnaud Ch\'{e}ritat for helpful discussions and offering the manuscript \cite{BCR09}. We are also very grateful to Davoud Cheraghi for pointing out a gap in an earlier version of the paper and providing many invaluable comments and suggestions. The second author was indebted to Institut de Math\'{e}matiques de Toulouse for its hospitality during his visit in 2014/2015 where and when partial of this paper was written. He would also like to thank Davoud Cheraghi and Arnaud Ch\'{e}ritat for persistent encouragements.
We are very grateful to both referees for their very insightful comments and suggestions.
This work was supported by NSFC (grant Nos.\,12222107, 12071210), NSF of Jiangsu Province (grant No.\,BK20191246) and the CSC program (2014/2015).

\begin{nota}
We use $\N$, $\N^+$, $\Z$, $\Q$, $\R$ and $\C$ to denote the set of all natural numbers, positive integers, integers, rational numbers, real numbers and complex numbers, respectively. The Riemann sphere and the unit disk are denoted by $\EC=\C\cup\{\infty\}$ and $\D=\{z\in\C:|z|<1\}$ respectively. A round disk in $\C$ is denoted by $\D(a,r)=\{z\in\C:|z-a|<r\}$ and $\OD(a,r)$ is its closure. Let $x\in\R$ be a non-negative number, we use $\lfloor x\rfloor$ to denote the integer part of $x$.

For a set $X\subset \C$ and a number $\delta>0$, let $B_\delta(X):=\bigcup_{z\in X}\D(z,\delta)$ be the $\delta$-neighborhood of $X$. For a number $a\in\C$ and a set $X\subset \C$, we denote $aX:=\{az:z\in X\}$ and $X\pm a:=\{z\pm a:z\in X\}$. Let $A,B$ be two subsets in $\C$. We say that $A$ is \textit{compactly contained} in $B$ if the closure of $A$ is compact and contained in the interior $\Int(B)$ of $B$ and we denote it by $A\Subset B$. We use $\diam(X)$ to denote the Euclidean diameter of a set $X\subset\C$ and $\len(\gamma)$ the Euclidean length of a rectifiable curve $\gamma\subset\C$.
\end{nota}

\section{Near-parabolic renormalization scheme}

In this section, we summarize some results in \cite{IS08}, \cite{BC12}, \cite{AC18} and \cite{Che19} which will be used in this paper. Parts of the theories can be also found in \cite{Shi98} and \cite{Shi00a}.

\subsection{Inou-Shishikura's class}\label{subsec-IS-class}

Let $P(z):=z(1+z)^2$ be a cubic polynomial with a parabolic fixed point at $0$ with multiplier $1$. Then $P$ has a critical point $\cp_{P}:=-1/3$ which is mapped to the critical value $\cv_P:=-{4}/{27}$. It has also another critical point $-1$ which is mapped to $0$. Consider the ellipse
\begin{equation}\label{ellipse}
E:=\left\{x+y\ii\in\C:\left(\frac{x+0.18}{1.24}\right)^2+\left(\frac{y}{1.04}\right)^2\leqslant 1\right\}
\end{equation}
and define\footnote{\,The domain $U$ is denoted by $V$ in \cite{IS08}. }
\begin{equation}\label{U-and-psi-1}
U:=\psi_1(\EC\setminus E), \text{~where~} \psi_1(z):=-\frac{4z}{(1+z)^2}.
\end{equation}
The domain $U$ is symmetric about the real axis, contains the parabolic fixed point $0$ and the critical point $\cp_P$, but $\overline{U}\cap(-\infty,-1]=\emptyset$ (see \cite[\S 5.A]{IS08} and Figure \ref{Fig_U-zoom}).

\begin{figure}[!htpb]
  \setlength{\unitlength}{1mm}
  \centering
  \includegraphics[height=55mm]{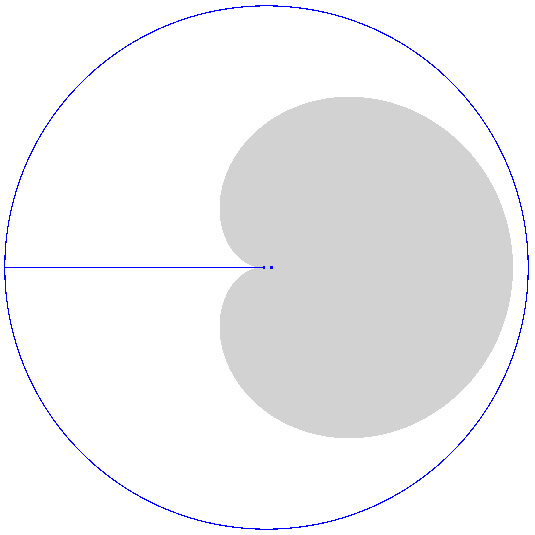}\hskip0.3cm
  \includegraphics[height=55mm]{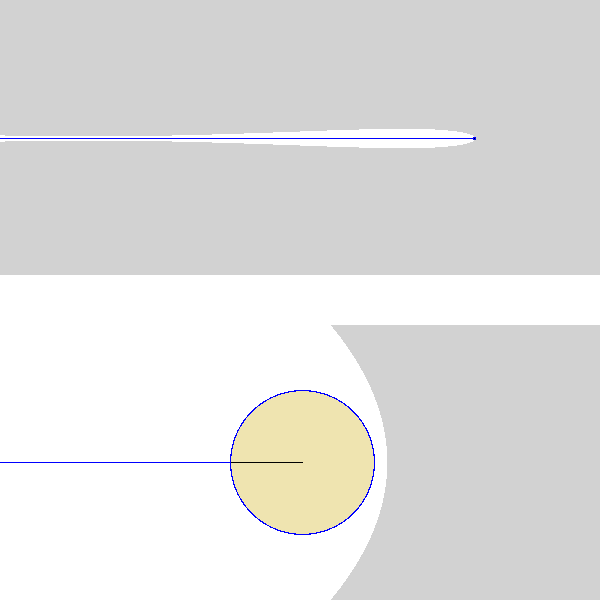}
  \put(-25,48){$U$}
  \put(-10,41){$-1$}
  \put(-30,9){$-1$}
  \put(-35,4){$B$}
  \put(-11,12){$U$}
  \put(-101,40){$U'$}
  \put(-76,32){$U$}
  \put(-84,26){$0$}
  \put(-89,23){$-1$}
  \caption{The domains $U$ (the gray part), $U'$ (the white region bounded by the blue curves, see \eqref{equ-U-pri} for the definition) and their successive zooms near $-1$. The outer boundary of $U'$ looks like a circle with radius about $35$ and the rightmost point of $U$ is about $32.2$. The widths of these pictures are $72$, $0.6$ and $0.0075$ respectively. It can be seen clearly from these pictures that $\overline{U}\cap(-\infty,-1]=\emptyset$ and $U\Subset U'$.}
  \label{Fig_U-zoom}
\end{figure}

For a given function $f$, we denote its domain of definition by $U_f$. Following \cite[\S 4]{IS08}, we define a class of maps\footnote{\,The definition of $\IS_0$ is based on the class $\MF_1$ in \cite{IS08}. There the conformal map $\varphi$ in the definition of $\IS_0$ is required to have a quasiconformal extension to $\C$. This condition is used by Inou and Shishikura to prove the uniform contraction of the near-parabolic renormalization operator under the Teichm\"{u}ller metric. We modify the definition here since we will not use this property in this paper.}
\begin{equation}
\IS_0:=
\left\{f=P\circ\varphi^{-1}:U_f\to\C
\left|
\begin{array}{l}
0\in U_f \text{ is open in }\C, ~\varphi:U\to U_f \text{ is} \\
\text{conformal},~\varphi(0)=0 \text{ and } \varphi'(0)=1
\end{array}
\right.
\right\}.
\end{equation}
Each map in this class has a parabolic fixed point at $0$, a unique critical point at $\cp_f:=\varphi(-1/3)\in U_f$ and a unique critical value at $\cv:=-4/27$ which is independent of $f$.

For $\alpha\in\R$, we define
\begin{equation}
\IS_\alpha:=\{f(z)=f_0(e^{2\pi\ii\alpha}z):e^{-2\pi\ii\alpha}\,U_{f_0}\to\C ~|~f_0\in\IS_0\}.
\end{equation}
For convenience, we normalize the quadratic polynomials to
\begin{equation}
Q_\alpha(z)= e^{2\pi\ii\alpha}z+\frac{27}{16}e^{4\pi\ii\alpha}z^2
\end{equation}
such that all $Q_\alpha$ have the same critical value $-4/27$ as the maps in $\IS_\alpha$. In particular, $Q_\alpha=Q_0\circ R_\alpha$, where $R_\alpha(z)=e^{2\pi\ii\alpha}z$. We would like to mention that the quadratic polynomial $Q_\alpha$ is not in the class $\IS_\alpha$.

\begin{thm}[{Leau-Fatou \cite[\S 10]{Mil06} and Inou-Shishikura \cite{IS08}}]\label{thm-IS-attr-rep-1}
For all $f\in\IS_0\cup\{Q_0\}$, there exist two simply connected domains $\MP_{attr,f}$, $\MP_{rep,f}\subset U_f$ and two univalent maps $\Phi_{attr,f}:\MP_{attr,f}\to\C$, $\Phi_{rep,f}:\MP_{rep,f}\to\C$ such that
\begin{enumerate}
\item $\MP_{attr,f}$ and $\MP_{rep,f}$ are bounded by piecewise analytic curves and are compactly contained in $U_f$, $\cp_f\in\partial \MP_{attr,f}$ and $\partial \MP_{attr,f}\cap\partial \MP_{rep,f}=\{0\}$;
\item The image $\Phi_{attr,f}(\MP_{attr,f})$ is a right half plane and $\Phi_{rep,f}(\MP_{rep,f})$ is a left half plane; and
\item $\Phi_{attr,f}(f(z))=\Phi_{attr,f}(z)+1$ for $z\in\MP_{attr,f}$ and $\Phi_{rep,f}^{-1}(\zeta)=f(\Phi_{rep,f}^{-1}(\zeta-1))$ for $\zeta\in\Phi_{rep,f}(\MP_{rep,f})$.
\end{enumerate}
\end{thm}

\noindent\textbf{Normalization of $\Phi_{attr,f}$ and $\Phi_{rep,f}$.}
The univalent map $\Phi_{attr,f}$ (resp. $\Phi_{rep,f}$) in Theorem \ref{thm-IS-attr-rep-1} is called an \emph{attracting (resp. repelling) Fatou coordinate} of $f$ and $\MP_{attr,f}$ (resp. $\MP_{rep,f}$) is called an \emph{attracting (resp. repelling) petal}. The attracting Fatou coordinate $\Phi_{attr,f}$ can be naturally extended to the immediate attracting basin $\MA_{attr,f}$ of $0$. Specifically, for $z\in \MA_{attr,f}$ such that $f^{\circ k}(z)\in\MP_{attr,f}$ with $k\geqslant 0$, one can define
\begin{equation}
\Phi_{attr,f}(z):=\Phi_{attr,f}(f^{\circ k}(z))-k.
\end{equation}
Since $\Phi_{attr,f}$ is unique up to an additive constant, we \emph{normalize} it by $\Phi_{attr,f}(\cp_f)=0$. Therefore, we have  $\Phi_{attr,f}(\MP_{attr,f})=\{\zeta\in\C:\re\zeta>0\}$.

Every $f\in\IS_0\cup\{Q_0\}$ can be written as $f(z)=z+a_2 z^2+a_3 z^3+\MO(z^4)$ in a neighborhood of $0$, where $a_2\neq 0$.
For $z$ in a component $\Omega_f$ of $\MA_{attr,f}\cap\MP_{rep,f}$ such that $\im\Phi_{rep,f}(z)\to+\infty$ as $z\to 0$, we have (see \cite[Proposition 2.2.1]{Shi00a}):
\begin{equation}
\begin{split}
\Phi_{attr,f}(z)&=-\frac{1}{a_2 z}-\gamma\log\Big(-\frac{1}{a_2 z}\Big)+C_{attr}+o(1), \\
\Phi_{rep,f}(z)&=-\frac{1}{a_2 z}-\gamma\log\Big(-\frac{1}{a_2 z}\Big)+C_{rep}+o(1),
\end{split}
\end{equation}
where $\gamma=1-{a_3}/{a_2^2}$ is the \textit{iterative residue} of $f$ and $C_{attr}$, $C_{rep}$ are constants.
Since $\Phi_{rep,f}$ is also unique up to an additive constant, we \emph{normalize} it by setting $C_{rep}:=C_{attr}$, i.e.,
$\Phi_{attr,f}(z)-\Phi_{rep,f}(z)\to 0$ as $z\to 0$ in $\Omega_f$.

\subsection{Near-parabolic renormalization}\label{subsec-near-para}

We need to consider the case that a sequence of functions converges to a limiting function and the neighborhoods of a function need to be defined.

\begin{defi}[{Neighborhoods of a function}]
Let $f:U_f\to\C$ be a given function. A \emph{neighborhood} of $f$ is
\begin{equation}
\MN=\MN(f;K,\varepsilon)=\left\{g:U_g\to\EC\,\left|\, K\subset U_g \text{~and~} \mathop{\textup{sup}}\limits_{z\in K}d_{\EC}(g(z),f(z))<\varepsilon\right.\right\},
\end{equation}
where $d_{\EC}$ denotes the spherical distance, $K$ is a compact subset contained in $U_f$ and $\varepsilon>0$. A sequence $(f_n)$ is said to \textit{converge} to $f$ \emph{uniformly on compact sets} if for any neighborhood $\MN$ of $f$, there exists $n_0>0$ such that $f_n\in\MN$ for all $n\geqslant n_0$.
\end{defi}

If $f\in\bigcup_{\alpha\in[0,1)}\IS_\alpha\cup\{Q_\alpha\}$, we denote by $\alpha_f\in[0,1)$ the rotation number of $f$ at the origin, i.e., the real number $\alpha_f\in[0,1)$ so that $f'(0)=e^{2\pi\ii\alpha_f}$. If $\alpha_f>0$ is small, besides the origin, the map $f$ has another fixed point $\sigma_f\neq0$ near $0$ in $U_f$, which depends continuously on $f$ (see \cite[\S 3.2]{Shi00a} or \cite[Lemma 9, p.\,707]{BC12}).

\begin{prop}[{\cite[Proposition 12, p.\,707]{BC12}, see Figure \ref{Fig_perturbed-Fatou-coor}}]\label{prop-BC-prop-12}
There exist $\kc\in\N^+$ and $\varepsilon_1>0$ satisfying $\lfloor\tfrac{1}{\varepsilon_1}\rfloor-\kc>1$, such that for all $f\in\IS_\alpha\cup\{Q_\alpha\}$ with $\alpha\in(0,\varepsilon_1]$, there exist a Jordan domain $\MP_f\subset U_f$ and a univalent map $\Phi_f:\MP_f\to\C$, such that
\begin{enumerate}
\item $\MP_f$ contains $\cv$ and it is bounded by two arcs joining $0$ and $\sigma_f$;
\item $\Phi_f(\cv)=1$, $\Phi_f(\MP_f)=\{\zeta\in\C:0<\re \zeta<\lfloor\tfrac{1}{\alpha_f}\rfloor-\kc\}$ with $\im \Phi_f(z)\to +\infty$ as $z\to 0$ and $\im \Phi_f(z)\to -\infty$ as $z\to \sigma_f$ in $\MP_f$;
\item If $z\in\MP_f$ and $\re\Phi_f(z)<\lfloor\tfrac{1}{\alpha_f}\rfloor-\kc-1$, then $f(z)\in\MP_f$ and $\Phi_f(f(z))=\Phi_f(z)+1$; and
\item If $(f_n)$ is a sequence of maps in $\bigcup_{\alpha\in(0,\varepsilon_1]}\IS_\alpha\cup\{Q_\alpha\}$ converging to a map $f_0\in\IS_0\cup\{Q_0\}$, then any compact set $K\subset\MP_{attr,f_0}$ is contained in $\MP_{f_n}$ for $n$ large enough and the sequence $(\Phi_{f_n})$ converges to $\Phi_{attr,f_0}$ uniformly on $K$; Moreover, any compact set $K\subset\MP_{rep,f_0}$ is contained in $\MP_{f_n}$ for $n$ large enough and the sequence $(\Phi_{f_n}-\frac{1}{\alpha_{f_n}})$ converges to $\Phi_{rep,f_0}$ uniformly on $K$.
\end{enumerate}
\end{prop}

\begin{figure}[!htpb]
  \setlength{\unitlength}{1mm}
  \centering
  \includegraphics[width=0.95\textwidth]{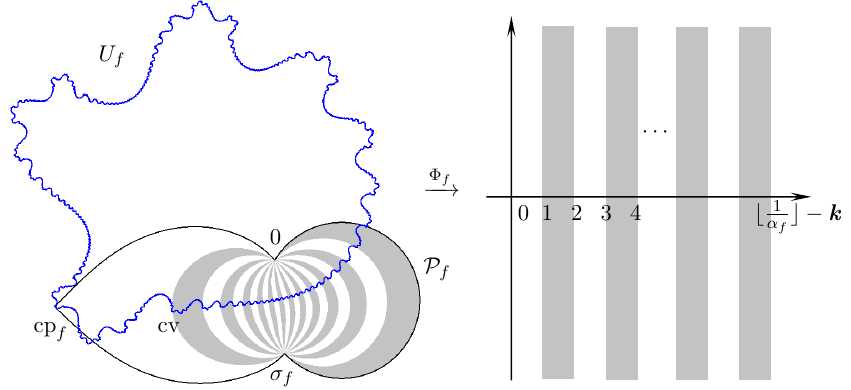}
  \caption{The perturbed Fatou coordinate $\Phi_f$ and its domain of definition $\MP_f$. The image of $\MP_f$ under $\Phi_f$ has been colored accordingly by the same color on the right. The blue set on the left depicts the forward orbit of the critical point $\cp_f$.}
  \label{Fig_perturbed-Fatou-coor}
\end{figure}

Proposition \ref{prop-BC-prop-12} was proved in \cite{BC12} only for Inou-Shishikura's class. However, when $f=Q_\alpha$ with sufficiently small $\alpha>0$, the existence of the domain $\MP_f$ and the coordinate $\Phi_f:\MP_f\to\C$ satisfying the properties in the above proposition is classic (see \cite{Shi00a}). The map $\Phi_f$ in Proposition \ref{prop-BC-prop-12} is called the \emph{(perturbed) Fatou coordinate} of $f$ and $\MP_f$ is called a \emph{(perturbed) petal}.

\begin{defi}[{see Figure \ref{Fig_near-para-norm-defi}}]
Let $f\in\IS_\alpha\cup\{Q_\alpha\}$ with $\alpha\in(0,\varepsilon_1]$, where $\varepsilon_1>0$ is the constant introduced in Proposition \ref{prop-BC-prop-12}. Define
\begin{equation}\label{defi-C-f-alpha}
\begin{split}
\MC_f:=&\,\{z\in\MP_f:1/2\leqslant\re\Phi_f(z)\leqslant 3/2 \text{~and~} -2<\im\Phi_f(z)\leqslant 2\}, \text{~and}\\
\MC_f^\sharp:=&\,\{z\in\MP_f:1/2\leqslant\re\Phi_f(z)\leqslant 3/2 \text{~and~} \im\Phi_f(z)\geqslant 2\}.
\end{split}
\end{equation}
Note that $\cv=-4/27\in \Int\, \MC_f$ and $0\in\partial \MC_f^\sharp$.
\end{defi}

\begin{prop}[{\cite[Proposition 2.7]{Che19}, see Figure \ref{Fig_near-para-norm-defi}}]\label{prop-CC-2}
There exist constants $\varepsilon_1'\in(0,\varepsilon_1]$ and $\kc_0\in\N^+$ such that for all $f\in\IS_\alpha\cup\{Q_\alpha\}$ with $\alpha\in(0,\varepsilon_1']$, there exists a positive integer $k_f\in[1,\kc_0]$ such that
\begin{enumerate}
\item For all $1\leqslant k\leqslant k_f$, the unique connected component $(\MC_f^\sharp)^{-k}$ of $f^{-k}(\MC_f^\sharp)$ that contains $0$ in its closure is relatively compact in $U_f$ and $f^{\circ k}:(\MC_f^\sharp)^{-k}\to\MC_f^\sharp$ is an isomorphism, and the unique connected component $\MC_f^{-k}$ of $f^{-k}(\MC_f)$ that intersects $(\MC_f^\sharp)^{-k}$ is relatively compact in $U_f$ and $f^{\circ k}:\MC_f^{-k}\to\MC_f$ is a covering of degree $2$ ramified above $\cv$; and
\item $k_f$ is the \emph{smallest} positive integer such that $\MC_f^{-k_f}\cup(\MC_f^\sharp)^{-k_f}\subset\{z\in\MP_f:0<\re\Phi_f(z)<\lfloor\tfrac{1}{\alpha_f}\rfloor-\kc-\tfrac{1}{2}\}$.
\end{enumerate}
\end{prop}

The same statement as Proposition \ref{prop-CC-2} without the uniform bound of $k_f$ is proved in \cite[Proposition 13, p.\,713]{BC12}. For the corresponding statements of Propositions \ref{prop-BC-prop-12} and \ref{prop-CC-2} with $\alpha\in\C$ (specifically, when $|\arg\alpha|<\pi/4$ and $|\alpha|$ is small), we refer to \cite[\S 2]{CS15}.

\begin{figure}[!htpb]
  \setlength{\unitlength}{1mm}
  \centering
  \includegraphics[width=0.95\textwidth]{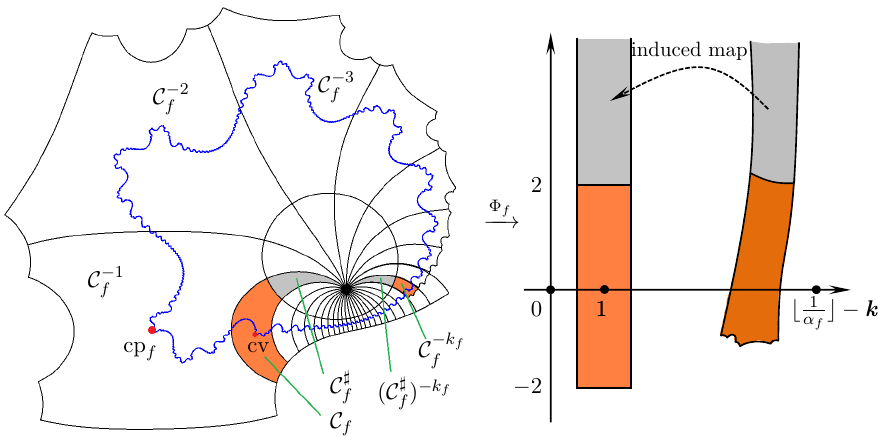}
  \put(-24,9){\small{$\Phi_f(S_f)$}}
  \put(-45,3.5){\small{$\Phi_f(\MC_f\cup\MC_f^\sharp)$}}
  \put(-70,22.5){\footnotesize{$S_f$}}
  \caption{Left: The sets $\MC_f$, $\MC_f^\sharp$ and some of their preimages. The blue set depicts the forward orbit of the critical point $\cp_f$. Right: The images of $\MC_f\cup\MC_f^\sharp$ and $S_f$ under the perturbed Fatou coordinate $\Phi_f$ and it shows how the near-parabolic renormalization map is induced.}
  \label{Fig_near-para-norm-defi}
\end{figure}

\begin{defi}[{Near-parabolic renormalization, see Figure \ref{Fig_near-para-norm-defi}}]
For $f\in\IS_\alpha\cup\{Q_\alpha\}$ with $\alpha\in(0,\varepsilon_1']$, define
\begin{equation}
S_f:=\MC_f^{-k_f}\cup(\MC_f^\sharp)^{-k_f},
\end{equation}
and consider the map
\begin{equation}
\Phi_f\circ f^{\circ k_f}\circ\Phi_f^{-1}:\Phi_f(S_f)\to\C.
\end{equation}
This map commutes with the translation by one. Hence it projects by the \textit{modified} exponential map\footnote{Note that $\Expo(0)=-4/27$ is a critical value of $\MMR f$ and $\Expo(+\infty\ii)=0$. In some literature, the modified exponential map is defined as $\zeta\mapsto -\frac{4}{27} e^{2\pi\ii \zeta}$ so that $(\MMR f)'(0)=e^{-2\pi\ii/\alpha_f}$. In order to apply the classical continued fraction expansion conveniently, in this paper we put a complex conjugacy $s$ in the definition of $\Expo$.}
\begin{equation}\label{equ-modify-exp}
\Expo(\zeta):=-\frac{4}{27} s\,(e^{2\pi\ii \zeta})
\end{equation}
to a well-defined map $\MMR f$ which is defined on a set punctured at zero, where $s:z\mapsto\overline{z}$ is the complex conjugacy. One can check that $\MMR f$ extends across zero and satisfies  $(\MMR f)(0)=0$ and $(\MMR f)'(0)=e^{2\pi\ii/\alpha_f}$. The map $\MMR f$ is called the \emph{near-parabolic renormalization}\footnote{This is the \textit{top} near-parabolic renormalization and the \textit{bottom} near-parabolic renormalization around the fixed point $\sigma_f$ can be defined similarly. See \cite[\S 3]{IS08}.} of $f$.
\end{defi}

Let $P(z)=z(1+z)^2$ be the cubic polynomial introduced at the beginning of \S\ref{subsec-IS-class}. Define
\begin{equation}\label{equ-U-pri}
U':=P^{-1}(\D(0,\tfrac{4}{27}e^{4\pi}))\setminus ((-\infty,-1]\cup \overline{B}),
\end{equation}
where $B$ is the connected component of $P^{-1}(\D(0,\frac{4}{27}e^{-4\pi}))$ containing $-1$. By an explicit calculation, one can prove that $\overline{U}\subset U'$ (see \cite[Proposition 5.2]{IS08} and Figure \ref{Fig_U-zoom}).

\begin{thm}[{\cite[Main Theorem 3]{IS08}}]\label{thm-IS-attr-rep-3}
For every $f=P\circ\varphi^{-1}\in\IS_\alpha$ or $f=Q_\alpha$ with $\alpha\in(0,\varepsilon_1']$, the near-parabolic renormalization $\MMR f$ is well-defined and the restriction of $\MMR f$ in a domain containing $0$ can be written as $P\circ\psi^{-1}\in\IS_{1/\alpha}$. Moreover, $\psi$ extends to a univalent function on $e^{-2\pi\ii/\alpha}\,U'$.
\end{thm}

From Theorem \ref{thm-IS-attr-rep-3} we know that the near-parabolic renormalization of $\MMR f$ can be also defined if the fractional part of $1/\alpha$ is contained in $(0,\varepsilon_1']$. This implies that the near-parabolic renormalization operator $\MMR$ can be applied infinitely many times to $f$ if $\alpha$ is of sufficiently high type.

\subsection{Some sets in the Fatou coordinate planes}\label{subsec-esti-1}

For a set $X$ in $\C$, we use $\Int(X)$ to denote the interior of $X$.
Let $f\in\IS_\alpha\cup\{Q_\alpha\}$ with $\alpha\in(0,\varepsilon_1']$. We define a set in the Fatou coordinate plane of $f$:
\begin{equation}\label{equ-MD-tilde-f}
\widetilde{\MD}_f:=\Int\Big(\Phi_f(\MP_f)\cup\bigcup_{j=0}^{b_f}(\Phi_f(S_f)+j)\Big),
\end{equation}
where $b_f:=k_f+\lfloor\tfrac{1}{\alpha}\rfloor-\kc -2$ is the largest integer\footnote{In particular, from the proof one can see that Lemma \ref{lema:Phi-inverse} will not be true if $b_f$ is chosen as $k_f+\lfloor\tfrac{1}{\alpha}\rfloor-\kc -1$.} such that one can extend $\Phi_f^{-1}:\Phi_f(\MP_f)\to\MP_f$ holomorphically to a domain like $\widetilde{\MD}_f$. See Figure \ref{Fig_chi-map}.

\begin{lem}\label{lema:Phi-inverse}
The map $\Phi_f^{-1}:\Phi_f(\MP_f)\to\MP_f$ can be extended to a holomorphic map
\begin{equation}
\Phi_f^{-1}:\widetilde{\MD}_f\to \MP_f\cup\bigcup_{j=0}^{k_f}f^{\circ j}(S_f),
\end{equation}
such that for all $\zeta\in\C$ with $\zeta,\zeta+1\in\widetilde{\MD}_f$, then $\Phi_f^{-1}(\zeta+1)=f\circ\Phi_f^{-1}(\zeta)$.
\end{lem}

This lemma has been proved in \cite[Lemma 1.8]{AC18}. For completeness and clarifying some ideas we include a sketch of the construction of $\Phi_f^{-1}$ here.

\begin{proof}
By \eqref{defi-C-f-alpha}, the definition of $S_f$, Propositions \ref{prop-BC-prop-12}(b) and \ref{prop-CC-2}(a), we have $f^{\circ k_f}(S_f)=\MC_f\cup\MC_f^\sharp$ and $f^{\circ j}(S_f)$ is well-defined for all $0\leqslant j\leqslant b_f$. If $\zeta\in\widetilde{\MD}_f\setminus\Phi_f(\MP_f)$, then there exists an integer $j\in[1,b_f]$ so that $\zeta\in\Phi_f(S_f)+j$. For such $\zeta$ we define
\begin{equation}
\Phi_f^{-1}(\zeta):=f^{\circ j}(\Phi_f^{-1}(\zeta-j)).
\end{equation}
Note that there may exist two choices\footnote{For example, this happens when $\zeta$ lies on $\big(\Phi_f(S_f)+j\big)\cap \big(\Phi_f(S_f)+j+1\big)$ for $1\leqslant j\leqslant b_f-1$.} of $j$ for some point $\zeta$.
Assume that $\zeta\in\Phi_f(S_f)+j'$ for some $j'\in[1,b_f]$ and $j'\neq j$. Then $|j'-j|=1$. Without loss of generality, we assume that $j'=j+1$. By Proposition \ref{prop-BC-prop-12}(c), we have $\Phi_f^{-1}(\zeta+1)=f\circ\Phi_f^{-1}(\zeta)$ for all $\zeta\in\C$ with $\zeta,\zeta+1\in\Phi_f(\MP_f)$. Thus we have
\begin{equation}
f^{\circ j'}(\Phi_f^{-1}(\zeta-j'))=f^{\circ (j'-1)}(\Phi_f^{-1}(\zeta-j'+1))=f^{\circ j}(\Phi_f^{-1}(\zeta-j)).
\end{equation}
This implies that $\Phi_f^{-1}$ is well-defined in $\widetilde{\MD}_f$ and it is straightforward to check that $\Phi_f^{-1}$ is holomorphic. Finally a completely similar calculation shows that $\Phi_f^{-1}(\zeta+1)=f\circ\Phi_f^{-1}(\zeta)$ for all $\zeta\in\C$ with $\zeta,\zeta+1\in\widetilde{\MD}_f$.
\end{proof}

\begin{figure}[!htpb]
  \setlength{\unitlength}{1mm}
  \centering
  \includegraphics[width=0.45\textwidth]{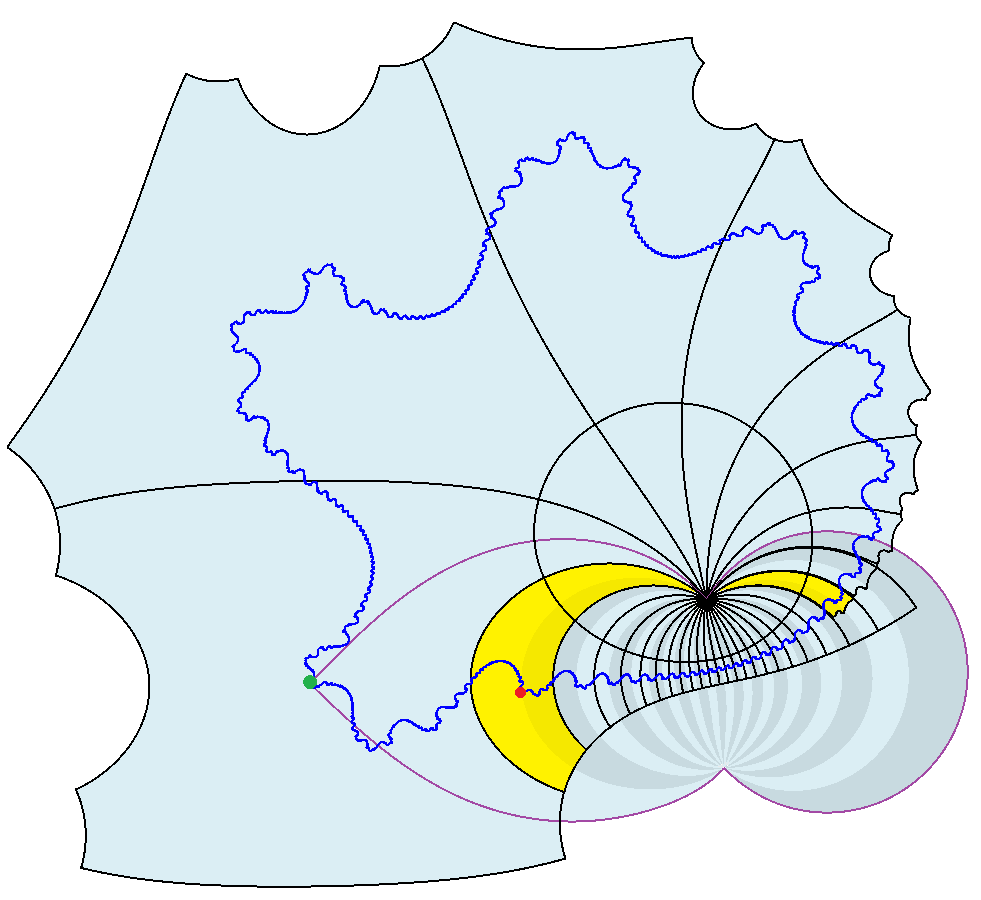}
  \includegraphics[width=0.53\textwidth]{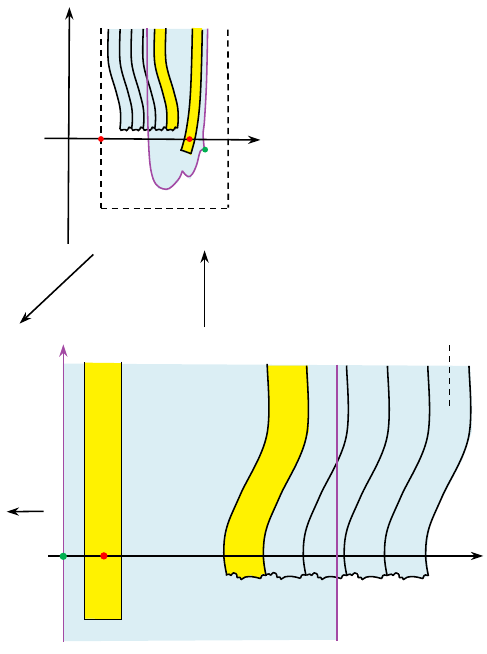}
  \put(-39,6.5){\small{$\Phi_f(S_f)$}}
  \put(-57,1.5){\small{$\Phi_f(\MC_f\cup\MC_f^\sharp)$}}
  \put(-53.7,9.8){\small{$1$}}
  \put(-60.6,10){\small{$0$}}
  \put(-63,4){\small{$-2$}}
  \put(-16,44){\small{$\Phi_f(S_f)+b_f$}}
  \put(-28,3.5){\small{$\lfloor\frac{1}{\alpha_f}\rfloor-\kc$}}
   \put(-66,22.5){$\Phi_f^{-1}$}
   \put(-37.5,50){$\chi_f$}
   \put(-68,51){$\Expo$}
   \put(-43,27){$\widetilde{\MD}_f$}
   \put(-52.5,68){\small{$1$}}
  \put(-60.6,68){\small{$0$}}
  \put(-63,61){\small{$-2$}}
  \put(-34.5,67.5){\small{$\kc_1+2$}}
  \put(-112,11){\small{$\cp_f$}}
  \put(-96.5,9.5){\small{$\cv$}}
  \put(-85,4){\small{$\sigma_f$}}
  \put(-70,6){\small{$\MP_f$}}
  \caption{The inverse $\Phi_f^{-1}$ of the perturbed Fatou coordinate can be extended holomorphically to $\widetilde{\MD}_f$ (colored cyan). It can be seen that the image $\Phi_f^{-1}(\widetilde{\MD}_f)$ wraps around $0$. The holomorphic map $\Phi_f^{-1}$ has an anti-holomorphic lift $\chi_f$ such that $\Expo\circ\chi_f=\Phi_f^{-1}$ (note that $\Expo$ is anti-holomorphic). Some special points are also marked.}
  \label{Fig_chi-map}
\end{figure}

Note that $S_f$ is contained in $\{z\in\MP_f:0<\re\Phi_f(z)<\lfloor\tfrac{1}{\alpha}\rfloor-\kc-\tfrac{1}{2}\}$ and
$f^{\circ b_f}(S_f)=\{z\in\MP_f:\lfloor\tfrac{1}{\alpha}\rfloor-\kc-\tfrac{3}{2}\leqslant\re\Phi_f(z)\leqslant\lfloor\tfrac{1}{\alpha}\rfloor-\kc-\tfrac{1}{2} \text{~and~} \im\Phi_f(z)>-2\}$.
According to Proposition \ref{prop-CC-2}(b), if we consider the local rotation of $f$ near the origin, this implies that
\begin{equation}\label{equ-k-f-kc}
b_f=k_f+\lfloor\tfrac{1}{\alpha}\rfloor-\kc -2\geqslant \lfloor\tfrac{1}{\alpha}\rfloor+1, \text{\quad i.e.,\quad} k_f\geqslant \kc+3.
\end{equation}

The modified exponential map $\Expo:\C\to\C\setminus\{0\}$ defined in \eqref{equ-modify-exp} is an anti-holomorphic covering map. The map $\Phi_f^{-1}:\widetilde{\MD}_f\to\C\setminus\{0\}$ can be lifted to obtain an anti-holomorphic map
\begin{equation}
\chi_f:\widetilde{\MD}_f\to\C
\end{equation}
such that
\begin{equation}
\quad \Expo\circ\chi_f(\zeta)=\Phi_f^{-1}(\zeta), \text{ for all }\zeta\in\widetilde{\MD}_f.
\end{equation}
See Figure \ref{Fig_chi-map}.
There are infinitely many choices of $\chi_f:\widetilde{\MD}_f\to\C$. But the following result holds.

\begin{prop}[{\cite[Proposition 1.9]{AC18}}]\label{prop-unif-inverse}
There exists $\kc_1\in\N^+$ such that for all $f\in\IS_\alpha\cup\{Q_\alpha\}$ with $\alpha\in(0,\varepsilon_1']$ and any choice of the lift $\chi_f$, we have
\begin{equation}
\sup\{|\re(\zeta-\zeta')|:\zeta,\zeta'\in\chi_f(\widetilde{\MD}_f)\}\leqslant \kc_1.
\end{equation}
\end{prop}

Proposition \ref{prop-unif-inverse} was proved by applying Proposition \ref{prop-CC-2}, the pre-compactness of the class $\IS_\alpha$ and a uniform bound on the total spiral of the set $\MP_f$ about the origin (see \cite[Proposition 12]{BC12} or \cite[Proposition 2.4]{Che19}).

\medskip
From \cite[\S 5.A]{IS08} or \cite[Propositions 2.6 and 2.7]{CS15} (the top and bottom near-parabolic renormalizations can be defined for all $f\in\IS_\alpha\cup\{Q_\alpha\}$ with $\alpha\in(0,\varepsilon_1']$), $\MP_f$ is contained in the image of $f$. By Lemma \ref{lema:Phi-inverse}, we have $\Phi_f^{-1}(\widetilde{\MD}_f)\subset f(U_f)$. Since $f(U_f)\subset P(U')=\D(0,\tfrac{4}{27}e^{4\pi})$, we have $\im\zeta>-2$ for every $\zeta\in\chi_f(\widetilde{\MD}_f)$, where $P(z)=z(1+z)^2$ and $U'$ is defined in \eqref{equ-U-pri}. Therefore, by Proposition \ref{prop-unif-inverse}, there exists a choice of $\chi_f$, denoted by $\chi_{f,0}$ such that
\begin{equation}\label{equ-chi-choice}
\chi_{f,0}(\widetilde{\MD}_f)\subset \{\zeta\in\C:1\leqslant\re\zeta<\kc_1+2 \text{ and }\im\zeta>-2\}.
\end{equation}
We define
\begin{equation}\label{equ-MD-f}
\MD_f:=\Int\Big(\Phi_f(\MP_f)\cup\bigcup_{j=0}^{k_f+\kc_0+\kc_1+2}(\Phi_f(S_f)+j)\Big),
\end{equation}
where $\kc_0$, $\kc_1\in\N^+$ are integers introduced in Propositions \ref{prop-CC-2} and \ref{prop-unif-inverse} respectively. Let $\kc\in\N^+$ be the integer introduced in Proposition \ref{prop-BC-prop-12}.

\begin{lem}\label{lema:D-n}
For all $f\in\IS_\alpha\cup\{Q_\alpha\}$ with $0<\alpha\leqslant \widetilde{\varepsilon}_1:=\min\{\varepsilon_1',1/(\kc+\kc_0+\kc_1+4)\}$, we have $\MD_f\subset\widetilde{\MD}_f$. Moreover,
\begin{align}
\MD_f &\subset\Phi_f(\MP_f)\cup\{\zeta\in\C:0\leqslant\re\zeta-(\lfloor\tfrac{1}{\alpha}\rfloor-\kc)<2\kc_0+\kc_1+\tfrac{3}{2}\}\text{ and }\label{equ-D-n-loc} \\
\MD_f &\supset\Phi_f(\MP_f)\cup\{\zeta\in\C:0\leqslant\re\zeta-(\lfloor\tfrac{1}{\alpha}\rfloor-\kc)\leqslant\kc_0+\kc_1+3 \text{ and }\im\zeta\geqslant 0\}.
\end{align}
\end{lem}

\begin{proof}
The condition on $\alpha$ implies that $k_f+\kc_0+\kc_1+2\leqslant k_f+\lfloor\tfrac{1}{\alpha}\rfloor-\kc -2$. Then we have $\MD_f\subset\widetilde{\MD}_f$ by definition.

Since $\Phi_f(S_f)\subset\{\zeta\in\C:0<\re\zeta<\lfloor\tfrac{1}{\alpha}\rfloor-\kc-\tfrac{1}{2}\}$ by Proposition \ref{prop-CC-2}(b), for $\zeta\in\MD_f$ we have $\re\zeta<\lfloor\tfrac{1}{\alpha}\rfloor+k_f+\kc_0+\kc_1-\kc+\tfrac{3}{2}\leqslant \lfloor\tfrac{1}{\alpha}\rfloor+2\kc_0+\kc_1-\kc+\tfrac{3}{2}$. Hence \eqref{equ-D-n-loc} holds.

By \eqref{ellipse} and \eqref{U-and-psi-1}, we have $U\supset\D(0,\tfrac{8}{9})$ (see also \cite[Lemma 6.1]{Che19}). For any $f\in\IS_\alpha\cup\{Q_\alpha\}$, by Koebe's $\tfrac{1}{4}$-theorem we have $U_f\supset\D(0,\tfrac{2}{9})$. Since $\Expo(\Phi_f(S_f))\supset U_{\MMR f}\setminus\{0\}$ and $\MMR f\in\IS_{1/\alpha}$, we have $\D(0,\tfrac{2}{9})\subset\Expo(\Phi_f(S_f))$. Since $f^{\circ k_f}(S_f)=\MC_f\cup\MC_f^\sharp\subset\MP_f$, we have $\re\zeta>\lfloor\tfrac{1}{\alpha}\rfloor-\kc$ for all $\zeta\in\Phi_f(S_f)+k_f$. This implies that $\{\zeta\in\C:-\tfrac{3}{2}\leqslant\re\zeta-(\lfloor\tfrac{1}{\alpha}\rfloor-\kc)\leqslant 1 \text{ and }\im\zeta>-\tfrac{1}{2\pi}\log\tfrac{3}{2}\}$ is contained in the interior of $\bigcup_{j=0}^{k_f}(\Phi_f(S_f)+j)$. Therefore, $\MD_f\setminus\Phi_f(\MP_f)$ contains
$\{\zeta\in\C:0\leqslant\re\zeta-(\lfloor\tfrac{1}{\alpha}\rfloor-\kc)\leqslant \kc_0+\kc_1+3 \text{ and }\im\zeta\geqslant 0\}$.
\end{proof}

\subsection{Some quantitative estimates}\label{subsec-esti-2}

Let $\sigma_f\neq0$ be another fixed point of $f\in\IS_\alpha\cup\{Q_\alpha\}$ near $0$ which is contained in $\partial \MP_f$ for small $\alpha>0$ (see Figure \ref{Fig_perturbed-Fatou-coor}). It depends continuously on $f$ and has asymptotic expansion
\begin{equation}\label{equ-sigma-f}
\sigma_f=-4\pi\alpha\ii/f_0''(0)+o(\alpha)
\end{equation}
as $f\to f_0\in\IS_0\cup\{Q_0\}$ in a fixed neighborhood of $0$ (see \cite[\S 3.2.1]{Shi00a}). By \cite[Main Theorem 1(a)]{IS08}, $|f_0''(0)|$ is contained in $[3,7]$ for all $f_0\in\IS_0$. By the pre-compactness of $\IS_0$, there exists a constant $D_0'>1$ such that for all $f\in\IS_\alpha\cup\{Q_\alpha\}$ with $\alpha\in(0,\varepsilon_1]$, one has
\begin{equation}\label{equ-range-sigma-f}
\alpha/D_0'\leqslant |\sigma_f|\leqslant D_0' \alpha.
\end{equation}
For a general statement of \eqref{equ-range-sigma-f} (i.e., $\alpha\in\C$), see \cite[Lemma 3.25(1)]{CS15}.

Let
\begin{equation}\label{equ-tau}
\tau_f(w):=\frac{\sigma_f}{1-e^{-2\pi\ii\alpha w}}
\end{equation}
be a universal covering from $\C$ to $\EC\setminus\{0,\sigma_f\}$ with period $1/\alpha$. Then $\tau_f(w)\to 0$ as $\im w\to +\infty$ and $\tau_f(w)\to \sigma_f$ as $\im w\to -\infty$. There exists a unique lift $F_f$ of $f$ under $\tau_f$ such that
\begin{equation}
f\circ \tau_f(w)=\tau_f\circ F_f(w) \text{\quad with\quad} \lim_{\im w\to+\infty}(F_f(w)-w)=1.
\end{equation}
The set $\tau_f^{-1}(\MP_f)$ consists of countably many simply connected components. Each of them is bounded by piecewise analytic curves going from $-\infty\ii$ to $+\infty\ii$. Let $\widetilde{\MP}_f$ be the unique component separating $0$ from $1/\alpha$. Define
\begin{equation}\label{equ-L-f}
L_f:=\Phi_f\circ \tau_f:\widetilde{\MP}_f\to\C.
\end{equation}
Then $L_f$ is univalent and it is the Fatou coordinate of $F_f$ since $L_f( F_f(w))=L_f(w)+1$ if both $w$ and $F_f(w)$ are contained in $\widetilde{\MP}_f$.

\medskip
For $\alpha\in(0,\widetilde{\varepsilon}_1]$ and $R\in(0,+\infty)$, we define
\begin{equation}\label{equ-Theta-alpha}
\Theta_\alpha(R):=\C\setminus\bigcup_{n\in\Z}\,\D(n/\alpha,R).
\end{equation}
For $C>0$, we denote $a_C:=C e^{5\pi\ii/12}$ and define a piecewise analytic curve
\begin{equation}
\begin{split}
\ell_C:=\{w\in\C:\arg(w-a_C)=\tfrac{11}{12}\pi\}
& \cup \{w\in\C:\arg(w-\overline{a}_C)=-\tfrac{11}{12}\pi\} \\
& \cup \{C e^{\ii\theta}:\theta\in[-\tfrac{5\pi}{12},\tfrac{5\pi}{12}]\}.
\end{split}
\end{equation}
Then $\ell_C\cup(-\ell_C+1/\alpha)$ divides $\C$ into three connected components\footnote{We always assume that $\alpha$ is small such that $\Theta_\alpha(C)$ is connected and hence $1/(2\alpha)\in\Theta_\alpha(C)$.}. Let $A_1(C)$ be the component of $\C\setminus(\ell_C\cup(-\ell_C+1/\alpha))$ containing $1/(2\alpha)$.
The following result is a summary of Lemmas 6.4, 6.7(2), 6.6 and 6.11 in \cite{Che19}.

\begin{lem}\label{lema-key-pre-Che}
There are constants $\varepsilon_2\in(0,\widetilde{\varepsilon}_1]$, $C_0$, $C_0'>0$ and $C_0''\geqslant 6$ such that for all $f\in\IS_\alpha\cup\{Q_\alpha\}$ with $\alpha\in(0,\varepsilon_2]$, we have
\begin{enumerate}
\item $F_f$ is defined and univalent in $\Theta_\alpha(C_0')$, and for all $r\in(0,1/2]$ and all $w\in\Theta_\alpha(r/\alpha)\cap\Theta_\alpha(C_0')$, then
\begin{equation}
|F_f(w)-(w+1)|,\,|F_f'(w)-1|<\min\Big\{\frac{1}{4},C_0\frac{\alpha}{r}e^{-2\pi\alpha\im w}\Big\};
\end{equation}
\item For all\,\footnote{In \cite[Lemma 6.7(2)]{Che19}, $R$ is contained in $[3.25,1/(2\alpha)]$. In fact the estimate of $|L_f'(w)|$ there still holds if $R\in [3.25,C/\alpha]$ for every $C\geqslant 1/2$ (the only difference is that the constants in the estimate need to be modified).} $R\in[C_0'',2/\alpha]$ and all $w$ with $\D(w,R)\subset A_1:=A_1(C_0')$ and $\im w\geqslant -1/\alpha$, then
\begin{equation}
\frac{1}{|L_f'(w)|} \leqslant 1+\frac{C_0}{R};
\end{equation}
\item $L_f:\widetilde{\MP}_f\to\C$ has a unique univalent extension onto $\widetilde{\MP}_f \cup A_1$ such that $L_f( F_f(w))=L_f(w)+1$ if both $w$ and $F_f(w)$ belong to $\widetilde{\MP}_f \cup A_1$;
\item For any $r>0$ there is $K_r\geqslant 1$ depending only on $r$ such that\,\footnote{By Lemma \ref{lema-key-pre-Che}(c), the number $x_f$ defined in \cite[Equation (50)]{Che19} satisfies $x_f\geqslant \lfloor\tfrac{1}{\alpha}\rfloor-\kc$. Hence by \cite[Lemma 6.11]{Che19} this part holds for all $\zeta\in\Phi_f(\MP_f)\setminus\D(0,r)$.}
\begin{equation}
K_r^{-1}\leqslant |(L_f^{-1})'(\zeta)|\leqslant K_r \text{ for all } \zeta\in\Phi_f(\MP_f)\setminus\D(0,r).
\end{equation}
\end{enumerate}
\end{lem}

The following Lemma \ref{lema-Cheraghi-L-f-add} and Proposition \ref{prop-Cheraghi-L-f} are useful in the estimates of the locations of the points under $\Phi_f^{-1}$ and $\chi_f$.

\begin{lem}\label{lema-Cheraghi-L-f-add}
There exists a constant $D_0>0$ such that for any $D_1'>0$, there exists $D_1>0$ such that for all $f\in\IS_\alpha\cup\{Q_\alpha\}$ with $\alpha\in(0,\varepsilon_2]$, we have
\begin{enumerate}
\item $D_0\leqslant |L_f^{-1}(\zeta)|\leqslant D_1$ for $\zeta\in\Phi_f(\MP_f)\cap\overline{\D}(0,D_1')$; and
\item $D_0\leqslant |L_f^{-1}(\zeta)-1/\alpha|\leqslant D_1$ for $\zeta\in\Phi_f(\MP_f)\cap\overline{\D}(1/\alpha,D_1')$.
\end{enumerate}
\end{lem}

\begin{proof}
By the continuous dependence of the Fatou coordinates of the maps in $\IS_0$, the pre-compactness of $\IS_0$ and note that $\MP_f$ is compactly contained in the domain of definition of $f$, there exists a constant $R_1>0$ such that \begin{equation}
\MP_f\subset\D(0,R_1) \text{ for all } f\in\IS_\alpha\cup\{Q_\alpha\} \text{ with } \alpha\in(0,\varepsilon_2].
\end{equation}
By \eqref{equ-range-sigma-f} and the formula of $\tau_f$ in \eqref{equ-tau}, a direct calculation shows that there exists a constant $D_0>0$ such that the Euclidean distance satisfies $\dist(L_f^{-1}(\zeta),\Z/\alpha)\geqslant D_0$ for all $f\in\IS_\alpha\cup\{Q_\alpha\}$ with $\alpha\in(0,\varepsilon_2]$ and all $\zeta\in\Phi_f(\MP_f)$.

\medskip
By Lemma \ref{lema-key-pre-Che}(d), there exists a constant $K_1>1$ such that
\begin{equation}\label{equ-bound-L-f}
K_1^{-1}\leqslant |(L_f^{-1})'(\zeta)|\leqslant K_1
\end{equation}
for all $f\in\IS_\alpha\cup\{Q_\alpha\}$ with $\alpha\in(0,\varepsilon_2]$ and all $\zeta\in\Phi_f(\MP_f)\setminus\D$. From \cite[Proposition 6.17]{Che19}, there exists a constant $C_1>0$ such that for all $f\in\IS_\alpha\cup\{Q_\alpha\}$ with $\alpha\in(0,\varepsilon_2]$ we have
\begin{equation}\label{equ-bound-pt}
|L_f^{-1}(\tfrac{3}{2})|<C_1.
\end{equation}

Without loss of generality we assume that $D_1'>1$. Combining \eqref{equ-bound-L-f} and \eqref{equ-bound-pt}, there exists a constant $C_2>0$ depending only on $K_1$, $C_1$ and $D_1'$ such that $|L_f^{-1}(\zeta)|<C_2$ for all $f\in\IS_\alpha\cup\{Q_\alpha\}$ with $\alpha\in(0,\varepsilon_2]$ and all $\zeta\in(\Phi_f(\MP_f)\cap\overline{\D}(0,D_1'))\setminus\D$. On the other hand, by Lemma \ref{lema-key-pre-Che}(a) and applying
\begin{equation}
L_f^{-1}(\zeta)=F_f^{-1}\circ L_f^{-1}(\zeta+1),
\end{equation}
there exists a constant $C_3>0$ such that $|L_f^{-1}(\zeta)|<C_3$ for all $f\in\IS_\alpha\cup\{Q_\alpha\}$ with $\alpha\in(0,\varepsilon_2]$ and all $\zeta\in\Phi_f(\MP_f)\cap\D$.

By Lemma \ref{lema-key-pre-Che}(d) and \cite[Proposition 6.16]{Che19}, there exists a constant $C_4>0$ depending on $D_1'$ such that $|L_f^{-1}(\zeta)-1/\alpha|\leqslant C_4$ for all $f\in\IS_\alpha\cup\{Q_\alpha\}$ with $\alpha\in(0,\varepsilon_2]$ and all $\zeta\in\Phi_f(\MP_f)\cap\overline{\D}(1/\alpha,D_1')$. Then the proof is complete if we set $D_1:=\max\{C_2,C_3,C_4\}$.
\end{proof}

\begin{prop}[{\cite[Propositions 6.19 and 6.17]{Che19}}]\label{prop-Cheraghi-L-f}
There are constants $\varepsilon_2'\in(0,\varepsilon_2]$ and $D_2>0$ such that for all $f\in\IS_\alpha\cup\{Q_\alpha\}$ with $\alpha\in(0,\varepsilon_2']$, we have
\begin{enumerate}
\item If $\zeta\in[0,\lfloor\tfrac{1}{\alpha}\rfloor-\kc\,]+\ii\,[-3,+\infty)$, then
\begin{equation}
|L_f^{-1}(\zeta)-\zeta|\leqslant D_2 \log(1+1/\alpha).
\end{equation}
\item If $\zeta\in [0,\lfloor\tfrac{1}{\alpha}\rfloor-\kc\,]+\ii\,[-3,1/\alpha]$, then
\begin{equation}
|L_f^{-1}(\zeta)-\zeta|\leqslant D_2 \min\{\log(2+|\zeta|),\,\log(2+|\zeta-1/\alpha|)\}.
\end{equation}
\end{enumerate}
\end{prop}

Proposition \ref{prop-Cheraghi-L-f}(a) was proved in \cite[Proposition 6.19]{Che19} (see also \cite[Proposition 6.15]{Che19}). The statement (b) was proved in \cite[Proposition 6.17]{Che19} for $\zeta\in[0,\lfloor\tfrac{1}{\alpha}\rfloor-\kc\,]$ (i.e., $\zeta\in\R$). However, the arguments there can be applied to  $\zeta\in[0,\lfloor\tfrac{1}{\alpha}\rfloor-\kc\,]+ \ii\,[-3,1/\alpha]$ completely similarly by using \cite[Lemma 6.7]{Che19} and Lemma \ref{lema-Cheraghi-L-f-add}.
For more details on the study of $L_f$ and $L_f^{-1}$, see \cite[\S\S 6.3-6.6]{Che19} and \cite[\S 3.5]{CS15}.

\medskip
Let $X\geqslant 0$ and $Y\geqslant 0$ be two numbers. We use $X\asymp Y$ to denote that $X$ and $Y$ are in the same order, i.e., there exist two universal positive constants $C_1$ and $C_2$ such that $C_1Y\leqslant X\leqslant C_2 Y$. Let $\MD_f$ be the set defined in \eqref{equ-MD-f}.

\begin{lem}\label{lema-key-estimate-lp}
There exist constants $\varepsilon_3\in(0,\varepsilon_2']$ and $D_3>0$ such that for all $f\in\IS_\alpha\cup\{Q_\alpha\}$ with $\alpha\in(0,\varepsilon_3]$, we have
\begin{enumerate}
\item If $\zeta\in\MD_f$ with $\im\zeta\geqslant 1/\alpha$, then
\begin{equation}
|\Phi_f^{-1}(\zeta)| \asymp \frac{\alpha}{e^{2\pi\alpha\im\zeta}} \text{\quad and\quad}
\Big|\im \chi_f(\zeta)-\Big(\alpha\,\im\zeta+\frac{1}{2\pi}\log\frac{1}{\alpha}\Big)\Big|\leqslant D_3.
\end{equation}
\item If $\zeta\in\MD_f$ with $\im \zeta\in[-3,1/\alpha]$, then
\begin{equation}
|\Phi_f^{-1}(\zeta)|\asymp~ \max\Big\{\frac{1}{1+|\zeta|},\frac{1}{1+|\zeta-1/\alpha|}\Big\} \text{\quad and}
\end{equation}
\begin{equation}
\big|\im \chi_f(\zeta)-\tfrac{1}{2\pi}\min\big\{\log(1+|\zeta|),\log(1+|\zeta-1/\alpha|)\big\}\big|\leqslant D_3.
\end{equation}
\end{enumerate}
\end{lem}

\begin{proof}
By the definition of $\Phi_f^{-1}$ in Lemma \ref{lema:Phi-inverse}, if $\zeta\in\MD_f\setminus\Phi_f(\MP_f)$, then there exists a positive integer $j\in[1,k_f+\kc_0+\kc_1+2]$ such that $\zeta-j\in\Phi_f(\MP_f)$ and $\Phi_f^{-1}(\zeta)=f^{\circ j}(\Phi_f^{-1}(\zeta-j))$. By the pre-compactness of $\IS_\alpha$, it is sufficient to prove the statements in this lemma for $\zeta\in\Phi_f(\MP_f)$.

\medskip
(a) By Proposition \ref{prop-Cheraghi-L-f}(a), we have
\begin{equation}
\im\zeta-D_2\log(1+1/\alpha)\leqslant \im L_f^{-1}(\zeta) \leqslant \im\zeta+D_2\log(1+1/\alpha).
\end{equation}
If $\alpha$ is small, then $\alpha\log(1+1/\alpha)$ is also. Suppose $\zeta\in\MD_f$ with $\im\zeta\geqslant 1/\alpha$. Decreasing $\alpha$ if necessary, we assume that $\im \zeta-D_2\log(1+1/\alpha)>1/(2\alpha)$. Denote $w:=L_f^{-1}(\zeta)$. Then $|e^{-2\pi\ii\alpha w}|=|e^{2\pi\alpha \im w}\cdot e^{-2\pi\ii\alpha \re w}|>e^\pi$. Note that $\alpha\log(1+1/\alpha)$ is uniformly bounded above. Since $\im\zeta\geqslant 1/\alpha$, we have
\begin{equation}
|1-e^{-2\pi\ii\alpha w}|\asymp e^{2\pi\alpha\im w}\asymp e^{2\pi\alpha\im\zeta}.
\end{equation}
By \eqref{equ-range-sigma-f}, \eqref{equ-tau} and \eqref{equ-L-f}, we have
\begin{equation}
|\Phi_f^{-1}(\zeta)|=|\tau_f\circ L_f^{-1}(\zeta)|=\left|\frac{\sigma_f}{1-e^{-2\pi\ii\alpha w}}\right| \asymp \frac{\alpha}{e^{2\pi\alpha\im\zeta}}.
\end{equation}
Denote $y:=\im \Expo^{-1}\circ\Phi_f^{-1}(\zeta)$. By definition we have $\tfrac{4}{27}e^{-2\pi y}\asymp \alpha/e^{2\pi\alpha\im\zeta}$. A direct calculation shows that $y=\alpha\,\im\zeta+\tfrac{1}{2\pi}\log\frac{1}{\alpha}+\MO(1)$, where $\MO(1)$ is a number whose absolute value is less than a universal constant.

\medskip
(b) We divide the arguments into two cases. Firstly we assume that $\zeta\in\MD_f$ with $\re\zeta\in[0,1/(2\alpha)]$. By Proposition \ref{prop-Cheraghi-L-f}(b), we have
\begin{equation}\label{equ-Lf-D-2}
|L_f^{-1}(\zeta)-\zeta|\leqslant D_2\log(2+|\zeta|).
\end{equation}
Let $D_1'>0$ be the smallest constant depending only on $D_2$ such that if $|\zeta|\geqslant D_1'$, then $|\zeta|\geqslant D_2\log(2+|\zeta|)+1$. If $|\zeta|\geqslant D_1'$, $\re\zeta\in[0,1/(2\alpha)]$ and $\im\zeta\in[-3,1/\alpha]$, by \eqref{equ-Lf-D-2} we have
\begin{equation}\label{equ-Lf-asym}
|L_f^{-1}(\zeta)|\asymp |\zeta|+1.
\end{equation}
If $|\zeta|\leqslant D_1'$, $\re\zeta\in[0,1/(2\alpha)]$ and $\im\zeta\in[-3,1/\alpha]$, by Lemma \ref{lema-Cheraghi-L-f-add}(a), there exists a constant $D_1>1$ depending only on $D_1'$ such that $D_0\leqslant |L_f^{-1}(\zeta)|\leqslant D_1$. Therefore, we still have \eqref{equ-Lf-asym}.

Next we assume that $\re\zeta\in[1/(2\alpha),\lfloor\tfrac{1}{\alpha}\rfloor-\kc\,]$. By Proposition \ref{prop-Cheraghi-L-f}(b), we have
\begin{equation}\label{equ-Lf-D-2-2}
|L_f^{-1}(\zeta)-\zeta|\leqslant D_2\log(2+|\zeta-1/\alpha|).
\end{equation}
If $|\zeta-1/\alpha|\geqslant D_1'$, then $|\zeta-1/\alpha|\geqslant D_2\log(2+|\zeta-1/\alpha|)+1$. If $|\zeta-1/\alpha|\geqslant D_1'$, $\re\zeta\in[1/(2\alpha),\,\lfloor\tfrac{1}{\alpha}\rfloor-\kc\,]$ and $\im\zeta\in[-3,1/\alpha]$, by \eqref{equ-Lf-D-2-2} we have
\begin{equation}\label{equ-Lf-asym-2}
|L_f^{-1}(\zeta)-1/\alpha|=|(L_f^{-1}(\zeta)-\zeta)+(\zeta-1/\alpha)|\asymp |\zeta-1/\alpha|+1.
\end{equation}
If $|\zeta-1/\alpha|\leqslant D_1'$, $\re\zeta\in[1/(2\alpha),\lfloor\tfrac{1}{\alpha}\rfloor-\kc\,]$ and $\im\zeta\in[-3,1/\alpha]$, by Lemma \ref{lema-Cheraghi-L-f-add}(b), we have $D_0\leqslant |L_f^{-1}(\zeta)-1/\alpha|\leqslant D_1$. Therefore, in this case we still have \eqref{equ-Lf-asym-2}.

Denote $w:=L_f^{-1}(\zeta)$. By \eqref{equ-Lf-D-2} and \eqref{equ-Lf-D-2-2}, if $\alpha$ is small enough, then $-\tfrac{1}{4}\leqslant\re(\alpha w)\leqslant \tfrac{5}{4}$ and $|\alpha w|\leqslant\tfrac{3}{2}$. By \eqref{equ-range-sigma-f}, \eqref{equ-L-f}, \eqref{equ-Lf-asym} and \eqref{equ-Lf-asym-2}, we have
\begin{equation}
\begin{split}
|\Phi_f^{-1}(\zeta)|=\left|\frac{\sigma_f}{1-e^{-2\pi\ii\alpha w}}\right| \asymp &~\max\Big\{\frac{1}{|w|},\frac{1}{|w-1/\alpha|}\Big\}\\
\asymp&~ \max\Big\{\frac{1}{1+|\zeta|},\frac{1}{1+|\zeta-1/\alpha|}\Big\}.
\end{split}
\end{equation}
Then the estimate of $\im\, \Expo^{-1}\circ\Phi_f^{-1}(\zeta)$ follows by a direct calculation.
\end{proof}

\begin{rmk}
(1) There exist some overlaps between the estimates in Lemma \ref{lema-key-estimate-lp}(a) and (b). Indeed, if $\zeta\in\MD_f$ and $\im\zeta\asymp 1/\alpha$, then
\begin{equation}
|\Phi_f^{-1}(\zeta)|\asymp\alpha \quad\text{and}\quad
\im \Expo^{-1}\circ\Phi_f^{-1}(\zeta)=\tfrac{1}{2\pi}\log\tfrac{1}{\alpha}+\MO(1).
\end{equation}

(2) Lemma \ref{lema-key-estimate-lp} illustrates how the renormalization microscopes $\chi_f$ reshapes the geometry of the Siegel disk at deeper scales.
Specifically, Part (a) is for the points deep in the Siegel disk while Part (b) is for the points close to the Siegel boundary.
\end{rmk}

The following lemma can be seen as an inverse version of Lemma \ref{lema-key-estimate-lp}.

\begin{lem}\label{lema-key-esti-inverse}
There exist constants $D_4$, $D_5>1$ and $\varepsilon_3'\in(0,\varepsilon_3]$ such that for all $f\in\IS_\alpha\cup\{Q_\alpha\}$ with $\alpha\in(0,\varepsilon_3']$, we have
\begin{enumerate}
\item If $\zeta'\in\C$ satisfies $\im\zeta'\geqslant\frac{1}{2\pi}\log\tfrac{1}{\alpha}+D_4$, $\Expo(\zeta')\in\MP_f$  and $\Phi_f\circ \Expo(\zeta')\in (0,2]+\ii\,[-2,+\infty)$, then
\begin{equation}
\left|\im \Phi_f\circ\Expo(\zeta')-\frac{1}{\alpha}\Big(\im\zeta'-\frac{1}{2\pi}\log\frac{1}{\alpha}\Big)\right|\leqslant\frac{D_5}{\alpha}.
\end{equation}
\item If $\zeta'\in\C$ satisfies $\im\zeta'< \frac{1}{2\pi}\log\tfrac{1}{\alpha}+D_4$, $\Expo(\zeta')\in\MP_f$ and $\Phi_f\circ \Expo(\zeta')\in (0,2]+\ii\,[-2,+\infty)$, then
\begin{equation}
\big|\log\big(3+\im \Phi_f\circ\Expo(\zeta')\big)-2\pi\im\zeta'\big|\leqslant D_5.
\end{equation}
\end{enumerate}
\end{lem}

\begin{proof}
(a) Denote $\zeta=\Phi_f\circ \Expo(\zeta')\in\Phi_f(\MP_f)$. By Lemma \ref{lema-key-estimate-lp}(a), if $\im\zeta\geqslant 1/\alpha$ we have
\begin{equation}\label{equ-down-1}
\Big|\im \zeta-\frac{1}{\alpha}\Big(\im\zeta'-\frac{1}{2\pi}\log\frac{1}{\alpha}\Big)\Big|\leqslant\frac{D_3}{\alpha}.
\end{equation}
Suppose $\re\zeta\in(0,2]$ and $\im \zeta\in[-2,1/\alpha)$. By Lemma \ref{lema-key-estimate-lp}(b), we have
\begin{equation}
\im\zeta'\leqslant \frac{1}{2\pi}\log(1+|\zeta|)+D_3<\frac{1}{2\pi}\log\Big(\frac{1}{\alpha}+3\Big)+D_3<\frac{1}{2\pi}\log\frac{1}{\alpha}+D_3+1.
\end{equation}
Therefore, if $\im\zeta'\geqslant \tfrac{1}{2\pi}\log\tfrac{1}{\alpha}+D_3+1$, then $\im\zeta\geqslant 1/\alpha$ or $\im\zeta<-2$. By the assumption in the lemma we have $\im\zeta\geqslant 1/\alpha$ and \eqref{equ-down-1} holds. Then Part (a) follows if we set $D_4:=D_3+1$ and $D_5:=D_3$.

\medskip
(b) Denote $\zeta=\Phi_f\circ \Expo(\zeta')\in (0,2]+\ii\,[-2,+\infty)$. By \eqref{equ-down-1}, if $\im\zeta\in[1/\alpha,(1+2D_3)/\alpha]$, we have $|\log\tfrac{1}{\alpha}+2\pi\alpha\im\zeta-2\pi\im\zeta'|\leqslant 2\pi D_3$ and hence
\begin{equation}
\begin{split}
|\log(3+\im\zeta)-2\pi\im\zeta'|
\leqslant &~ |\log(3\alpha+\alpha\im\zeta)-2\pi\alpha\im\zeta|+2\pi D_3\\
\leqslant &~ \log(4+2 D_3)+6\pi D_3+2\pi.
\end{split}
\end{equation}
By Lemma \ref{lema-key-estimate-lp}(b), if $\re\zeta\in(0,2]$ and $\im\zeta\in(-2,1/\alpha)$ we have $|\log(1+|\zeta|)-2\pi\im \zeta'|\leqslant 2\pi D_3$ and hence
\begin{equation}
\begin{split}
|\log(3+\im\zeta)-2\pi\im\zeta'|
\leqslant &~ |\log(3+\im\zeta)-\log(1+|\zeta|)|+2\pi D_3\\
\leqslant &~ \log 5+2\pi D_3.
\end{split}
\end{equation}
Set $D_5=\log(4+2 D_3)+6\pi D_3+2\pi$. Then if $\im\zeta<(1+2D_3)/\alpha$ we have
\begin{equation}\label{equ-down-2}
|\log(3+\im\zeta)-2\pi\im\zeta'|\leqslant D_5.
\end{equation}
Suppose $\im \zeta\geqslant (1+2D_3)/\alpha$. By Lemma \ref{lema-key-estimate-lp}(a), we have
\begin{equation}
\im\zeta'\geqslant \alpha\,\im\zeta+\frac{1}{2\pi}\log\frac{1}{\alpha}-D_3\geqslant \frac{1}{2\pi}\log\frac{1}{\alpha}+D_3+1.
\end{equation}
Therefore, if $\im\zeta'< \tfrac{1}{2\pi}\log\tfrac{1}{\alpha}+D_3+1$, then $\im\zeta< (1+2D_3)/\alpha$ and we have \eqref{equ-down-2}.

Summering the constants in Parts (a) and (b), the lemma follows if we set $D_4:=D_3+1$ and $D_5:=\log(4+2 D_3)+6\pi D_3+2\pi$.
\end{proof}

In the following, we use $h'$ to denote $\partial h/\partial z$ if $h$ is holomorphic and denote $\partial \overline{h}/\partial z$ if $h$ is anti-holomorphic.
The following result is useful in the estimate of the Euclidean length of curves in Fatou coordinate planes.

\begin{prop}\label{prop-key-estimate-yyy}
There exist positive constants $\varepsilon_4\in(0,\varepsilon_3']$ and $D_2'$, $D_6'$, $D_6>1$ such that for all $f\in\IS_\alpha\cup\{Q_\alpha\}$ with $\alpha\in(0,\varepsilon_4]$, we have
\begin{enumerate}
\item If $\zeta\in\MD_f$ with $\im\zeta\geqslant 1/(4\alpha)$, then
\begin{equation}
|\chi_f'(\zeta)-\alpha| \leqslant D_6\alpha e^{-2\pi\alpha\im \zeta}.
\end{equation}

\item If $\zeta\in\MD_f$ with $\im \zeta\in[-2,1/(4\alpha)]$ and $r=\min\{|\zeta|,\,|\zeta-1/\alpha|\}\geqslant D_6'$, then
\begin{equation}
|\chi_f'(\zeta)|\leqslant \frac{\alpha}{1-e^{-2\pi\alpha(r-D_2'\log(2+r))}}\left(1+\frac{D_6}{r}\right),
\end{equation}
where $D_2'$ and $D_6'$ are chosen such that $r-2D_2'\log(2+r)\geqslant 4$ if $r\geqslant D_6'$.
\end{enumerate}
\end{prop}

\begin{proof}
Part (a) is proved in \cite[Proposition 3.3]{Che13}. We only prove Part (b).
For the continuous function
\begin{equation}
\varphi(z):=|1-e^{2\pi\ii z}|,
\end{equation}
where $z\in\Xi_\varrho:=\{\varrho e^{\ii\theta}:\theta\in[-\tfrac{\pi}{4},\tfrac{5\pi}{4}]\}$ with $0<\varrho\leqslant\tfrac{2}{3}$, by a direct calculation\footnote{By setting $r:=2\pi\varrho$, $\beta:=\theta-\frac{\pi}{2}$ and considering the derivative of $\beta\mapsto\big(\varphi(\frac{r}{2\pi}e^{\ii(\beta+\frac{\pi}{2})})\big)^2$, it suffices to verify that $e^{-r \cos\beta}\sin\beta-\sin(\beta-r\sin\beta)>0$ for any $r\in(0,\frac{4\pi}{3}]$ and $\beta\in(0,\frac{3\pi}{4}]$. This can be done by considering three cases: (1) $\beta-r\sin\beta\in[-\pi,0]$; (2) $\beta-r\sin\beta\in(0,\frac{\pi}{2}]$ and $\beta\in(0,\frac{\pi}{2}]$; and (3) $\beta-r\sin\beta\in(0,\frac{3\pi}{4}]$ and $\beta\in(\frac{\pi}{2},\frac{3\pi}{4}]$. } we have
\begin{equation}\label{equ-varphi-min}
\min_{z\in\Xi_\varrho}\varphi(z)=\varphi(\varrho e^{\ii\frac{\pi}{2}})=\varphi(\varrho\ii)=1-e^{-2\pi \varrho}.
\end{equation}

\textbf{Case 1.}
We first consider $\zeta\in\Lambda_1:=\MD_f\cap\{\zeta\in\C:\re\zeta\in(0,1/(2\alpha)]\text{ and } \im\zeta\in[-2,1/(4\alpha)]\}$ and denote $w:=L_f^{-1}(\zeta)\in\widetilde{\MP}_f$. By \eqref{equ-modify-exp}, \eqref{equ-tau}, \eqref{equ-L-f} and a straightforward calculation we have
\begin{equation}\label{equ-chi-deri-1}
\begin{split}
\chi_f'(\zeta)=&~(\Expo^{-1}\circ\Phi_f^{-1})'(\zeta)=(\Expo^{-1}\circ\tau_f\circ L_f^{-1})'(\zeta)\\
=&~-\frac{\alpha}{1-e^{2\pi\ii\alpha w}}\cdot\frac{1}{L_f'(w)}.
\end{split}
\end{equation}

By Proposition \ref{prop-Cheraghi-L-f}(b), we have
\begin{equation}\label{equ:w-zeta}
w\in\overline{\D}(\zeta,D_2\log(2+|\zeta|)).
\end{equation}
Let $C_0''\geqslant 6$ be the constant and $A_1=A_1(C_0')$ be the domain introduced in Lemma \ref{lema-key-pre-Che}(b). Let $C_1\geqslant 1$ be a constant depending only on $C_0''$ and $D_2$ such that if $|\zeta|\geqslant C_1$, then
\begin{equation}\label{equ-D-w-C}
|\zeta|-2D_2\log(2+|\zeta|)\geqslant 4 \text{\quad and\quad} \overline{\D}(w,C_0'')\subset A_1.
\end{equation}
We assume that $\widehat{\varepsilon}_1>0$ is small such that if $\alpha\in(0,\widehat{\varepsilon}_1]$, then $\alpha|\zeta|<\tfrac{3}{5}$ and $D_2\alpha\log(2+|\zeta|)< \tfrac{1}{15}$ for all $\zeta\in \Lambda_1$.
Hence
\begin{equation}\label{equ:alpha-zeta}
\alpha|\zeta|+D_2\alpha\log(2+|\zeta|)<\tfrac{2}{3} \text{\quad for all } \zeta\in\Lambda_1.
\end{equation}
By \eqref{equ:w-zeta}, \eqref{equ-D-w-C} and \eqref{equ:alpha-zeta}, for $\zeta\in\Lambda_1':=\Lambda_1\cap\{\zeta\in\C:|\zeta|\geqslant C_1\}$ we have
$\alpha w\in\{\varrho e^{\ii\theta}:0<\varrho\leqslant\frac{2}{3} \text{ and }-\tfrac{\pi}{4}<\theta<\tfrac{3\pi}{4}\}$. According to \eqref{equ-varphi-min}, we have
\begin{equation}\label{equ-1-exp-1}
|1-e^{2\pi\ii\alpha w}|\geqslant 1-e^{-2\pi\alpha(|\zeta|-D_2\log(2+|\zeta|))}.
\end{equation}
On the other hand, by \eqref{equ-D-w-C}, Lemma \ref{lema-key-pre-Che}(b)(d) and Proposition \ref{prop-Cheraghi-L-f}(b), there exists a constant $C_2\geqslant 1$ depending only on $C_1$ and $D_2$ such that if $\zeta\in\Lambda_1'$ then
\begin{equation}\label{equ-L-f-deri-1}
\frac{1}{|L_f'(w)|}\leqslant 1+\frac{C_2}{|\zeta|}.
\end{equation}
Combining \eqref{equ-chi-deri-1}, \eqref{equ-1-exp-1} and \eqref{equ-L-f-deri-1}, if $\zeta\in\Lambda_1'$ we have
\begin{equation}
|\chi_f'(\zeta)|\leqslant\frac{\alpha}{1-e^{-2\pi\alpha(|\zeta|-D_2\log(2+|\zeta|))}}\left(1+\frac{C_2}{|\zeta|}\right).
\end{equation}

\textbf{Case 2.}
Suppose $\zeta\in\Lambda_2:=\MD_f\cap\{\zeta\in\C:\re\zeta>1/(2\alpha) \text{ and } \im\zeta\in[-2,1/(4\alpha)]\}$. By the definition of $\MD_f$ in \eqref{equ-MD-f}, there exist an integer $J\geqslant 1$ which is independent of $f$ and an integer $j_0\in\N$ with $j_0\leqslant J$ such that $\zeta-j_0\in\Phi_f(\MP_f)\cap\{\zeta:\re\zeta>1/(2\alpha)\}$. We denote $w:=L_f^{-1}(\zeta-j_0)\in\widetilde{\MP}_f$ and $\widetilde{w}:=F_f^{\circ j_0}(w)$. Then
\begin{equation}\label{equ-chi-deri-2}
\begin{split}
\chi_f'(\zeta)=&~(\Expo^{-1}\circ f^{\circ j_0}\circ \Phi_f^{-1})'(\zeta-j_0)\\
=&~(\Expo^{-1}\circ\tau_f\circ F_f^{\circ j_0}\circ L_f^{-1})'(\zeta-j_0)=-\frac{\alpha}{1-e^{2\pi\ii\alpha \widetilde{w}}}\cdot\frac{(F_f^{\circ j_0})'(w)}{L_f'(w)}.
\end{split}
\end{equation}

By Proposition \ref{prop-Cheraghi-L-f}(b), we have
\begin{equation}
w\in\overline{\D}\big(\zeta-j_0,D_2\log(2+|\zeta-j_0-\tfrac{1}{\alpha}|)\big).
\end{equation}
Let $C_0''\geqslant 6$ and $A_1=A_1(C_0')$ be introduced as in Lemma \ref{lema-key-pre-Che}(b). By Lemma \ref{lema-key-pre-Che}(a), there exist two positive constants $C_1'$ and $C_1''$ depending only on $C_0''$, $D_2$ and $J$ such that if $|\zeta-1/\alpha|\geqslant C_1'$, then
\begin{equation}\label{equ-D-w-C-2}
\overline{\D}(w,C_0'')\subset A_1 \text{ and } |F_f^{\circ j}(w)-1/\alpha|\geqslant C_1''|\zeta-1/\alpha|
\end{equation}
for all $j=0,1,\cdots,j_0$.
Also by Lemma \ref{lema-key-pre-Che}(a), there exists a constant $D_2'\geqslant D_2$ depending only on $C_0''$, $C_1''$, $D_2$ and $J$ such that
\begin{equation}
\widetilde{w}=F_f^{\circ j_0}(w)\in\D\big(\zeta,D_2'\log(2+|\zeta-1/\alpha|)\big)
\end{equation}
and
\begin{equation}\label{equ-F-deri}
|(F_f^{\circ j_0})'(w)|\leqslant 1+\frac{D_2'}{|\zeta-1/\alpha|}.
\end{equation}

Let $C_2'\geqslant C_1'$ be a constant depending only on $C_1'$ and $D_2'$ such that if $|\zeta-1/\alpha|\geqslant C_2'$, then
\begin{equation}
|\zeta-1/\alpha|-2D_2'\log(2+|\zeta-1/\alpha|)\geqslant 4.
\end{equation}
Moreover, we assume that $\widehat{\varepsilon}_2>0$ is small such that if $\alpha\in(0,\widehat{\varepsilon}_2]$, then
\begin{equation}
\alpha|\zeta-1/\alpha|+D_2'\alpha\log(2+|\zeta-1/\alpha|)<\tfrac{2}{3} \text{\quad for all } \zeta\in\Lambda_2.
\end{equation}
For $\zeta\in\Lambda_2':=\Lambda_2\cap\{\zeta\in\C:|\zeta-1/\alpha|\geqslant C_2'\}$, we have
$\alpha \widetilde{w}-1\in\{\varrho e^{\ii\theta}:0<\varrho\leqslant\frac{2}{3} \text{ and }\tfrac{\pi}{4}<\theta<\tfrac{5\pi}{4}\}$.
By \eqref{equ-varphi-min} and $|1-e^{2\pi\ii z}|=|1-e^{2\pi\ii(z-1)}|$, we have
\begin{equation}\label{equ-1-exp-2}
|1-e^{2\pi\ii\alpha \widetilde{w}}|\geqslant 1-e^{-2\pi\alpha(|\zeta-1/\alpha|-D_2'\log(2+|\zeta-1/\alpha|))}.
\end{equation}
Similarly, by \eqref{equ-D-w-C-2}, Lemma \ref{lema-key-pre-Che}(b)(d) and Proposition \ref{prop-Cheraghi-L-f}(b), there exists a constant $C_3\geqslant 1$ depending only on $C_1''$, $C_2'$ and $D_2'$ such that if $\zeta\in\Lambda_2'$ then
\begin{equation}\label{equ-L-f-deri-2}
\frac{1}{|L_f'(w)|}\leqslant 1+\frac{C_3}{|\zeta-1/\alpha|}.
\end{equation}
Combining \eqref{equ-chi-deri-2}, \eqref{equ-F-deri}, \eqref{equ-1-exp-2} and \eqref{equ-L-f-deri-2}, if $\zeta\in\Lambda_2'$ we have
\begin{equation}
|\chi_f'(\zeta)|\leqslant\frac{\alpha}{1-e^{-2\pi\alpha(|\zeta-1/\alpha|-D_2'\log(2+|\zeta-1/\alpha|))}}\left(1+\frac{C_3'}{|\zeta-1/\alpha|}\right)
\end{equation}
for a constant $C_3'>0$ depending only on $C_3$ and $D_2'$.
The proof is complete if we set $\varepsilon_4:=\min\{\varepsilon_3',\widehat{\varepsilon}_1,\widehat{\varepsilon}_2\}$, $D_6':=\max\{C_1,C_2'\}$ and $D_6:=\max\{C_2,C_3'\}$.
\end{proof}

\begin{rmk}
Proposition \ref{prop-key-estimate-yyy} will be used in the proof of Lemma \ref{lema-go-up}. In \cite[Proposition 6.18]{Che19}, an estimate of $|\chi_f'(\zeta)|$ has been obtained for $\zeta\in[1,1/(2\alpha)]$ in another form.
\end{rmk}

\subsection{Renormalization tower and orbit relations}\label{subsec-basic-defi}

In the rest of this paper, we always assume that the integer $N$ is large so that $N\geqslant 1/\varepsilon_4$, where $\varepsilon_4>0$ is the constant introduced in Proposition \ref{prop-key-estimate-yyy}. Let $[0;a_1,a_2,\cdots]$ be the continued fraction expansion of $\alpha\in\HT_N$. Define $\alpha_0:=\alpha$, and inductively for $n\geqslant 1$, define the sequence of real numbers $\alpha_n\in(0,1)$ as
\begin{equation}\label{equ-gauss}
\alpha_n=\frac{1}{\alpha_{n-1}}-\Big\lfloor\frac{1}{\alpha_{n-1}}\Big\rfloor, \text{ where } n\geqslant 1.
\end{equation}
Then each $\alpha_n$ has the continued fraction expansion $[0;a_{n+1},a_{n+2},\cdots]$. By definition, we have $\alpha_n\in(0,\varepsilon_4]$ for all $n\in\N$.

Let $\alpha\in \HT_N$ and $f_0\in\IS_\alpha\cup \{Q_\alpha\}$. By Theorem \ref{thm-IS-attr-rep-3}, the following sequence of maps is well-defined for all $n\geqslant 0$:
\begin{equation}
f_{n+1}:=\MMR f_n:U_{f_{n+1}}\to\C.
\end{equation}
Let $U_n:=U_{f_n}$ be the domain of definition of $f_n$ for $n\geqslant 0$. Then for all $n$, we have
\begin{equation}
f_n:U_n\to \C, ~f_n(0)=0,~f_n'(0)=e^{2\pi\ii\alpha_n} \text{\quad and\quad} \cv=\cv_{f_n}=-4/27.
\end{equation}

For $n\geqslant 0$, let $\Phi_n:=\Phi_{f_n}$ be the Fatou coordinate of $f_n:U_n\to\C$ defined in the perturbed petal $\MP_n:=\MP_{f_n}$ and let $\MC_n:=\MC_{f_n}$ and $\MC_n^\sharp:=\MC_{f_n}^\sharp$ be the corresponding sets for $f_n$ defined in \eqref{defi-C-f-alpha}. Let $k_n:=k_{f_n}$ be the positive integer in Proposition \ref{prop-CC-2} such that
\begin{equation}
S_n^0:=S_{f_n}=\MC_n^{-k_n}\cup(\MC_n^\sharp)^{-k_n}\subset\{z\in\MP_n:0<\re\Phi_n(z)<\lfloor\tfrac{1}{\alpha_n}\rfloor-\textbf{\textit{k}}-\tfrac{1}{2}\}.
\end{equation}
For $n\geqslant 0$, let $\widetilde{\MD}_n:=\widetilde{\MD}_{f_n}$ and $\MD_n:=\MD_{f_n}$ be the sets defined in \eqref{equ-MD-tilde-f} and \eqref{equ-MD-f} respectively. Note that $\MD_n\subset\widetilde{\MD}_n$ by Lemma \ref{lema:D-n}. According to Lemma \ref{lema:Phi-inverse}, we have a holomorphic map
\begin{equation}
\Phi_n^{-1}:\widetilde{\MD}_n\to U_n\setminus\{0\}
\end{equation}
such that $\Phi_n^{-1}(\zeta+1)=f_n\circ\Phi_n^{-1}(\zeta)$ if $\zeta$, $\zeta+1\in\widetilde{\MD}_n$. We denote the lift $\chi_{f_n,0}$ in \eqref{equ-chi-choice} by $\chi_{n,0}$. Then, for $n\geqslant 1$ we have
\begin{equation}\label{equ-chi-n-0}
\chi_{n,0}(\widetilde{\MD}_n)\subset\{\zeta\in\C:1\leqslant\re\zeta<\kc_1+2 \text{ and }\im\zeta>-2\}\subset\Phi_{n-1}(\MP_{n-1}).
\end{equation}
Each $\chi_{n,0}$ is anti-holomorphic. For $j\in\Z$ we define
\begin{equation}\label{eq:chi-n}
\chi_{n,j}:=\chi_{n,0}+j.
\end{equation}
In the following we are mainly interested in $\chi_{n,j}$ with $0\leqslant j\leqslant a_n=\lfloor\tfrac{1}{\alpha_{n-1}}\rfloor$.

\medskip
For $\delta>0$, let $B_\delta(X)$ be the $\delta$-neighborhood of a set $X\subset\C$ with respect to the Euclidean metric.
The following lemma will be used to prove the uniform contraction with respect to the hyperbolic metrics in the domains of adjacent renormalization levels (see Lemma \ref{lema:exp-conv}).

\begin{lem}[{\cite[Lemma 2.1]{AC18}}]\label{lema:comp-inclu}
There exists a constant $\delta_0>0$ depending only on the class $\IS_0$, such that for all $n\geqslant 1$ and $0\leqslant j\leqslant a_n$, then
\begin{equation}
B_{\delta_0}\big(\chi_{n,j}(\MD_n)\big)\subset\MD_{n-1}.
\end{equation}
\end{lem}

For $n\geqslant 0$, recall that $\MP_n$ is the perturbed petal of $f_n$. For $n\geqslant 1$, we define an anti-holomorphic map $\psi_n$ by
\begin{equation}\label{equ-psi-n}
\psi_n:=\Phi_{n-1}^{-1}\circ\chi_{n,0}\circ\Phi_n:\MP_n\to\MP_{n-1}.
\end{equation}
Hence we have the following diagrams:
\begin{equation}
\begin{CD}
\MP_{n-1} @<\Phi_{n-1}^{-1}<< \Phi_{n-1}(\MP_{n-1}) \\
@AA\psi_n A   @AA\chi_{n,0} A \\
\MP_n @>\Phi_n>> \Phi_n(\MP_n)
\end{CD}
\qquad\text{ and }\qquad
\begin{CD}
U_{n-1} @<\Phi_{n-1}^{-1}<< \MD_{n-1} \\
@.   @AA\chi_{n,j} A \\
U_n @<\Phi_n^{-1}<< \MD_n.
\end{CD}
\end{equation}
Each $\psi_n$ extends continuously to $0\in\partial\MP_n$ by mapping it to $0$.
For $n\geqslant 1$, we define the composition
\begin{equation}
\Psi_n:=\psi_1\circ\psi_2\circ\cdots\circ\psi_n:\MP_n\to\MP_0\subset U_0.
\end{equation}
For $n\geqslant 0$ and $i\geqslant 1$, define the sector
\begin{equation}
S_n^i:=\psi_{n+1}\circ\cdots\circ\psi_{n+i}(S_{n+i}^0)\subset \MP_n.
\end{equation}
In particular, $S_0^n\subset\MP_0$ for all $n\geqslant 0$. Define
\begin{equation}
\MP_n':=\{z\in\MP_n:0<\re \Phi_n(z)<\lfloor\tfrac{1}{\alpha_n}\rfloor-\textbf{\textit{k}}-1\}.
\end{equation}

Let $q_n$ be the denominator of the convergents $[0;a_1,\cdots,a_n]$ of the continued fraction expansion of $\alpha$.
Recall that $k_n=k_{f_n}$ is the positive integer introduced in Proposition \ref{prop-CC-2}. The following lemma was proved in \cite[\S 3]{Che19} and parts of the results can be also found in \cite[\S 1.5.5]{BC12}. The proof is based on the definition of near-parabolic renormalization.

\begin{lem}[{\cite[Lemmas 3.3 and 3.4]{Che19}}]\label{lema-Cheraghi-2}
For every $n\geqslant 1$, we have
\begin{enumerate}
\item For every $z\in\MP_n'$, $f_{n-1}^{\circ a_n}\circ\psi_n(z)=\psi_n\circ f_n(z)$ and $f_0^{\circ q_n}\circ\Psi_n(z)=\Psi_n\circ f_n(z)$;
\item For every $z\in S_n^0$, $f_{n-1}^{\circ (k_n a_n+1)}\circ\psi_n(z)=\psi_n\circ f_n^{\circ k_n}(z)$ and $f_0^{\circ (k_n q_n+q_{n-1})}\circ\Psi_n(z)=\Psi_n\circ f_n^{\circ k_n}(z)$; and
\item For every $m<n$, $f_n:\MP_n'\to \MP_n$ and $f_n^{\circ k_n}:S_n^0\to \MC_n\cup\MC_n^\sharp$ are conjugate to some iterates of $f_m$ on the set $\psi_{m+1}\circ\cdots\circ\psi_n(\MP_n)$.
\end{enumerate}
In particular,  the dynamics of $f_n$ is conjugate to the dynamics of $f_0$. Specifically, the first $k_n$ iterates of $f_n$ on $S_n^0$ corresponds to $k_n q_n+q_{n-1}$ iterates of $f_0$ and the next $\lfloor\tfrac{1}{\alpha_n}\rfloor-\textbf{\textit{k}}-2$ iterates corresponds to $q_n(\lfloor\tfrac{1}{\alpha_n}\rfloor-\textbf{\textit{k}}-2)$ iterates of $f_0$.
\end{lem}

For each $n\in\N$, by \eqref{equ-k-f-kc} we have
\begin{equation}\label{equ-b-n}
b_n:=k_n+\lfloor\tfrac{1}{\alpha_n}\rfloor-\textbf{\textit{k}}-2\geqslant a_{n+1}+1.
\end{equation}
From the definition of $\widetilde{\MD}_n$ in \eqref{equ-MD-tilde-f} and by Lemma \ref{lema-Cheraghi-2}, the following sets are well-defined for each $n\geqslant 0$ :
\begin{equation}\label{equ-omega-n}
\Omega_n^0:=\bigcup_{j=0}^{b_n}f_n^{\circ j}(S_n^0)\cup\{0\} \text{\quad and\quad}
\Omega_0^n:=\bigcup_{j=0}^{b_n q_n+q_{n-1}}f_0^{\circ j}(S_0^n)\cup\{0\}.
\end{equation}

\begin{defi}[High type Brjunos]
Let $N$ be the integer fixed before. Define
\begin{equation}\label{equ:Brjuno}
\MB_N:=
\left\{\alpha=[0;a_1,a_2,\cdots]\in (0,1)\setminus\Q
\left|
\begin{array}{l}
\alpha \text{ is Brjuno and}\\
a_n\geqslant N, \,\forall\, n\geqslant 1
\end{array}
\right.
\right\}.
\end{equation}
Then $\MB_N$ is strictly contained in $\HT_N$.
\end{defi}

\begin{prop}[{\cite[Propositions 3.5 and 5.10(2)]{Che19}}]\label{prop-Cherahi-nest}
Let $f_0\in\IS_\alpha\cup\{Q_\alpha\}$ with $\alpha\in \HT_N$. Then for all $n\geqslant 0$, we have
\begin{enumerate}
\item $\Omega_0^{n+1}$ is compactly contained in the interior of $\Omega_0^n$ and $f_0(\Omega_0^{n+1})\subset\Omega_0^n$;
\item If $\alpha\in\MB_N$, then $\Int\,(\bigcap_{n=0}^\infty\Omega_0^n)=\Delta_0$, where $\Delta_0$ is the Siegel disk of $f_0$.
\end{enumerate}
\end{prop}

In the rest of this paper, unless otherwise stated, for a given map $f_0\in\IS_\alpha\cup\{Q_\alpha\}$ with $\alpha\in\HT_N$, we use $f_n$ to denote the map after $n$-th near-parabolic renormalization. We also use $U_n$, $\MP_n$ and $\Phi_n$ etc to denote the domain of definition, the perturbed petal and the Fatou coordinate etc of $f_n$ respectively.

\section{The suitable heights}\label{sec-height}

\subsection{Radii of Siegel disks}

The following classical distortion theorem can be found in \cite[Theorem 1.6, p.\,21]{Pom75}.

\begin{thm}[{Koebe's distortion theorem}]\label{thm-Koebe}
Suppose $f:\D\to\C$ is a univalent map with $f(0)=0$ and $f'(0)=1$. Then for each $z\in\D$ we have
\begin{enumerate}
\item $\tfrac{1-|z|}{(1+|z|)^3}\leqslant|f'(z)|\leqslant\tfrac{1+|z|}{(1-|z|)^3}$;
\item $\tfrac{|z|}{(1+|z|)^2}\leqslant|f(z)|\leqslant\tfrac{|z|}{(1-|z|)^2}$; and
\item $|\arg f'(z)|\leqslant 2\log\tfrac{1+|z|}{1-|z|}$.
\end{enumerate}
\end{thm}

Let $\alpha_0:=\alpha\in \MB_N$ and $\alpha_n\in(0,1)$ be the number defined inductively as in \eqref{equ-gauss} for $n\geqslant 1$. Denote $\beta_{-1}=1$ and $\beta_n:=\prod_{i=0}^n\alpha_i$ for $n\geqslant 0$. The \emph{Brjuno sum} $\MB(\alpha)$ of $\alpha$ in the sense of Yoccoz is defined as
\begin{equation}\label{equ-Brjuno-sum-Yoccoz}
\MB(\alpha):=\sum_{n=0}^{+\infty}\beta_{n-1}\log\frac{1}{\alpha_n}=\log\frac{1}{\alpha_0}+\alpha_0\log\frac{1}{\alpha_1}+\alpha_0\alpha_1\log\frac{1}{\alpha_2}+\cdots.
\end{equation}
It is proved in \cite[\S 1.5]{Yoc95} that $|\MB(\alpha)-\sum_{n=0}^\infty q_n^{-1}\log q_{n+1}|\leqslant C'$ for a universal constant $C'>0$.

Suppose a holomorphic map $f$ has a Siegel disk $\Delta_f$ centered at the origin which is compactly contained in the domain of definition of $f$. The \emph{inner radius} of $\Delta_f$ is the radius of the largest open disk centered at the origin that is contained in $\Delta_f$.

\begin{lem}\label{lema-fixed-disk}
There exists a universal constant $D_7>1$ such that for all $f_0\in\IS_\alpha\cup\{Q_\alpha\}$ with $\alpha\in \MB_N$, the inner radius of the Siegel disk of $f_n$ is $c_ne^{-\MB(\alpha_n)}$ with $1/D_7 \leqslant c_n\leqslant D_7$ for every $n\in\N$.
\end{lem}

\begin{proof}
By the definition of near-parabolic renormalization, it follows that $f_n\in\IS_{\alpha_n}$ with $\alpha_n\in \MB_N$ for all $n\geqslant 1$. Then according to \cite{Brj71}, each $f_n$ with $n\geqslant 0$ has a Siegel disk centered at the origin. By the definition of Inou-Shishikura's class and Koebe's distortion theorem (Theorem \ref{thm-Koebe}(b)), $f_n$ is univalent in $\D(0,\widetilde{c})$ for a universal constant $\widetilde{c}>0$. According to Yoccoz \cite[p.\,21]{Yoc95}, the Siegel disk of $f_n$ contains a round disk $\D(0, C_1 e^{-\MB(\alpha_n)})$ for a universal constant $C_1>0$, where
\begin{equation}\label{equ:Brj-Yoccoz-n}
\MB(\alpha_n):=\log\frac{1}{\alpha_n}+\sum_{k=1}^{+\infty}\alpha_n\cdots\alpha_{n+k-1}\log\frac{1}{\alpha_{n+k}}
\end{equation}
is the Brjuno sum of $\alpha_n$ defined in \eqref{equ-Brjuno-sum-Yoccoz}.
On the other hand, by \cite[Theorem G]{Che19}, there is a universal constant $C_2>1$ such that the inner radius of the Siegel disk of $f_n$ is bounded above by $C_2 e^{-\MB(\alpha_n)}$ for all $n\in\N$. The lemma follows if we set $D_7:=\max\{C_2,1/C_1\}$.
\end{proof}

\subsection{Definition of the heights}

In the following, we use $\Delta_n$ to denote the Siegel disk of $f_n$ for all $n\geqslant 0$, where $f_0\in\IS_\alpha\cup\{Q_\alpha\}$ with $\alpha\in \MB_N$ and $f_n$ is obtained by applying the near-parabolic renormalization operator.

\begin{defi}[{The heights}]
Let $M\geqslant 1$. For $n\geqslant 0$, we define
\begin{equation}\label{defi-eta}
h_n:=\frac{\MB(\alpha_{n+1})}{2\pi}+\frac{M}{\alpha_n}.
\end{equation}
\end{defi}

There are many choices of the height $h_n$. One of the candidates is $\tfrac{\MB(\alpha_{n+1})}{2\pi}+M$. In order to apply Lemma \ref{lema-key-estimate-lp}(a) directly, we choose $h_n$ above so that $h_n> 1/\alpha_n$.
Similar to \eqref{defi-C-f-alpha} (see Figure \ref{Fig_near-para-norm-defi}), we define
\begin{equation}
\widetilde{\MC}_n^\sharp:=\{z\in\MP_n:1/2\leqslant\re\Phi_n(z)\leqslant 3/2 \text{~and~} \im\Phi_n(z)\geqslant h_n\}.
\end{equation}
Let $(\widetilde{\MC}_n^\sharp)^{-k_n}$ be the component of $f_n^{-k_n}(\widetilde{\MC}_n^\sharp)$ contained in $(\MC_n^\sharp)^{-k_n}$. Recall that $\psi_n$ is defined in \eqref{equ-psi-n}. For $n\geqslant 0$ and $i\geqslant 1$, we denote
\begin{equation}
V_n^0:=(\widetilde{\MC}_n^\sharp)^{-k_n}\subset S_n^0 \text{\quad and\quad} V_n^i:=\psi_{n+1}\circ\cdots\circ\psi_{n+i}(V_{n+i}^0)\subset S_n^i.
\end{equation}

\begin{lem}\label{lema-Jordan-domain}
There exists a universal constant $M_1\geqslant 1$ such that if $M\geqslant M_1$, then for all $n\geqslant 0$ and $i\geqslant 0$, $V_n^i$ is compactly contained in $\Delta_n$.
\end{lem}

\begin{proof}
We first prove that $V_n^0$ is compactly contained in $\Delta_n$ for all $n\geqslant 0$ if $M\geqslant 1$ is large enough. By a straightforward calculation, the image of $\Phi_n(\widetilde{\MC}_n^\sharp)$ under $\Expo$ is a punctured rounded disk centered at the origin with radius
\begin{equation}
\iota_n:=\frac{4}{27}e^{-2\pi h_n}=\frac{4}{27}e^{-\frac{2\pi M}{\alpha_n}}\cdot e^{-\MB(\alpha_{n+1})}<\frac{1}{D_7}e^{-\MB(\alpha_{n+1})}
\end{equation}
if $M\geqslant M_1:=\frac{1}{2\pi}\log D_7+1$, where $D_7> 1$ is the universal constant introduced in Lemma \ref{lema-fixed-disk}. This implies that $\Expo\circ\Phi_n(\widetilde{\MC}_n^\sharp)$ is compactly contained in the Siegel disk of $f_{n+1}$ if $M\geqslant M_1$. Hence there exists a small open neighborhood $D$ of $\widetilde{\MC}_n^\sharp$ in $\MP_n$ such that $\Expo\circ\Phi_n(D)$ is compactly contained in the Siegel disk $\Delta_{n+1}$. By Lemma \ref{lema-Cheraghi-2}(c), it follows that $f_n$ can be iterated infinitely many times in $D$ and the orbit is compactly contained in the domain of definition of $f_n$. Note that $0$ is contained in $\overline{D}$. Therefore, $D$ is contained in the Siegel disk of $f_n$ and $\widetilde{\MC}_n^\sharp\Subset\Delta_n$. Since $f_n^{\circ k_n}(V_n^0)=\widetilde{\MC}_n^\sharp$ and $0\in\partial V_n^0$, we have $V_n^0\Subset\Delta_n$.

For each $z\in V_n^0$, there exists a small open neighborhood of $z$ on which $f_n$ can be iterated infinitely many times. By Lemma \ref{lema-Cheraghi-2}(b), there exists a small open neighborhood of $\Psi_n(z)\in V_0^n$ on which $f_0$ can be also iterated infinitely many times. Since each $z\in V_n^0$ satisfies this property and $0\in\partial V_0^n$, it follows that $V_0^n\Subset\Delta_0$. By a completely similar argument, we have $V_n^i\Subset \Delta_n$ for any $i> 0$ and $n> 0$.
\end{proof}

Note that the forward orbit of $V_n^i$ is compactly contained in $\Delta_n$ for any $n\geqslant 0$ and $i\geqslant 0$. Moreover, the backward orbit of $V_n^i$ is also compactly contained in $\Delta_n$ if the preimage under $f_n$ is chosen in $\Delta_n$.
In the following, we always assume that $M\geqslant M_1$ unless otherwise stated.

\subsection{The location of the neighborhoods}\label{subsec-loc-neig}

For $n\geqslant 0$, each $V_n^0\cup\{0\}$ is a closed topological triangle\footnote{Here we use the fact that for any $x\in (0,\lfloor\tfrac{1}{\alpha_n}\rfloor-\kc)$, $\lim_{y\to+\infty}\Phi_n^{-1}(x+y\ii)=0$. See \cite[Proposition 2.4(a)]{CS15} or \cite[Lemma 6.9]{Che19}.} whose boundary consists of three analytic curves. We use $\partial^l V_n^0$, $\partial^r V_n^0$ and $\partial^b V_n^0$ to denote the three smooth edges of $V_n^0$, where $f_n(\partial^l V_n^0)=\partial^r V_n^0$ and $\partial^l V_n^0\cap \partial^r V_n^0=\{0\}$. The superscripts `$l$', `$r$' and `$b$' denote `left', `right' and `bottom', respectively. See Figure \ref{Fig_Chessboard-n-v-0}.

\begin{figure}[!htpb]
  \setlength{\unitlength}{1mm}
  \setlength{\fboxsep}{0pt}
  \centering
  \fbox{\includegraphics[width=0.8\textwidth]{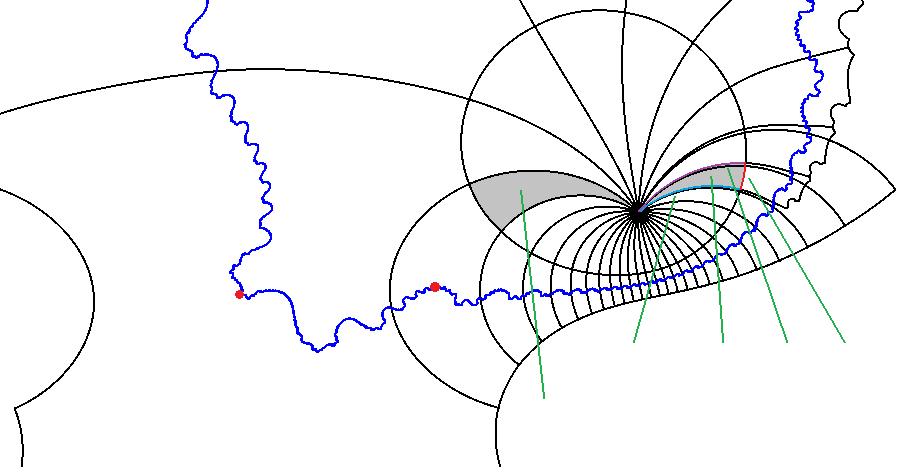}}
  \put(-25,9.5){$V_n^0$}
  \put(-36,9.5){$\partial^l V_n^0$}
  \put(-18,9.5){$\partial^r V_n^0$}
  \put(-8.5,9.5){$I_n^0$}
  \put(-43,2.7){$\widetilde{\MC}_n^\sharp$}
  \put(-78.5,15.5){$\cp_{f_n}$}
  \put(-55.5,16){$\cv$}
  \caption{In the dynamical plane of $f_n$, the sets $\partial^l V_n^0$, $\partial^r V_n^0$ and $I_n^0$ are colored cyan, purple and red respectively. The blue set depicts the (partial) forward orbit of the critical point $\cp_{f_n}$. The sets $V_n^0$ and $\widetilde{\MC}_n^\sharp=f_n^{\circ k_n}(V_n^0)$ are colored gray.}
  \label{Fig_Chessboard-n-v-0}
\end{figure}

Similar naming convention is adopted to $V_n^i$ and their forward images for all $n\geqslant 0$ and $i\geqslant 0$. For example, $\partial^l V_n^i:=\psi_{n+1}\circ\cdots\circ\psi_{n+i}(\partial^l V_{n+i}^0)$ if $i$ is even while $\partial^l V_n^i:=\psi_{n+1}\circ\cdots\circ\psi_{n+i}(\partial^r V_{n+i}^0)$ if $i$ is odd (note that each $\psi_j$ is anti-holomorphic). For simplicity, we denote the segment
\begin{equation}\label{eq-I-n}
I_n^0:=\partial^b V_n^0\subset\Delta_n.
\end{equation}
The `left' and the `right' end points of $I_n^0$ are denoted by $\partial^l I_n^0$ and $\partial^r I_n^0$ respectively so that $f_n(\partial^l I_n^0)=\partial^r I_n^0$. Similar naming convention is adopted to $I_n^i$ and their forward images for all $n\geqslant 0$ and $i\geqslant 0$. In particular, by Lemma \ref{lema-Cheraghi-2}(a) we have $f_0^{\circ q_n}(\partial^l I_0^n)=\partial^r I_0^n$ if $n$ is even and $f_0^{\circ q_n}(\partial^r I_0^n)=\partial^l I_0^n$ if $n$ is odd.
Moreover, let $\partial^l S_n^i$ and $\partial^r S_n^i$ be the smooth edges of $S_n^i$ containing $\partial^l V_n^i$ and $\partial^r V_n^i$ respectively.

Let $k_n=k_{f_n}\geqslant 1$ be the integer introduced in Proposition \ref{prop-CC-2}, $D_3>0$ be a constant introduced in Lemma \ref{lema-key-estimate-lp} and $\MD_n=\MD_{f_n}$ be the set defined in \eqref{equ-MD-f}.

\begin{lem}[{see Figure \ref{Fig_W-n}}]\label{lema-stru-R-n}
There exists a constant $M_2\geqslant 1$ such that if $M\geqslant M_2$, then for all $n\in\N$, we have
\begin{enumerate}
\item $\diam(\Phi_n(I_n^0))\leqslant 2$ and $|\im\zeta- h_n|\leqslant 1$ for all $\zeta\in\Phi_n(I_n^0)$;
\item For all $y\geqslant  h_n-1$, $u_n(y):=\{\zeta\in\C:\im\zeta=y\}\cap\Phi_n(\partial^l S_n^0)$ is a singleton;
\item $\diam(\beta_n')\leqslant 1$, where $\beta_n'$ is the arc in $\Phi_n(\partial^l S_n^0)$ connecting $u_n( h_n)$ with $\Phi_n(\partial^l I_n^0)$.
\end{enumerate}
\end{lem}

\begin{proof}
The proof is mainly based on applying Koebe's distortion theorem and the definition of near-parabolic renormalization.

\medskip
(a) By the definition of near-parabolic renormalization, we have
\begin{equation}
f_{n+1}(\Expo\circ\Phi_n(V_n^0))=\Expo\circ\Phi_n(\widetilde{\MC}_n^\sharp).
\end{equation}
Note that $\Expo\circ\Phi_n(\widetilde{\MC}_n^\sharp)\cup\{0\}$ is a closed round disk with radius
\begin{equation}
\iota_n=\frac{4}{27}e^{-\frac{2\pi M}{\alpha_n}}\cdot e^{-\MB(\alpha_{n+1})}.
\end{equation}
By Lemma \ref{lema-fixed-disk}, $\Delta_{n+1}$ contains the disk $\D(0,\varsigma_n)$, where
\begin{equation}
\varsigma_n:=D_7^{-1} e^{-\MB(\alpha_{n+1})}.
\end{equation}
Therefore,
\begin{equation}\label{equ-inv-f-g}
g:=f_{n+1}^{-1}:\D(0,\varsigma_n)\to\Delta_{n+1}
\end{equation}
is a well-defined univalent map with $|g'(0)|=1$. If $M$ is large enough such that $\iota_n$ is much smaller than $\varsigma_n$, then by Theorem \ref{thm-Koebe} the distortion of the circle $g(\partial\D(0,\iota_n))$ relative to $\partial\D(0,\iota_n)$ can be arbitrarily small. Part (a) is proved if we notice that $\Phi_n(I_n^0)$ is the closure of a connected component of $\Expo^{-1}\circ g(\partial\D(0,\iota_n)\setminus\{\iota_n\})$.

\medskip
(b) Still by the definition of near-parabolic renormalization, we have
\begin{equation}
f_{n+1}(\Expo\circ\Phi_n(\partial^l S_n^0))=(0,\tfrac{4}{27}e^{4\pi}].
\end{equation}
Since $\D(0,\varsigma_n)\subset\Delta_{n+1}$, we have $f_{n+1}^{-1}([0,\tfrac{4}{27}e^{4\pi}])\cap g(\D(0,\varsigma_n))=g([0,\varsigma_n))$, where $g$ is defined in \eqref{equ-inv-f-g}. On the other hand, by \eqref{equ-inv-f-g} and Theorem \ref{thm-Koebe}(b), we assume that $M$ is large such that $\iota_n$ is small and $g(\D(0,\varsigma_n))\supset\overline{\D}(0,e^{2\pi}\iota_n)$. According to Theorem \ref{thm-Koebe}(c), we assume further that $M$ is large such that $g([0,\varsigma_n))\cap \partial\D(0,r)$ is a singleton for any $0<r\leqslant e^{2\pi}\iota_n$. Therefore,
\begin{equation}
\Expo\circ\Phi_n(\partial^l S_n^0)\cap\{z\in\C:|z|=r\}
\end{equation}
is a singleton, where $0<r\leqslant e^{2\pi}\iota_n=\tfrac{4}{27}e^{-2\pi( h_n-1)}$. This proves Part (b).

\medskip
(c) By the definition of near-parabolic renormalization, we have
\begin{equation}
\Expo(u_n( h_n))=g([0,\varsigma_n))\cap\partial\D(0,\iota_n) \text{\quad and\quad} \Expo\circ\Phi_n(\partial^l I_n^0)=g(\iota_n).
\end{equation}
Moreover, by the definition of $\beta_n'$ we have $\Expo(\beta_n')\subset g([0,\varsigma_n))$. By Theorem \ref{thm-Koebe}, the Euclidean length of the arc $\Expo(\beta_n')$ with end points $g([0,\varsigma_n))\cap\partial\D(0,\iota_n)$ and $g(\iota_n)$ can be arbitrarily small if $M$ is large enough. This proves Part (c).
\end{proof}

Let $D_3>0$ be introduced in Lemma \ref{lema-key-estimate-lp}.
In the following we always assume that $M\geqslant\max\{M_2,D_3+\tfrac{1}{2\pi}\log\tfrac{4D_7}{27}+2\}$ unless otherwise stated. Then
\begin{equation}\label{equ-y-n-alpha}
y_n:=\frac{\MB(\alpha_{n+1})}{2\pi}+M-D_3-\frac{3}{2} > \frac{\MB(\alpha_{n+1})}{2\pi}-\frac{1}{2\pi}\log\frac{27 c_{n+1}}{4}.
\end{equation}
This implies that if $\im\zeta\geqslant y_n$, then $\zeta\in\Expo^{-1}(\Delta_{n+1})$.

\section{The sequence of the curves is convergent}\label{sec-convergent}

In this section, we define a sequence of continuous curves $(\gamma_n^i)_{n\in\N}$ in the Fatou coordinate planes with $i\in\N$. The image of each $\gamma_n^i$ under $\Phi_n^{-1}$ is a continuous closed curve contained in the Siegel disk $\Delta_n$ of $f_n$. We shall prove that $(\gamma_0^n)_{n\in\N}$ convergents uniformly to the boundary of $\Delta_0$.

\subsection{Definition of the curves and its parametrization}

For each $n\in\N$, note that $a_{n+1}=\lfloor\tfrac{1}{\alpha_n}\rfloor$. Recall that
\begin{equation}
u_n:=u_n( h_n)=\{\zeta\in\C:\im\zeta= h_n\}\cap\Phi_n(\partial^l S_n^0)
\end{equation}
is introduced in Lemma \ref{lema-stru-R-n}(b). Since $f_n^{\circ k_n}(S_n^0)=\MC_n\cup\MC_n^\sharp$, we have $\re\zeta>a_{n+1}-\kc$ for all $\zeta\in\Phi_n(S_n^0)+k_n$. Therefore, we have
\begin{equation}\label{equ-u-n-domain}
a_{n+1}-\kc-k_n< \re u_n< a_{n+1}-\kc-\tfrac{3}{2}.
\end{equation}
We denote
\begin{equation}
u_n':=a_{n+1}-\kc-k_n-\tfrac{1}{2}+ h_n\ii.
\end{equation}
According to \eqref{equ-u-n-domain}, we have $\re u_n'<\re u_n$.
Denote
\begin{equation}
u_n'':=\Phi_n(\partial^l I_n^0).
\end{equation}
Let $\beta_n'$ be the arc in $\Phi_n(\partial^l S_n^0)$ connecting $u_n$ with $u_n''$. See Figure \ref{Fig_W-n}.
We first give the definitions of two curves $\gamma_n^0(t)$ and $\gamma_n^1(t)$, where $t\in[0,1]$, and then define the curves $(\gamma_n^i(t))_{n\in\N}$ inductively.

\medskip
\textbf{Definition of $\gamma_n^0$}: The curve $\gamma_n^0(t):[0,1]\to\C$ is defined piecewise as following:
\begin{enumerate}
\item[(a$_0$)] For $t\in[0,1-\tfrac{\kc\,+k_n+1}{a_{n+1}}]$, define $\gamma_n^0(t):=a_{n+1}t+\tfrac{1}{2}+ h_n\ii$;
\item[(b$_0$)] Let $\gamma_n^0(t):[1-\tfrac{\kc\,+k_n+1}{a_{n+1}},1-\tfrac{k_n}{a_{n+1}}]\to [u_n',u_n]\cup\beta_n'$ be a homeomorphism such that
\begin{equation}
\gamma_n^0(1-\tfrac{\kc\,+k_n+1}{a_{n+1}})=u_n' \text{ and } \gamma_n^0(1-\tfrac{k_n}{a_{n+1}})=u_n'';
\end{equation}
\item[(c$_0$)] Let $\gamma_n^0(t):[1-\tfrac{k_n}{a_{n+1}},1-\tfrac{k_n-1}{a_{n+1}}]\to \Phi_n(I_n^0)$ be a homeomorphism such that
\begin{equation}
\gamma_n^0(1-\tfrac{k_n}{a_{n+1}})=u_n'' \text{ and }\gamma_n^0(1-\tfrac{k_n-1}{a_{n+1}})=u_n''+1;
\end{equation}
\item[(d$_0$)] For $t\in[1-\tfrac{k_n-j}{a_{n+1}},1-\tfrac{k_n-j-1}{a_{n+1}}]$ with $1\leqslant j\leqslant k_n-1$, define $\gamma_n^0(t):=\gamma_n^0(t-\tfrac{j}{a_{n+1}})+j$.
\end{enumerate}

\begin{figure}[!htpb]
  \setlength{\unitlength}{1mm}
  \centering
  \includegraphics[width=0.95\textwidth]{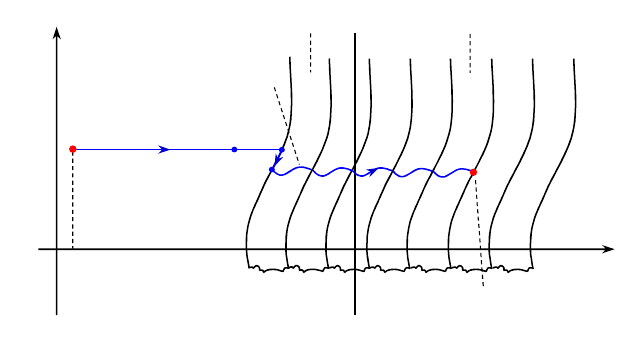}
  \put(-91,31){$\gamma_n^0$}
  \put(-77,37){$u_n'$}
  \put(-71,37){$u_n$}
  \put(-70,27){$u_n''$}
  \put(-32,6){$u_n''+k_n$}
  \put(-77,49){$\Phi_n(I_n^0)$}
  \put(-114,34){$ h_n$}
  \put(-113,12){$0$}
  \put(-108,12){$\tfrac{1}{2}$}
  \put(-52,7){$\lfloor\tfrac{1}{\alpha_n}\rfloor-\kc$}
  \put(-67,59){$\Phi_n(S_n^0)$}
  \put(-45,59){$\Phi_n(S_n^0)+k_n-1$}
  \caption{The sketch of the construction of the continuous curve $\gamma_n^0$ (in blue) in the Fatou coordinate plane of $f_n$. The two red dots denote the initial and terminal points of $\gamma_n^0$ and they have the same image under the map $\Phi_n^{-1}$. In particular, $\Phi_n^{-1}(\gamma_n^0)$ is a continuous closed curve in the Siegel disk of $f_n$.}
  \label{Fig_W-n}
\end{figure}

\begin{lem}[{See Figure \ref{Fig_W-n}}]\label{lema-gamma-n-0}
The map $\gamma_n^0(t):[0,1]\to\C$ has the following properties:
\begin{enumerate}
\item $\gamma_n^0$ and $\gamma_n^0+1$ are simple arcs in $\MD_n$;
\item $\gamma_n^0(0)=\tfrac{1}{2}+ h_n\ii$ and $\gamma_n^0(1)=u_n''+k_n$;
\item $\Phi_n^{-1}(\gamma_n^0)$ is a continuous closed curve in $\Delta_n$; and
\item $|\im\gamma_n^0(t)- h_n|\leqslant 1$ for all $t\in[0,1]$.
\end{enumerate}
\end{lem}

\begin{proof}
Parts (a) and (b) follow from the definition of $\gamma_n^0$.
For Part (c), since $f_n^{\circ k_n}(\Phi_n^{-1}(u_n''))=\Phi_n^{-1}(\tfrac{1}{2}+ h_n\ii)$, we have $\Phi_n^{-1}(\tfrac{1}{2}+ h_n\ii)=\Phi_n^{-1}(u_n''+k_n)$ by Lemma \ref{lema:Phi-inverse}. This implies that $\Phi_n^{-1}(\gamma_n^0)$ is a continuous closed curve in $\Delta_n$.
Part (d) is an immediate consequence of Lemma \ref{lema-stru-R-n}(a)(c).
\end{proof}

Before introducing $\gamma_n^1$, we define a thickened curve $\widetilde{\gamma}_n^0(t):[0,1]\to\C$ of $\gamma_n^0$:
\begin{equation}\label{equ-gamma-tilde-n-0}
\widetilde{\gamma}_n^0(t):=
\left\{
\begin{aligned}
& \gamma_n^0\Big(\tfrac{a_{n+1}}{a_{n+1}-1}t\Big)    & ~~~\text{if } &  t\in[0,1-\tfrac{1}{a_{n+1}}],\\
& \gamma_n^0(t)+1 & ~~~\text{if } &  t\in (1-\tfrac{1}{a_{n+1}},1].
\end{aligned}
\right.
\end{equation}
One can see that $\widetilde{\gamma}_n^0=\gamma_n^0\cup (\gamma_n^0([1-\tfrac{1}{a_{n+1}},1])+1)=\gamma_n^0\cup (\Phi_n(I_n^0)+k_n)$ and $\widetilde{\gamma}_n^0(t):[0,1]\to\C$ is a continuous curve in $\MD_n$. Let $\chi_{n,0}:=\chi_{f_n,0}$ be the anti-holomorphic map defined in \eqref{equ-chi-choice}.

\medskip
\textbf{Definition of $\gamma_n^1$}: The curve $\gamma_n^1(t):[0,1]\to\C$ is defined piecewise as following:
\begin{enumerate}
\item[(a$_1$)] For $t\in[0,\tfrac{1}{a_{n+1}}]$, define $\gamma_n^1(t):=\chi_{n+1,0}\circ \widetilde{\gamma}_{n+1}^0(1-a_{n+1}t)$;
\item[(b$_1$)] For $t\in (\tfrac{j}{a_{n+1}},\tfrac{j+1}{a_{n+1}}]$, where $1\leqslant j\leqslant a_{n+1}-1$, define
\begin{equation}
\gamma_n^1(t):=\chi_{n+1,j}\circ \gamma_{n+1}^0(j+1-a_{n+1}t),
\end{equation}
where $\chi_{n+1,j}=\chi_{n+1,0}+j$ is defined in \eqref{eq:chi-n}.
\end{enumerate}

\medskip
Let $D_3>0$ be the constant introduced in Lemma \ref{lema-key-estimate-lp}.

\begin{lem}\label{lema-gamma-n-1}
The map $\gamma_n^1(t):[0,1]\to\C$ has the following properties:
\begin{enumerate}
\item $\gamma_n^1$ and $\gamma_n^1+1$ are continuous curves in $\MD_n$;
\item $\gamma_n^1(0)=\chi_{n+1,0}(\gamma_{n+1}^0(1)+1)$ and $\gamma_n^1(1)=\chi_{n+1,0}(\gamma_{n+1}^0(1))+a_{n+1}$;
\item $\Phi_n^{-1}(\gamma_n^1(0))=\Phi_n^{-1}(\gamma_n^1(1))$ and $\Phi_n^{-1}(\gamma_n^1)$ is a continuous closed curve in $\Delta_n$; and
\item There exists a constant $D_8>0$ which is independent of $n$ such that for all $t\in[0,1]$, $|\re\gamma_n^0(t)-\re\gamma_n^1(t)|\leqslant D_8$ and
$|\im\gamma_n^1(t)-\tfrac{\MB(\alpha_{n+1})}{2\pi}-M|\leqslant D_3+\tfrac{1}{2}$.
\end{enumerate}
\end{lem}

\begin{proof}
(a) Since $\chi_{n+1,j}$ is anti-holomorphic for all $j\in\Z$, we have
\begin{equation}
\chi_{n+1,j}(\gamma_{n+1}^0(0))=\chi_{n+1,j}(\gamma_{n+1}^0(1))+1=\chi_{n+1,j+1}(\gamma_{n+1}^0(1)),
\end{equation}
where $0\leqslant j\leqslant a_{n+1}-2$. Therefore, $\gamma_n^1(t):[0,1]\to\C$ is a continuous curve. By Lemma \ref{lema:comp-inclu}, $\gamma_n^1$ and $\gamma_n^1+1$ are continuous curves in $\MD_n$.

\medskip
(b) By the definition of $\gamma_n^1$, we have
\begin{equation}
\gamma_n^1(0)=\chi_{n+1,0}\circ \widetilde{\gamma}_{n+1}^0(1)=\chi_{n+1,0}(\gamma_{n+1}^0(1)+1)
\end{equation}
and
\begin{equation}
\begin{split}
\gamma_n^1(1)=&~\chi_{n+1,a_{n+1}-1}(\gamma_{n+1}^0(0)) \\
=&~\chi_{n+1,a_{n+1}-1}(\gamma_{n+1}^0(1))+1=\chi_{n+1,0}(\gamma_{n+1}^0(1))+a_{n+1}.
\end{split}
\end{equation}

\medskip
(c) By Lemma \ref{lema-Cheraghi-2}(a), we have
\begin{equation}
\Phi_n^{-1}\circ\chi_{n+1,0}(\gamma_{n+1}^0(1)+1)=f_n^{\circ a_{n+1}}(\Phi_n^{-1}\circ\chi_{n+1,0}(\gamma_{n+1}^0(1))).
\end{equation}
This implies that $\Phi_n^{-1}(\gamma_n^1(0))=\Phi_n^{-1}(\gamma_n^1(1))$ by Part (b). Therefore, $\Phi_n^{-1}(\gamma_n^1)$ is a continuous closed curve in $\Delta_n$.

\medskip
(d) By \eqref{equ-chi-n-0} we have
\begin{equation}\label{equ-chi-n-1}
\re\chi_{n+1,j}(\widetilde{\gamma}_{n+1}^0)\subset[1+j,\kc_1+2+j], \text{ where } j\in\Z.
\end{equation}
Hence for $t\in[0,1-\tfrac{\kc\,+k_n+1}{a_{n+1}}]$, we have
\begin{equation}
|\re\gamma_n^0(t)-\re\gamma_n^1(t)|\leqslant \kc_1+\tfrac{3}{2}.
\end{equation}
For $t\in[1-\tfrac{\kc\,+k_n+1}{a_{n+1}},1-\tfrac{k_n}{a_{n+1}}]$, by \eqref{equ-u-n-domain} and Lemma \ref{lema-stru-R-n}(c) we have
\begin{equation}
\re\gamma_n^0(t)\in[\re u_n'-\tfrac{1}{2},\re u_n+1]\subset[a_{n+1}-\kc\,-k_n-1,a_{n+1}-\kc-\tfrac{1}{2}].
\end{equation}
If $t\in[1-\tfrac{\kc\,+k_n+1}{a_{n+1}},1-\tfrac{k_n}{a_{n+1}}]$, then $\gamma_n^1(t)\in\bigcup_{i=0}^{\kc}\chi_{n+1,a_{n+1}-\kc-k_n-1+i}(\gamma_{n+1}^0)$. By \eqref{equ-chi-n-1} we have
\begin{equation}
\re\gamma_n^1(t)\in[a_{n+1}-k_n-\kc,\,a_{n+1}-k_n+\kc_1+1].
\end{equation}
Therefore, for $t\in[1-\tfrac{\kc\,+k_n+1}{a_{n+1}},1-\tfrac{k_n}{a_{n+1}}]$ we have
\begin{equation}
|\re\gamma_n^0(t)-\re\gamma_n^1(t)|\leqslant \max\{k_n-\tfrac{1}{2},\kc+\kc_1+2\}.
\end{equation}

By Lemma \ref{lema-stru-R-n}(a)(c), we have
\begin{equation}\label{equ-u-n-1}
u_n''\in\overline{\D}(u_n,1) \text{ and } \Phi_n(I_n^0)\subset \overline{\D}(u_n'',2).
\end{equation}
For $t\in [1-\tfrac{k_n}{a_{n+1}},1-\tfrac{k_n-1}{a_{n+1}}]$, by \eqref{equ-u-n-domain} and \eqref{equ-u-n-1} we have
\begin{equation}
\re\gamma_n^0(t)\in[a_{n+1}-\kc-k_n-3,a_{n+1}-\kc+\tfrac{3}{2}].
\end{equation}
On the other hand, we have
\begin{equation}
\re\gamma_n^1(t)\in[a_{n+1}-k_n+1,a_{n+1}-k_n+\kc_1+2].
\end{equation}
Since $\gamma_n^i(t+\tfrac{1}{a_{n+1}})=\gamma_n^i(t)+1$ for $t\in [1-\tfrac{k_n}{a_{n+1}},1-\tfrac{1}{a_{n+1}}]$, where $i=0,1$, it implies that for all $t\in[1-\tfrac{k_n}{a_{n+1}},1]$, we have
\begin{equation}
|\re\gamma_n^0(t)-\re\gamma_n^1(t)|\leqslant \max\{k_n-\kc+\tfrac{1}{2},\kc+\kc_1+5\}.
\end{equation}
Since $k_n\leqslant\kc_0$ by Proposition \ref{prop-CC-2}, it implies that $|\re\gamma_n^0(t)-\re\gamma_n^1(t)|\leqslant D_8:=\max\{\kc_0-\tfrac{1}{2},\kc+\kc_1+5\}$ for all $t\in[0,1]$. Finally, the statement on $\im\gamma_n^1(t)$ follows immediately from Lemma \ref{lema-key-estimate-lp}(a) and Lemma \ref{lema-gamma-n-0}(d).
\end{proof}

By \eqref{equ-y-n-alpha} and Lemma \ref{lema-gamma-n-1}(d), for any $t\in[0,1]$ and $\zeta\in\Expo^{-1}(\partial\Delta_{n+1})$, we have
\begin{equation}\label{equ-Siegel-below}
\im\gamma_n^1(t)\geqslant \frac{\MB(\alpha_{n+1})}{2\pi}+M-D_3-\frac{1}{2} >1+\im\zeta.
\end{equation}
For $\ell=1$, we define a thickened curve $\widetilde{\gamma}_n^\ell(t):[0,1]\to\C$ of $\gamma_n^\ell$:
\begin{equation}\label{equ-gam-tilde-n}
\widetilde{\gamma}_n^\ell(t):=
\left\{
\begin{aligned}
& \gamma_n^\ell\Big(\tfrac{a_{n+1}}{a_{n+1}-1}t\Big)    & ~~~\text{if } &  t\in[0,1-\tfrac{1}{a_{n+1}}],\\
& \gamma_n^\ell(t)+1 & ~~~\text{if } &  t\in (1-\tfrac{1}{a_{n+1}},1].
\end{aligned}
\right.
\end{equation}
One can see that $\widetilde{\gamma}_n^\ell=\gamma_n^\ell\cup (\gamma_n^\ell([1-\tfrac{1}{a_{n+1}},1])+1)=\gamma_n^\ell\cup\chi_{n+1,a_{n+1}}(\gamma_{n+1}^{\ell-1})$, and $\widetilde{\gamma}_n^\ell(t):[0,1]\to\C$ is a continuous curve in $\MD_n$.

\medskip
\textbf{Define $\gamma_n^i$ inductively}:  For all $n\in\N$ and $1\leqslant \ell\leqslant i$ with $i\geqslant 1$, we assume that the curves $\gamma_n^\ell(t):[0,1]\to\C$ and $\widetilde{\gamma}_n^\ell(t):[0,1]\to\C$ are defined and satisfy
\begin{enumerate}
\item[(a$_\ell$)] $\widetilde{\gamma}_n^\ell$ is defined as in \eqref{equ-gam-tilde-n};
\item[(b$_\ell$)] $\gamma_n^\ell(t):=\chi_{n+1,0}\circ \widetilde{\gamma}_{n+1}^{\ell-1}(1-a_{n+1}t)$ for $t\in [0,\tfrac{1}{a_{n+1}})$, and $\gamma_n^\ell(t):=\chi_{n+1,j}\circ \gamma_{n+1}^{\ell-1}(j+1-a_{n+1}t)$ for $t\in(\tfrac{j}{a_{n+1}},\tfrac{j+1}{a_{n+1}}]$ with $1\leqslant j\leqslant a_{n+1}-1$;
\item[(c$_\ell$)] $\gamma_n^\ell$ and $\gamma_n^\ell+1$ are continuous curves in $\MD_n$;
\item[(d$_\ell$)] $\gamma_n^\ell(0)=\chi_{n+1,0}(\gamma_{n+1}^{\ell-1}(1)+1)$ and $\gamma_n^\ell(1)=\chi_{n+1,0}(\gamma_{n+1}^{\ell-1}(1))+a_{n+1}$; and
\item[(e$_\ell$)] $\Phi_n^{-1}(\gamma_n^\ell(0))=\Phi_n^{-1}(\gamma_n^\ell(1))$ and $\Phi_n^{-1}(\gamma_n^\ell)$ is a continuous closed curve in $\Delta_n$.
\end{enumerate}

\medskip
Similar to the construction of $\gamma_n^i$, the curve $\gamma_n^{i+1}(t):[0,1]\to\C$ is defined as:
\begin{enumerate}
\item[(a$_{i+1}$)] For $t\in[0,\tfrac{1}{a_{n+1}}]$, define $\gamma_n^{i+1}(t):=\chi_{n+1,0}\circ \widetilde{\gamma}_{n+1}^i(1-a_{n+1}t)$;
\item[(b$_{i+1}$)] For $t\in (\tfrac{j}{a_{n+1}},\tfrac{j+1}{a_{n+1}}]$, where $1\leqslant j\leqslant a_{n+1}-1$, define
\begin{equation}\label{equ-defi-gam}
\gamma_n^{i+1}(t):=\chi_{n+1,j}\circ \gamma_{n+1}^i(j+1-a_{n+1}t).
\end{equation}
\end{enumerate}

\begin{lem}\label{lema-gamma-n-i}
The map $\gamma_n^{i+1}(t):[0,1]\to\C$ has the following properties:
\begin{enumerate}
\item $\gamma_n^{i+1}$ and $\gamma_n^{i+1}+1$ are continuous curves in $\MD_n$;
\item $\gamma_n^{i+1}(0)=\chi_{n+1,0}(\gamma_{n+1}^i(1)+1)$ and $\gamma_n^{i+1}(1)=\chi_{n+1,0}(\gamma_{n+1}^i(1))+a_{n+1}$;
\item $\Phi_n^{-1}(\gamma_n^{i+1}(0))=\Phi_n^{-1}(\gamma_n^{i+1}(1))$ and $\Phi_n^{-1}(\gamma_n^{i+1})$ is a continuous closed curve in $\Delta_n$.
\end{enumerate}
\end{lem}

The proof of Lemma \ref{lema-gamma-n-i} is completely similar to that of Lemma \ref{lema-gamma-n-1}. Moreover, one can define the thickened curve $\widetilde{\gamma}_n^{\ell}$ of $\gamma_n^{\ell}$  with $\ell=i+1$ as in \eqref{equ-gam-tilde-n} similarly.

\medskip
By the definition of $\widetilde{\gamma}_n^i$, we have

\begin{lem}\label{lema:gam-inverse}
For each $t_0\in[0,1]$, there exist two sequences $(t_n)_{n\in\N}$ with $t_n\in[0,1]$ and $(j_n)_{n\geqslant 1}$ with $0\leqslant j_n\leqslant a_n$, such that for all $n\geqslant 1$ and all $i\in\N$,
\begin{equation}
\widetilde{\gamma}_{n-1}^{i+1}(t_{n-1})=\chi_{n,j_n}\big(\widetilde{\gamma}_n^{i}(t_n)\big).
\end{equation}
\end{lem}

\subsection{The curves are convergent}\label{subsec-unif-cont}

Our main goal in this subsection is to prove:

\begin{prop}\label{prop-Cauchy-sequence}
There exists a constant $K>0$ such that for all $n\in\N$, we have
\begin{equation}\label{equ-Cauchy}
\sum_{i=0}^n\sup_{t\in[0,1]}|\gamma_0^i(t)-\gamma_0^{i+1}(t)|\leqslant K.
\end{equation}
In particular, the sequence of the continuous curves $(\gamma_0^n(t):[0,1]\to\C)_{n\in\N}$ converges uniformly as $n\to\infty$.
\end{prop}

In order to estimate the distance between $\gamma_0^i(t)$ and $\gamma_0^{i+1}(t)$ with $t\in[0,1]$, we will combine the uniform contraction with respect to the hyperbolic metrics and some quantitative estimates (with respect to the Euclidean metric) obtained in \S\ref{subsec-esti-2}. For any hyperbolic domain $X\subset\C$, we use $\rho_X(z)|\dd z|$ to denote the hyperbolic metric of $X$. The following lemma appears in \cite[Lemma 5.5]{Che19} in another form. For completeness we include a proof here.

\begin{lem}\label{lema-uni-con-prep}
Let $X$, $Y$ be two hyperbolic domains in $\C$ satisfying $\diam\,(\re(X))\leqslant A'$ and $B_\delta(X)\subset Y$, where $A'$ and $\delta$ are positive constants. Then there exists a number $0<\lambda<1$ depending only on $A'$ and $\delta$ such that for any $z\in X$,
\begin{equation}
\rho_Y(z)\leqslant \lambda\,\rho_X(z).
\end{equation}
\end{lem}

\begin{proof}
For any fixed $z_0\in X$, we consider the holomorphic function
\begin{equation}
F(z):=z+\frac{\delta\,(z-z_0)}{z-z_0+2A'+\delta}:X\to\C.
\end{equation}
Since $\diam\,(\re(X))\leqslant A'$, it follows that $|z-z_0|<|z-z_0+2A'+\delta|$ if $z\in X$. Thus we have $|F(z)-z|<\delta$ and $F(X)\subset Y$ by the assumption. Applying Schwarz-Pick's lemma to $F:X\to Y$ at $F(z_0)=z_0$, we have
\begin{equation}
\rho_Y(F(z_0))|F'(z_0)|=\rho_Y(z_0)\left(1+\frac{\delta}{2A'+\delta}\right)\leqslant \rho_X(z_0).
\end{equation}
The proof is finished if we set $\lambda:=(2A'+\delta)/(2A'+2\delta)$.
\end{proof}

Let $X$ be a set in $\C$ and $z_0\in X$. We use $\Comp_{z_0}X$ to denote the connected component of $X$ containing $z_0$. Let $\MD_n$ be the set defined in \eqref{equ-MD-f}. For $n\in\N$, we define
\begin{equation}\label{equ-D-n-pri}
\MD_n':=\Comp_1(\MD_n\cap\{\zeta\in\C:-3<\im\zeta< h_n+2\}),
\end{equation}
where $ h_n$ is the height defined in \eqref{defi-eta}.
Note that each $\MD_n'$ is a hyperbolic domain. Let $\rho_n(z)|\dd z|$ be the hyperbolic metric of $\MD_n'$.
We use $\len(\cdot)$ and $\len_{\rho_n}(\cdot)$ to denote the length of curves with respect to the Euclidean and the hyperbolic metric  $\rho_n(z)|\dd z|$ respectively.

\begin{lem}\label{lema:exp-conv}
Let $A'>0$ and $\delta>0$ be two constants. Then there exist $A>0$ and $0<\nu<1$ depending only on $A'$ and $\delta$ such that for any piecewise continuous curve $\vartheta_n$ in $\MD_n'$ with $\len(\vartheta_n)\leqslant A'$ and $B_\delta(\vartheta_n)\subset\MD_n'$, we have
\begin{equation}
\len(\chi_{1,j_1}\circ\cdots\circ\chi_{n,j_n}(\vartheta_n))\leqslant A\cdot\nu^n,
\end{equation}
where $0\leqslant j_i\leqslant a_i$ and $1\leqslant i\leqslant n$.
\end{lem}

\begin{proof}
Let $1\leqslant i\leqslant n$ and $0\leqslant j_i\leqslant a_i$. Note that we have assumed that $M>D_3$ in \eqref{equ-y-n-alpha}. By Lemma \ref{lema-key-estimate-lp}, for $\zeta\in\MD_i'$, we have
\begin{equation}\label{equ-im-chi-1}
\im\chi_{i,j_i}(\zeta)\leqslant\tfrac{\MB(\alpha_i)}{2\pi}+M+D_3+1<\tfrac{\MB(\alpha_i)}{2\pi}+\tfrac{M}{\alpha_{i-1}}+1= h_{i-1}+1.
\end{equation}
Since $\Phi_i^{-1}(\MD_i)$ is contained in the image of $f_i$, by the definition of near-parabolic renormalization (see also \eqref{equ-chi-choice}), we have
\begin{equation}\label{equ-im-chi-2}
\im\chi_{i,j_i}(\zeta)>-2, \text{ for all }\zeta\in\MD_i.
\end{equation}
By Lemma \ref{lema:comp-inclu}, we have $B_{\delta_0}(\chi_{i,j_i}(\MD_i))\subset\MD_{i-1}$ for a constant $\delta_0$ depending only on the class $\IS_0$. Without loss of generality, we assume that $\delta_0<1$. Combining \eqref{equ-im-chi-1} and \eqref{equ-im-chi-2}, we have
\begin{equation}\label{equ-B-delta-D}
B_{\delta_0}(\chi_{i,j_i}(\MD_i'))\subset\MD_{i-1}'.
\end{equation}

Note that $\chi_{i,j_i}:(\MD_i',\rho_i)\to (\MD_{i-1}',\rho_{i-1})$ can be decomposed as:
\begin{equation}
(\MD_i',\rho_i)  \xlongrightarrow{\chi_{i,j_i}} (\chi_{i,j_i}(\MD_i'), \tilde{\rho}_i) \overset{inc.}{\hookrightarrow} (\MD_{i-1}',\rho_{i-1}),
\end{equation}
where $\tilde{\rho}_i(z)|dz|$ is the hyperbolic metric of $\chi_{i,j_i}(\MD_i')$. According to Proposition \ref{prop-unif-inverse}, we have $\diam\,(\re\chi_{i,j_i}(\MD_i'))\leqslant \kc_1$. By Lemma \ref{lema-uni-con-prep}, the inclusion map
\begin{equation}
(\chi_{i,j_i}(\MD_i'), \tilde{\rho}_i) \overset{inc.}{\hookrightarrow} (\MD_{i-1}',\rho_{i-1})
\end{equation}
is uniformly contracting with respect to the hyperbolic metrics (and the contracting factor depends only on $\kc_1$ and $\delta_0$). Since $\chi_{i,j_i}:\MD_i'\to \chi_{i,j_i}(\MD_i')$ do not expand the hyperbolic metric, it follows that $\chi_{i,j_i}:(\MD_i',\rho_i)\to (\MD_{i-1}',\rho_{i-1})$ is also uniformly contracting.

Since $\vartheta_n$ is a piecewise continuous curve satisfying $\len(\vartheta_n)\leqslant A'$ and $B_\delta(\vartheta_n)\subset\MD_n'$, it follows that there exists a constant $A''>0$ depending only on $A'$ and $\delta$ (not on $n$) such that $\len_{\rho_n}(\vartheta_n)\leqslant A''$. Define
\begin{equation}
G_n:=\chi_{1,j_1}\circ\cdots\circ\chi_{n,j_n}: \MD_n'\to \MD_0'.
\end{equation}
By the uniform contraction of $\chi_{i,j_i}$ for $1\leqslant i\leqslant n$ with respect to the hyperbolic metrics, there exists a constant $0<\nu<1$ depending only on $\kc_1$ and $\delta_0$ such that
\begin{equation}
\len_{\rho_0} (G_n(\vartheta_n))\leqslant A''\cdot\nu^n.
\end{equation}
Since $B_{\delta_0}\big(G_n(\MD_n')\big)\subset \MD_0'$, the Euclidean metric and the hyperbolic metric $\rho_0$ of $\MD_0'$ are comparable in $G_n(\MD_n')$. Since $G_n(\vartheta_n)\subset G_n(\MD_n')\subset \MD_0'$, there exists a constant $A>0$ depending only on $A'$ and $\delta$ such that $\len (G_n(\vartheta_n))\leqslant A\cdot\nu^n$.
\end{proof}

Let $D_6'>1$ be the constant introduced in Proposition \ref{prop-key-estimate-yyy}.

\begin{lem}\label{lema-go-up}
There exists $K_1>0$ such that for any $n\geqslant 1$ and any continuous curve $\eta_n:[0,1]\to \MD_n$ with $\eta_n(0)\in\widetilde{\gamma}_n^0$ and $\len \big(\eta_n\big)\leqslant h_n-D_6'-1$, then
\begin{equation}
\len \big(\chi_{n,0}(\eta_n)\big)\leqslant \tfrac{1}{2\pi}\MB(\alpha_n)+K_1.
\end{equation}
\end{lem}

\begin{proof}
By Proposition \ref{prop-key-estimate-yyy}, we define
\begin{equation}
\begin{aligned}
& \phi_1(r):=(1+D_6 e^{-2\pi\alpha_nr})\alpha_n ~~~&\text{if }& r\in[\tfrac{1}{4\alpha_n},+\infty),\\
& \phi_2(r):=\frac{\alpha_n}{1-e^{-2\pi\alpha_n(r-D_2'\log(2+r))}}\left(1+\frac{D_6}{r}\right) ~~~&\text{if }& r\in[D_6',\tfrac{1}{4\alpha_n}].
\end{aligned}
\end{equation}
A direct calculation shows that
\begin{equation}\label{equ-integ-1}
J':=\int_{1/(4\alpha_n)}^{ h_n-1}\phi_1(r)\,\textup{d}r< \frac{1}{2\pi}\alpha_n\MB(\alpha_{n+1})+M+D_6.
\end{equation}
We claim that there exists $K_1'>0$ which is independent of $\alpha_n$ such that
\begin{equation}\label{equ-J}
J'':=\int_{D_6'}^{1/(4\alpha_n)}\phi_2(r)\,\textup{d}r< \frac{1}{2\pi}\log\frac{1}{\alpha_n}+K_1'.
\end{equation}
In fact, a direct calculation shows that $J''=J_1+D_2' J_2+D_6 J_3$, where
\begin{equation}
\begin{split}
J_1=&~ \frac{1}{2\pi}\int_{D_6'}^{\frac{1}{4\alpha_n}}\frac{2\pi\alpha_n e^{2\pi\alpha_n r}-2\pi\alpha_n D_2'(r+2)^{2\pi\alpha_n D_2'-1}}{e^{2\pi\alpha_n r}-(r+2)^{2\pi\alpha_n D_2'}}\,\textup{d}r,\\
J_2=&~ \int_{D_6'}^{\frac{1}{4\alpha_n}}\frac{\alpha_n (r+2)^{2\pi\alpha_n D_2'-1}}{e^{2\pi\alpha_n r}-(r+2)^{2\pi\alpha_n D_2'}}\,\textup{d}r,
\text{\quad and} \\
J_3=&~ \int_{D_6'}^{\frac{1}{4\alpha_n}}\frac{\alpha_n e^{2\pi\alpha_n r}}{e^{2\pi\alpha_n r}-(r+2)^{2\pi\alpha_n D_2'}}\cdot\frac{1}{r}\,\textup{d}r.
\end{split}
\end{equation}
We assume that $\alpha_n$ is small such that $2\pi\alpha_n D_2'\leqslant 1/2$ and $2\pi\alpha_n D_2'\log(2+\tfrac{1}{4\alpha_n})\leqslant 1/2$. Since $1+t\leqslant e^t\leqslant 1+2t$ for $t\in[0,1]$, if $D_6'\leqslant r\leqslant \frac{1}{4\alpha_n}$, we have
\begin{equation}\label{equ-J-1}
\begin{split}
e^{2\pi\alpha_n r}-(r+2)^{2\pi\alpha_n D_2'}
\geqslant &~ 1+2\pi\alpha_n r-(1+4\pi\alpha_n D_2'\log(r+2)) \\
= &~ 2\pi\alpha_n (r-2 D_2'\log(r+2)),
\end{split}
\end{equation}
where $r-2D_2'\log(2+r)\geqslant 4$ if $r\geqslant D_6'$ (see Proposition \ref{prop-key-estimate-yyy}(b)).

By \eqref{equ-J-1}, there exist $C_1$, $C_1'>0$ which are independent of $\alpha_n$ such that
\begin{equation}
J_1\leqslant C_1-\frac{1}{2\pi}\log(e^{2\pi\alpha_n D_6'}-(D_6'+2)^{2\pi\alpha_n D_2'})\leqslant \frac{1}{2\pi}\log\frac{1}{\alpha_n}+C_1'.
\end{equation}
For $J_2$, since the integral
\begin{equation}
\int_{D_6'}^{+\infty}\frac{1}{r-2 D_2'\log(2+r)}\cdot \frac{1}{(r+2)^{1/2}}\,\textup{d}r
\end{equation}
is convergent, it follows that there exists a constant $C_2>0$ which is independent of $\alpha_n$ so that $J_2\leqslant C_2$.
Similarly, there exists a constant $C_3>1$ which is independent of $\alpha_n$ so that $J_3\leqslant C_3$. Hence \eqref{equ-J} follows if we set $K_1':=C_1'+C_2 D_2'+C_3D_6$.

\medskip
Without loss of generality we assume that $r\mapsto r-D_2'\log(2+r)$ is monotonously increasing on $[D_6',+\infty)$. Therefore, $\phi_1(r)$ and $\phi_2(r)$ are monotonously decreasing on $[\tfrac{1}{4\alpha_n},+\infty)$ and $[D_6',\tfrac{1}{4\alpha_n}]$ respectively. Denote
\begin{equation}\label{equ:phi-r}
\phi(r):=
\left\{
\begin{aligned}
& \phi_1(r)   & ~~~\text{if } &  r\in[\tfrac{1}{4\alpha_n},+\infty),\\
& \max\big\{\phi_2(r),\phi_1(\tfrac{1}{4\alpha_n})\big\} & ~~~\text{if } &  r\in[D_6',\tfrac{1}{4\alpha_n}).
\end{aligned}
\right.
\end{equation}
Then $\phi(r)$ is monotonously (may not strictly) decreasing on $[D_6',+\infty)$.
By Lemma \ref{lema-gamma-n-0}(d), we have $|\im\eta_n(0)- h_n|\leqslant 1$. Since $\len \big(\eta_n\big)\leqslant h_n-D_6'-1$, we have $\eta_n\cap \big(\D(0,D_6')\cup\D(1/\alpha_n,D_6')\big)=\emptyset$.
By \eqref{equ-integ-1} and \eqref{equ-J} we have
\begin{equation}\label{equ-integral}
\begin{split}
\len \big(\chi_{n,0}(\eta_n)\big)
\le &~ \int_{D_6'}^{ h_n-1}\phi(r)\,\textup{d}r \leqslant J' +\Big(J''+\big(\tfrac{1}{4\alpha_n}-D_6'\big)\phi_1\big(\tfrac{1}{4\alpha_n}\big)\Big) \\
< &~ J'+J''+ \frac{1}{4}(D_6+1)<\frac{1}{2\pi}\MB(\alpha_n)+K_1,
\end{split}
\end{equation}
where $K_1:=M+\frac{3}{2}D_6+K_1'$. The proof is complete.
\end{proof}

\begin{proof}[{Proof of Proposition \ref{prop-Cauchy-sequence}}]
Note that $\gamma_0^n(t)=\widetilde{\gamma}_0^n(\tfrac{a_1-1}{a_1}t)$ for all $t\in[0,1]$ and all $n\in\N$. In order to prove \eqref{equ-Cauchy}, it suffices to prove that there exist $K>0$ and a sequence of non-negative numbers $(y_i)_{i\geqslant 0}$ such that for any $n\in\N$, any $0\leqslant i\leqslant n$ and any $t_0\in[0,1]$, we have
\begin{equation}\label{equ-convergent-1}
|\widetilde{\gamma}_0^i(t_0)-\widetilde{\gamma}_0^{i+1}(t_0)|\leqslant y_i \text{\quad and\quad} \sum_{i=0}^n y_i\leqslant K.
\end{equation}
We divide the argument into several steps.

\medskip
\textbf{Step 1. Basic settings.}
For any $t_0\in[0,1]$, by Lemma \ref{lema:gam-inverse}, there exist two sequences $(t_n)_{n\in\N}$ with $t_n\in[0,1]$ and $(j_n)_{n\geqslant 1}$ with $0\leqslant j_n\leqslant a_n$ such that for all $n\geqslant 1$ and all $i\in\N$,
\begin{equation}\label{equ-seg-t-n}
\widetilde{\gamma}_{n-1}^{i+1}(t_{n-1})=\chi_{n,j_n}\big(\widetilde{\gamma}_n^{i}(t_n)\big).
\end{equation}
For $n\in\N$, let
\begin{equation}
\xi_n^0:[0,1]\to [\widetilde{\gamma}_n^0(t_n),\widetilde{\gamma}_n^1(t_n)]
\end{equation}
be the segment with $\xi_n^0(0)=\widetilde{\gamma}_n^0(t_n)$ and $\xi_n^0(1)=\widetilde{\gamma}_n^1(t_n)$ (we assume that the parametrization of $\xi_n^0$ on $[0,1]$ is linear).

By the definition of $\MD_n'$ and Lemma \ref{lema:D-n}, the set $\MD_n'$ contains
\begin{equation}
\{\zeta\in\C:0<\re\zeta\leqslant\lfloor\tfrac{1}{\alpha_n}\rfloor-\kc+\kc_0+\kc_1+3 \text{ and } 0\leqslant\im\zeta<h_n+2\}.
\end{equation}
By Lemma \ref{lema-stru-R-n}(a)(c), \eqref{equ-u-n-domain} and Lemma \ref{lema-gamma-n-0}(d), we have
\begin{equation}
\widetilde{\gamma}_n^0\subset\{\zeta\in\C:\tfrac{1}{2}\leqslant\re\zeta\leqslant\lfloor\tfrac{1}{\alpha_n}\rfloor-\kc+k_n+\tfrac{3}{2} \text{ and }-1\leqslant\im\zeta-h_n\leqslant 1\}.
\end{equation}
In \eqref{equ-y-n-alpha} we assume that $M>D_3+\tfrac{1}{2\pi}\log\tfrac{4D_7}{27}+2> D_3+\frac{3}{2}$ (since $D_7> 1$).
Hence by \eqref{equ-chi-n-1} and Lemma \ref{lema-gamma-n-1}(d), we have
\begin{equation}\label{equ:gamma-ti-n-1}
\widetilde{\gamma}_n^1\subset\{\zeta\in\C:1\leqslant\re\zeta\leqslant\lfloor\tfrac{1}{\alpha_n}\rfloor+\kc_1+3 \text{ and }1\leqslant\im\zeta\leqslant h_n+1\}.
\end{equation}
Note that $k_n\leqslant \kc_0$ (see Proposition \ref{prop-CC-2}).
Hence we have $B_{1/2}(\xi_n^0)\subset\MD_n'$ for all $n\in\N$.
For $\ell\geqslant 1$, we define the Jordan arc $\xi_n^\ell:[0,1]\to\C$ as
\begin{equation}\label{equ-xi-i-ell}
\xi_n^\ell(s):=\chi_{n+1,j_{n+1}}\circ\cdots\circ \chi_{n+\ell,j_{n+\ell}}(\xi_{n+\ell}^0(s)), \text{ where } s\in[0,1].
\end{equation}

By \eqref{equ-seg-t-n} and \eqref{equ-xi-i-ell}, the following curve is continuous:
\begin{equation}\label{equ-Xi-i-1}
\begin{split}
\eta_n^\ell:=\xi_n^0\cup\xi_n^1\cup\cdots\cup\xi_n^\ell=&~\xi_n^0\cup\chi_{n+1,j_{n+1}}(\eta_{n+1}^{\ell-1}) \\
=&~\xi_n^0\cup\chi_{n+1,j_{n+1}}\big(\xi_{n+1}^0\cup\cdots\cup\xi_{n+1}^{\ell-1}\big).
\end{split}
\end{equation}
Denote $\eta_n^0:=\xi_n^0$.
According to \eqref{equ-B-delta-D}, for any $n\geqslant 0$ and $\ell\ge 0$ we have
\begin{equation}
B_{\delta}(\eta_n^\ell)\subset\MD_n', \text{\quad where\quad} \delta:=\min\{\delta_0,1/4\}.
\end{equation}
We give a parametrization of the continuous curve $\eta_n^\ell:[0,1]\to\C$ by
\begin{equation}\label{equ-Xi-i-para}
\eta_n^\ell(s):=\xi_n^j\big((\ell+1)s-j\big),
\end{equation}
where $s\in[\tfrac{j}{\ell+1},\tfrac{j+1}{\ell+1}]$ and $0\leqslant j\leqslant \ell$ (note that $\xi_n^j(1)=\xi_n^{j+1}(0)$ for every $0\leqslant j\leqslant \ell-1$).
By definition, we have $|\widetilde{\gamma}_0^i(t_0)-\widetilde{\gamma}_0^{i+1}(t_0)|\leqslant\len(\xi_0^i)$ for all $i\in\N$.
Therefore, in order to obtain \eqref{equ-convergent-1}, it suffices to prove that there exist $K>0$ and non-negative numbers $(y_i)_{i\geqslant 0}$ such that for any $n\in\N$ and any $0\leqslant i\leqslant n$, we have
\begin{equation}\label{equ-seq-seg}
\len (\xi_0^i)\leqslant y_i \text{\quad and\quad} \sum_{i=0}^n y_i\leqslant K.
\end{equation}

\textbf{Step 2. Decompositions of the curves.}
Note that we have assumed that $M> D_3+\frac{3}{2}$ (see \eqref{equ-y-n-alpha}).
By \eqref{equ-gam-tilde-n}, it follows that Lemma \ref{lema-gamma-n-1}(d) holds also for $\widetilde{\gamma}_n^0$ and $\widetilde{\gamma}_n^1$. By Lemma \ref{lema-gamma-n-0}(d) and a direct calculation, we have
\begin{equation}\label{equ-Xi-est}
\begin{split}
 &~\len (\eta_n^0)=\len (\xi_n^0)=|\widetilde{\gamma}_n^0(t_n)-\widetilde{\gamma}_n^1(t_n)|\\
\leqslant &~ h_n+1-\tfrac{\MB(\alpha_{n+1})}{2\pi}-M+D_3+\tfrac{1}{2}+D_8<  h_n-\tfrac{\MB(\alpha_{n+1})}{2\pi}+D_8.
\end{split}
\end{equation}
Hence $\eta_n^0=\xi_n^0:[0,1]\to\C$ can be written as the union of two continuous curves $\eta_{n,(0)}^0:=\eta_n^0([0,s_n])$ and $\eta_{n,(1)}^0:=\eta_n^0([s_n,1])$ for some $s_n\in (0,1)$ (the choice of $s_n$ is not unique), such that
\begin{equation}\label{equ:decomp-1}
\len(\eta_{n,(0)}^0)\leqslant h_n-D_6'-1 \text{\quad and\quad}\len(\eta_{n,(1)}^0)\leqslant D_6'+D_8+1.
\end{equation}
Since $B_{\delta}(\eta_n^0)\subset\MD_n'$, there exists a constant $K_2'>0$ depending only on $\delta$ and $D_6'+D_8+1$ such that
\begin{equation}\label{equ:decomp-2}
\len_{\rho_n}(\eta_{n,(1)}^0)\leqslant K_2',
\end{equation}
where $\rho_n(z)|dz|$ is the hyperbolic metric of $\MD_n'$.

\medskip
Let $K_1>0$ be the constant introduced in Lemma \ref{lema-go-up}. There exists a constant $K_2> K_2'$ depending only on $A':=K_1+D_6'+D_8+1$ and $\delta$, such that for any $n\in\N$ and any piecewise continuous curve $\xi'$ in $\MD_n'$ with $B_{\delta}(\xi')\subset\MD_n'$ and $\len(\xi')\leqslant K_1+D_6'+D_8+1$, one has
\begin{equation}\label{equ:K-2}
\len_{\rho_n}(\xi')\leqslant K_2.
\end{equation}
Let $\nu\in (0,1)$ be the number in Lemma \ref{lema:exp-conv} depending only on $A'$ and $\delta$.

Suppose $n\geqslant 1$. By Lemma \ref{lema-go-up} and \eqref{equ:decomp-1}, Lemma \ref{lema:exp-conv} and \eqref{equ:decomp-2}, $\xi_{n-1}^1$ is the union of two continuous curves $\chi_{n,j_n}(\eta_{n,(0)}^0)$ and $\chi_{n,j_n}(\eta_{n,(1)}^0)$, where
\begin{equation}\label{equ:len-1}
\begin{split}
\len\big(\chi_{n,j_n}(\eta_{n,(0)}^0)\big)\leqslant &~\tfrac{1}{2\pi}\MB(\alpha_n)+K_1 \text{\quad and }  \\
\len_{\rho_{n-1}}\big(\chi_{n,j_n}(\eta_{n,(1)}^0)\big)\leqslant &~ K_2'\nu < K_2\nu.
\end{split}
\end{equation}
Therefore, by \eqref{equ-Xi-est} we have
\begin{equation}\label{equ:len-2}
\begin{split}
&~ \len\big(\xi_{n-1}^0\cup \chi_{n,j_n}(\eta_{n,(0)}^0)\big)\\
\leqslant &~\big( h_{n-1}-\tfrac{\MB(\alpha_n)}{2\pi}+D_8\big) + \big(\tfrac{\MB(\alpha_n)}{2\pi}+K_1\big)=h_{n-1}+D_8+K_1.
\end{split}
\end{equation}
This implies that $\xi_{n-1}^0\cup \chi_{n,j_n}(\eta_{n,(0)}^0)=\xi_{n-1}^0\cup \xi_{n-1}^1([0,s_n])=\eta_{n-1}^1([0,\tfrac{1+s_n}{2}])$ can be written as the union of two continuous curves $\eta_{n-1,(0)}^1:=\eta_{n-1}^1([0,s_{n-1}])$ and $\eta_{n-1,(1)}^1:=\eta_{n-1}^1([s_{n-1},\tfrac{1+s_n}{2}])$ for some $s_{n-1}\in (0,\tfrac{1+s_n}{2})$, where
\begin{equation}\label{equ:len-3}
\begin{split}
\len(\eta_{n-1,(0)}^1)\leqslant &~ h_{n-1}-D_6'-1 \text{\quad and } \\
\len(\eta_{n-1,(1)}^1)\leqslant &~ A'=K_1+D_6'+D_8+1.
\end{split}
\end{equation}
Since $B_{\delta}(\eta_{n-1}^1)\subset\MD_{n-1}'$, by \eqref{equ:K-2} we have $\len_{\rho_{n-1}}(\eta_{n-1,(1)}^1)\leqslant K_2$.

Denote $\eta_{n-1,(2)}^1:=\eta_{n-1}^1([\frac{1+s_n}{2},1])=\chi_{n,j_n}(\eta_{n,(1)}^0)$, $s_{n-1}^{(1)}:=s_{n-1}$ and $s_{n-1}^{(2)}:=\tfrac{1+s_n}{2}$. Then the continuous curve
\begin{equation}
\eta_{n-1}^1=\xi_{n-1}^0\cup\xi_{n-1}^1=\eta_{n-1,(0)}^1\cup\eta_{n-1,(1)}^1\cup \eta_{n-1,(2)}^1
\end{equation}
satisfies:
\begin{itemize}
\item $\eta_{n-1,(0)}^1=\eta_{n-1}^1([0,s_{n-1}^{(1)}])$, $\eta_{n-1,(1)}^1=\eta_{n-1}^1([s_{n-1}^{(1)},s_{n-1}^{(2)}])$ and $\eta_{n-1,(2)}^1=\eta_{n-1}^1([s_{n-1}^{(2)},1])$; and
\item $\len(\eta_{n-1,(0)}^1)\leqslant h_{n-1}-D_6'-1$, $\len_{\rho_{n-1}}(\eta_{n-1,(1)}^1)\leqslant K_2$ and $\len_{\rho_{n-1}}(\eta_{n-1,(2)}^1)\leqslant K_2\nu$.
\end{itemize}

\textbf{Step 3. Inductive procedure.}
Suppose there exists $1\leqslant i\leqslant n-1$ such that $\eta_{n-i}^i=\bigcup_{\ell=0}^i \xi_{n-i}^\ell=\bigcup_{k=0}^{i+1}\eta_{n-i,(k)}^i$ with $B_{\delta}(\eta_{n-i}^i)\subset\MD_{n-i}'$ has the following properties:
\begin{itemize}
\item $\eta_{n-i,(k)}^i=\eta_{n-i}^i([s_{n-i}^{(k)},s_{n-i}^{(k+1)}])$ for some $0=s_{n-i}^{(0)}<s_{n-i}^{(1)}<\cdots<s_{n-i}^{(i+1)}<s_{n-i}^{(i+2)}=1$, where $0\leqslant k\leqslant i+1$; and

\item $\len(\eta_{n-i,(0)}^i)\leqslant h_{n-i}-D_6'-1$ and $\len_{\rho_{n-i}}(\eta_{n-i,(k)}^i)\leqslant  K_2\nu^{k-1}$ for every $1\leqslant k\leqslant i+1$.
\end{itemize}

By a similar argument to \eqref{equ:len-1}, \eqref{equ:len-2} and \eqref{equ:len-3}, there exist $0=s_{n-i-1}^{(0)}<s_{n-i-1}^{(1)}<\cdots<s_{n-i-1}^{(i+2)}<s_{n-i-1}^{(i+3)}=1$ such that the continuous curve $\eta_{n-i-1}^{i+1}=\bigcup_{\ell=0}^{i+1} \xi_{n-i-1}^\ell=\bigcup_{k=0}^{i+2}\eta_{n-i-1,(k)}^{i+1}$  with $B_{\delta}(\eta_{n-i-1}^{i+1})\subset\MD_{n-i-1}'$ has the following properties:
\begin{itemize}
\item $\eta_{n-i-1,(k)}^{i+1}=\eta_{n-i-1}^{i+1}([s_{n-i-1}^{(k)},s_{n-i-1}^{(k+1)}])$, where $0\leqslant k\leqslant i+2$; and

\item $\len(\eta_{n-i-1,(0)}^{i+1})\leqslant h_{n-i-1}-D_6'-1$ and $\len_{\rho_{n-i-1}}(\eta_{n-i-1,(k)}^{i+1})\leqslant  K_2\nu^{k-1}$ for every $1\leqslant k\leqslant i+2$.
\end{itemize}

Inductively (as $i$ increases), there exist $0=s_0^{(0)}<s_0^{(1)}<\cdots<s_0^{(n+1)}<s_0^{(n+2)}=1$ such that the continuous curve
$\eta_0^n=\bigcup_{\ell=0}^n \xi_0^\ell=\bigcup_{k=0}^{n+1}\eta_{0,(k)}^n$  with $B_{\delta}(\eta_0^n)\subset\MD_0'$ has the following properties:
\begin{itemize}
\item $\eta_{0,(k)}^n=\eta_0^n([s_0^{(k)},s_0^{(k+1)}])$, where $0\leqslant k\leqslant n+1$; and

\item $\len(\eta_{0,(0)}^n)\leqslant h_0-D_6'-1$ and $\len_{\rho_0}(\eta_{0,(k)}^n)\leqslant  K_2\nu^{k-1}$ for every $1\leqslant k\leqslant n+1$.
\end{itemize}

\medskip
\textbf{Step 4. The conclusion.}
Since $B_{\delta}(\eta_0^n)\subset\MD_0'$, the Euclidean metric and the hyperbolic metric $\rho_0$ of $\MD_0'$ are comparable in a small neighborhood of $\eta_0^n$.
Hence there exists a constant $C>0$ depending only on $\delta$ such that
\begin{equation}
\sum_{k=1}^{n+1}\len(\eta_{0,(k)}^n)\leqslant  C\sum_{k=1}^{n+1}\len_{\rho_0}(\eta_{0,(k)}^n) \leqslant \frac{C K_2}{1-\nu}.
\end{equation}
Therefore, for all $n\geqslant 0$ we have
\begin{equation}
\len(\eta_0^n)=\sum_{i=0}^n\len(\xi_0^i)=\sum_{k=0}^{n+1}\len(\eta_{0,(k)}^n)\leqslant K:=h_0-D_6'-1+\frac{C K_2}{1-\nu}.
\end{equation}

By \eqref{equ:phi-r}, \eqref{equ-integral} and the similar estimates to \eqref{equ:len-1} and \eqref{equ:len-2} in the above inductive procedure, it follows that for any $n\geqslant 0$, there exists a sequence of non-negative numbers $\{y_i^{(n)}:0\leqslant i\leqslant n\}$ which is independent of the sequence $(t_n)_{n\in\N}$ such that
for any $0\leqslant i\leqslant n$, we have
\begin{equation}
\len (\xi_0^i)\leqslant y_i^{(n)} \text{\quad and\quad} \sum_{i=0}^n y_i^{(n)}\leqslant K.
\end{equation}
Then \eqref{equ-seq-seg} holds if we set $y_i:=\inf_{n\in\N}\big\{y_i^{(n)}\big\}$.

\medskip
The estimate \eqref{equ-Cauchy} implies that the sequence of continuous curves $(\widetilde{\gamma}_n(t))_{n\in\N}$ converges uniformly on $[0,1]$. Since $\gamma_0^n(t)=\widetilde{\gamma}_0^n(\tfrac{a_1-1}{a_1}t)$ for all $t\in[0,1]$ and $n\in\N$, it implies that $(\gamma_0^n(t))_{n\in\N}$ converges uniformly on $[0,1]$.
\end{proof}

\begin{rmk}
 If $\alpha$ is of bounded type, or if there exists a universal constant $C>0$ such that $\MB(\alpha_{n+1})\geqslant C/\alpha_n$ for all $n\in\N$, then the sequence $(\gamma_0^n(t))_{n\in\N}$ converges exponentially fast as $n\to\infty$.
\end{rmk}

\subsection{The Siegel disks are Jordan domains}\label{subsec-conv-to-bdy}

By Proposition \ref{prop-Cauchy-sequence}, the sequence of the continuous curves $(\gamma_0^n(t))_{n\geqslant 0}$ has a limit:
\begin{equation}
\gamma_0^\infty(t):=\lim_{n\to\infty}\gamma_0^n(t), \text{\quad where } t\in[0,1].
\end{equation}

\begin{prop}\label{prop-tend-to-bdy}
The limit $\Phi_0^{-1}(\gamma_0^\infty)$ is the boundary of the Siegel disk of $f_0$.
\end{prop}

\begin{proof}
For $\zeta_0\in\gamma_0^{n+1}$, there exists $\zeta_n\in\widetilde{\gamma}_n^1\subset\bigcup_{j_{n+1}=0}^{a_{n+1}}\chi_{n+1,j_{n+1}}(\widetilde{\gamma}_{n+1}^0)$ such that
\begin{equation}
\zeta_0=\chi_{1,j_1}\circ\cdots\circ\chi_{n,j_n}(\zeta_n)
\end{equation}
for some sequence $(j_1,\cdots,j_n)$, where $0\leqslant j_i\leqslant a_i$ and $1\leqslant i\leqslant n$. By Lemma \ref{lema-gamma-n-1}(d) and \eqref{equ-chi-choice}, we have
\begin{equation}\label{equ-zeta-n}
\Big|\im\zeta_n-\frac{1}{2\pi}\MB(\alpha_{n+1})-M\Big|\leqslant D_3+\frac{1}{2} \text{\quad and\quad} 1\leqslant \re\zeta_n\leqslant a_{n+1}+\kc_1+2.
\end{equation}

By Proposition \ref{prop-Cherahi-nest}(b), each Siegel disk $\Delta_n$ is compactly contained in the domain of definition of $f_n$. For each $n\in\N$, $\Phi_n^{-1}$ is defined in $\MD_n$ (see Lemma \ref{lema:Phi-inverse}). We denote
\begin{equation}
\Delta_n':=\{\zeta\in\MD_n:\Phi_n^{-1}(\zeta)\in\Delta_n\}.
\end{equation}
By the definition of $\MD_n$, we have $\Phi_n^{-1}(\Delta_n')=\Delta_n$ and $\Expo(\Delta_n')=\Delta_{n+1}$.
By Lemma \ref{lema-fixed-disk}, the inner radius of the Siegel disk of $f_{n+1}$ is $c_{n+1}e^{-\MB(\alpha_{n+1})}$, where $c_{n+1}\in[D_7^{-1},D_7]$ and $D_7>1$ is a universal constant. According to the definition of near-parabolic renormalization $f_{n+1}=\MMR f_n$, there exists a point $\zeta_n'\in\partial\Delta_n'\cap\Expo^{-1}(\partial\Delta_{n+1})$ such that
\begin{equation}\label{equ-y-widehat}
|\re(\zeta_n-\zeta_n')|\leqslant \frac{1}{2} \quad\text{and}\quad \im \zeta_n'=\frac{1}{2\pi}\MB(\alpha_{n+1})-\frac{1}{2\pi}\log\frac{27 c_{n+1}}{4}.
\end{equation}

Let $[\zeta_n,\zeta_n']$ be the closed segment connecting $\zeta_n$ with $\zeta_n'$. By \eqref{equ-Siegel-below}, we have $[\zeta_n,\zeta_n')\subset\Delta_n'$.
By Lemma \ref{lema:D-n}, Lemma \ref{lema:comp-inclu} and \eqref{equ:gamma-ti-n-1}, we have $B_{\delta}([\zeta_n,\zeta_n'])\subset\MD_n'$ for $\delta=\min\{\delta_0,1/4\}$. Combining \eqref{equ-zeta-n} and \eqref{equ-y-widehat}, there exists a constant $A'>0$ which is independent of $n$ so that $|\zeta_n'-\zeta_n|\leqslant A'$. According to Lemma \ref{lema:exp-conv}, there exist two constants $A>0$ and $0<\nu<1$ which are independent of $n$ such that
\begin{equation}
\len(\chi_{1,j_1}\circ\cdots\circ\chi_{n,j_n}([\zeta_n,\zeta_n']))\leqslant A\cdot\nu^n,
\end{equation}
where $0\leqslant j_i\leqslant a_i$ and $1\leqslant i\leqslant n$. Denote $\zeta_0':=\chi_{1,j_1}\circ\cdots\circ\chi_{n,j_n}(\zeta_n')$. Then $|\zeta_0-\zeta_0'|\leqslant A\cdot\nu^n$. Since $\zeta_0'\in\partial\Delta_0'$, it implies that
\begin{equation}\label{equ:zeta-0}
\dist(\zeta_0,\partial\Delta_0')\leqslant A\cdot\nu^n .
\end{equation}
For any $t_0\in[0,1]$ and $n\geqslant 1$, we choose $\zeta_0=\zeta_0^{(n)}:=\gamma_0^{n+1}(t_0)$. By \eqref{equ:zeta-0} we have $\gamma_0^\infty(t_0)\in\partial\Delta_0'$. By the arbitrariness of $t_0\in[0,1]$, it follows that $\gamma_0^\infty\subset\partial\Delta_0'$. Therefore we have $\Phi_0^{-1}(\gamma_0^\infty)\subset\partial\Delta_0$.

By Lemma \ref{lema-gamma-n-i}(c), $\Phi_0^{-1}(\gamma_0^n)$ is a continuous closed curve for all $n\geqslant 0$.
Since $\gamma_0^n(t)$ converges uniformly to the limit $\gamma_0^\infty(t)$ on $[0,1]$ as $n\to\infty$, it follows that $\Phi_0^{-1}(\gamma_0^\infty)$ is a continuous closed curve which separates $\Delta_0$ from each component of $U_0\setminus\overline{\Delta}_0$, where $U_0$ is the domain of definition of $f_0$. In particular, we have $\Phi_0^{-1}(\gamma_0^\infty)=\partial\Delta_0$.
\end{proof}

\begin{proof}[{Proof of the the first part of the Main Theorem}]
Suppose $f_0\in\IS_\alpha\cup\{Q_\alpha\}$, where $\alpha\in\MB_N$ with $N$ sufficiently large. By Proposition \ref{prop-tend-to-bdy}, the boundary of the Siegel disk $\partial\Delta_0=\Phi_0^{-1}(\gamma_0^\infty)$ of $f_0$ is connected and locally connected. On the other hand, the Siegel disk $\Delta_0$ is compactly contained in the domain of definition of $f_0$ by Proposition \ref{prop-Cherahi-nest}(b). By the definition of $\Delta_0$, there exists a conformal map $\phi:\D\to\Delta_0$ so that $f_0\circ\phi(w)=\phi(e^{2\pi\ii\alpha}w)$. According to Carath\'{e}odory, the map $\phi$ can be extended continuously to $\phi:\OD\to\overline{\Delta}_0$.

For each $\theta\in[0,2\pi)$, let $\gamma_\theta:=\{\phi(r e^{\ii\theta}):0\leqslant r\leqslant 1\}$ be the internal ray of $\Delta_0$. Suppose there are two different rays $\gamma_{\theta_1}$ and $\gamma_{\theta_2}$ landing at a common point on $\partial\Delta_0$, i.e., $\phi(e^{\ii\theta_1})=\phi(e^{\ii\theta_2})$. Then $\gamma_{\theta_1}\cup \gamma_{\theta_2}$ is a Jordan curve contained in $\overline{\Delta}_0$. By the maximum modulus principle, $\{f_0^{\circ n}\}_{n\in\N}$ forms a normal family in the bounded domain $D_{\theta_1,\theta_2}$ which is bounded by $\gamma_{\theta_1}\cup \gamma_{\theta_2}$. This implies that $D_{\theta_1,\theta_2}$ is contained in the Fatou set and hence contained in $\Delta_0$. However, by Riesz brothers' theorem, $\phi$ must be a constant. This is a contradiction and each point in $\partial\Delta_0$ is the landing point of exactly one internal ray. Hence $\partial\Delta_0$ is a Jordan curve.
\end{proof}

\section{A Jordan arc and a new class of irrationals}\label{sec-canonical-Jordan}

In this section, we first define a Jordan arc $\Gamma$ connecting the origin with the critical value $\cv=-4/27$ in the domain of definition of $f\in\IS_\alpha\cup\{Q_\alpha\}$ with $\alpha\in\HT_N$. In particular, this arc is contained in $\MP_f$. Then we define a new class of irrational numbers based on the mapping relations between the different levels of the renormalization.

\subsection{A Jordan arc corresponding to $\alpha\in \HT_N$}\label{subsec-Jordan-arc}

Let $f\in\IS_\alpha\cup\{Q_\alpha\}$ with $\alpha\in\HT_N$, where $N\geqslant 1/\varepsilon_4$ is assumed in \S\ref{subsec-basic-defi}. We define a half-infinite strip
\begin{equation}\label{equ-E}
\mho:=\{\zeta\in\C:1/4<\re\zeta< 7/4 \text{ and } \im\zeta>-2\}
\end{equation}
and a topological triangle 
\begin{equation}
\MQ_f:=\{z\in\MP_f:\Phi_f(z)\in \mho\}.
\end{equation}

\begin{lem}\label{lema-height}
There exists $\varepsilon_4'\in(0,\varepsilon_4]$ such that for all $f\in\IS_\alpha$ with $\alpha\in(0,\varepsilon_4']$,
\begin{equation}\label{equ-M-c-subset}
\overline{\MQ}_f\setminus\{0\} \subset \D(0,\tfrac{4}{27}e^{3\pi})\setminus [0,\tfrac{4}{27}e^{3\pi}).
\end{equation}
\end{lem}


We postpone the proof of Lemma \ref{lema-height} to Appendix \ref{sec-arc-straight}. The inclusion relation \eqref{equ-M-c-subset} is proved for the maps in $\IS_0$ first and then a continuity argument is used.

\medskip
For $f_0\in\IS_\alpha\cup\{Q_\alpha\}$ with $\alpha\in\HT_N$, let $f_n:=\MMR f_{n-1}$ be the maps defined by the renormalization operator inductively, where $n\geqslant 1$. In the following, we always assume that $N\geqslant 1/\varepsilon_4'$ and denote $\MQ_n:=\MQ_{f_n}$. For $X\subset\C$ and $\delta>0$, we denote $B_\delta(X):=\bigcup_{z\in X}\D(z,\delta)$.

\begin{cor}\label{cor-anti-holo}
For each $n\geqslant 1$, there exists a unique anti-holomorphic inverse branch of the modified exponential map $\Expo$:
\begin{equation}
\Log:\MQ_n\to\Phi_{n-1}(\MQ_{n-1})=\mho,
\end{equation}
such that $\Log(-\tfrac{4}{27})=1$. Moreover, $B_{1/4}(\Log(\overline{\MQ}_n\setminus\{0\}))\subset\mho$ and $\Phi_{n-1}^{-1}\circ\Log:\overline{\MQ}_n\setminus\{0\}\to\MQ_{n-1}$ is well defined.
\end{cor}

\begin{proof}
Since $\Expo$ takes the value $-4/27$ at each integer, it follows that $\Expo$ has an inverse branch $\Log$ defined on $\overline{\MQ}_n\setminus\{0\}$ such that $\Log(-4/27)=1$ since $\overline{\MQ}_n\setminus\{0\}$ is simply connected and avoids the origin. By Lemma \ref{lema-height}, we have $\re\Log(\overline{\MQ}_n\setminus\{0\})\subset (1/2,3/2)$ and $\im\Log(\overline{\MQ}_n\setminus\{0\})> -3/2$. Therefore, $B_{1/4}(\Log(\overline{\MQ}_n\setminus\{0\}))$ is contained in $\mho$ and $\Phi_{n-1}^{-1}\circ\Log:\overline{\MQ}_n\setminus\{0\}\to\MQ_{n-1}$ is well defined.
\end{proof}

Define a half-infinite strip
\begin{equation}\label{equ-mho-new}
\mho':=\{\zeta\in\C:1/2<\re\zeta< 3/2 \text{ and } \im\zeta>-7/4\}\subset\mho
\end{equation}
and a topological triangle for every $n\geqslant 0$:
\begin{equation}
\MQ_n':=\{z\in\MP_n:\Phi_n(z)\in \mho'\}.
\end{equation}

\begin{defi}[see Figure \ref{Fig-K-n-Gamma}]
Let $K_0:=\MQ_0'$. For each $n\geqslant 1$, define
\begin{equation}
K_n:=\Phi_0^{-1}\circ\Log\circ\cdots\circ\Phi_{n-1}^{-1}\circ\Log(\MQ_n').
\end{equation}
By Corollary \ref{cor-anti-holo}, $K_{n+1}\subset K_n$ for all $n\geqslant 0$, the critical value $\cv=-4/27$ is contained in the interior of $K_n$ and $0\in\partial K_n$. Define
\begin{equation}\label{equ-Gamma}
\Gamma:=\bigcap_{n\geqslant 0}K_n.
\end{equation}
\end{defi}

\begin{figure}[!htpb]
  \setlength{\unitlength}{1mm}
  \centering
  \includegraphics[width=0.6\textwidth]{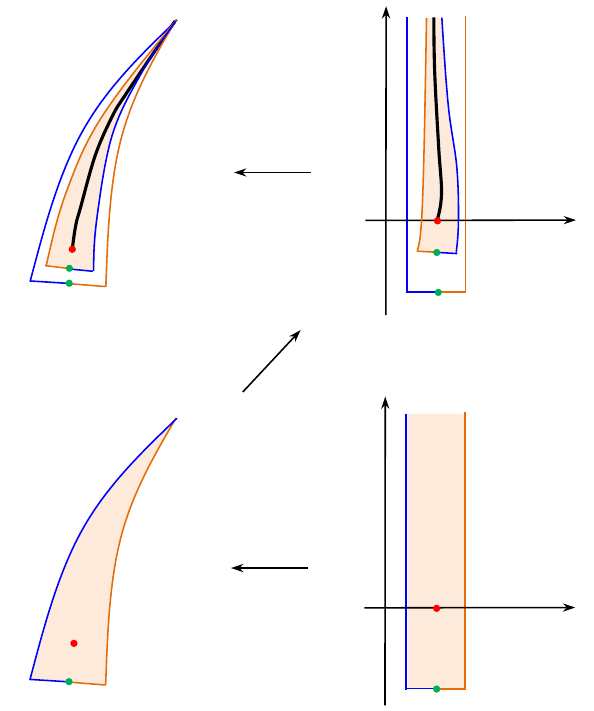}
  \put(-21,9.5){$1$}
  \put(-26.2,9.5){$\frac{1}{2}$}
  \put(-16,9.5){$\frac{3}{2}$}
  \put(-29,10){$0$}
  \put(-32,2.5){$-\frac{7}{4}$}
  \put(-15.5,25){$\mho'$}
  \put(-43.5,21){$\Phi_1^{-1}$}
  \put(-58,25){$\MQ_1'$}
  \put(-68,6){$\cv$}
  \put(-53,34.5){$0$}
   \put(-21,59.5){$1$}
  \put(-26,59){$\frac{1}{2}$}
  \put(-16,59){$\frac{3}{2}$}
  \put(-29,59.5){$0$}
  \put(-32,52.5){$-\frac{7}{4}$}
  \put(-15.5,74.5){$\mho'$}
  \put(-43.5,70.7){$\Phi_0^{-1}$}
  \put(-61,59){$K_0=\MQ_0'$}
  \put(-69,62){$K_1$}
  \put(-68.3,57){$\cv$}
  \put(-53,85){$0$}
   \put(-67,68){$\Gamma$}
   \put(-39,42){$\Log$}
  \caption{A sketch of the renormalization microscope between levels $0$ and $1$. The sets $\Gamma$, $\mho'$, $\MQ_n'$, $K_n$ with $n=0,1$ and some special points are marked.}
  \label{Fig-K-n-Gamma}
\end{figure}

\begin{lem}\label{lem-Jordan-arc}
The set $\Gamma\cup\{0\}$ is a Jordan arc connecting $\cv=-4/27$ with $0$.
\end{lem}

\begin{proof}
The general idea of the proof is to use the uniform contraction with respect to the hyperbolic metrics to prove that $\Gamma\cup\{0\}$ is locally connected and then prove that it must be a Jordan arc. Let us prove it in details.

\medskip
\textbf{Step 1}: We first define two continuous curves $\gamma_{0,\pm}^0:[0,+\infty)\to\mho$ as
\begin{equation}
\gamma_{0,\pm}^0(t):=
\left\{
\begin{aligned}
& 1\pm\tfrac{1}{2}+(t-\tfrac{11}{4})\,\ii  ~~~&\text{if }& t\in[1,+\infty),\\
& 1\pm\tfrac{t}{2}-\tfrac{7}{4}\,\ii  ~~~&\text{if }& t\in[0,1).
\end{aligned}
\right.
\end{equation}
Then $\gamma_{0,+}^0$ and $\gamma_{0,-}^0$ have the same initial point $\gamma_{0,\pm}^0(0)=1-\tfrac{7}{4}\ii$ and $\gamma_{0,+}^0\cup \gamma_{0,-}^0=\partial\mho'$, where $\mho'$ is defined in \eqref{equ-mho-new}.
For $\alpha\in(0,1)$, we define
\begin{equation}\label{equ-varphi}
\varphi_\alpha(t):=
\left\{
\begin{aligned}
& \frac{1}{\alpha}\Big(t-\frac{1}{2\pi}\log\frac{1}{\alpha}+1\Big) ~~~&\text{if }& t\geqslant\frac{1}{2\pi}\log\frac{1}{\alpha},\\
& e^{2\pi t}  ~~~&\text{if }& t< \frac{1}{2\pi}\log\frac{1}{\alpha}.
\end{aligned}
\right.
\end{equation}
It is easy to see that $\varphi_\alpha$ is continuous on $\R$ and strictly increasing. For $n\geqslant 1$, we define $\varphi_n:=\varphi_{\alpha_n}$. Then $\varphi_n\circ\cdots\circ\varphi_1(t)\to +\infty$ as $n\to\infty$ for all $t\in\R$.

In the following, we define two sequences of continuous curves $(\gamma_{n,\pm}^0)_{n\geqslant 0}$ inductively. For $n\geqslant 1$, suppose $\gamma_{n-1,\pm}^0:[0,+\infty)\to\partial\mho'$ has been defined. We define $\gamma_{n,\pm}^0:[0,+\infty)\to\partial\mho'$ as
\begin{equation}\label{equ-gamma-n-plus}
\gamma_{n,\pm}^0(t):=
\left\{
\begin{aligned}
& 1\pm\tfrac{1}{2}+\big(\varphi_n(\im \gamma_{n-1,\pm}^0(t))-e^{-7\pi/2}-\tfrac{7}{4}\big)\,\ii  ~~~&\text{if }& t\in[1,+\infty),\\
& 1\pm\tfrac{t}{2}-\tfrac{7}{4}\,\ii  ~~~&\text{if }& t\in[0,1).
\end{aligned}
\right.
\end{equation}
Note that $\gamma_{n,+}^0(1)=\tfrac{3}{2}-\tfrac{7}{4}\,\ii$ and $\gamma_{n,-}^0(1)=\tfrac{1}{2}-\tfrac{7}{4}\,\ii$.
Then both $\gamma_{n,+}^0:[0,+\infty)\to\partial\mho'$ and $\gamma_{n,-}^0:[0,+\infty)\to\partial\mho'$ are continuous injections and they have the same initial point $\gamma_{n,\pm}^0(0)=1-\tfrac{7}{4}\ii$. Moreover, $\gamma_{n,+}^0\cup \gamma_{n,-}^0=\partial\mho'$.

\medskip
For $t\in[0,+\infty)$, all $n\geqslant 1$ and $1\leqslant i\leqslant n$, by Corollary \ref{cor-anti-holo} the following curves are well-defined:
\begin{equation}\label{equ-gam-anti}
\gamma_{n-i,\pm}^i(t):=
\left\{
\begin{aligned}
& \Log\circ\Phi_{n-i+1}^{-1}\circ\cdots\circ\Log\circ\Phi_n^{-1}(\gamma_{n,\pm}^0(t))  & ~~~\text{if }& i \text{ is even},\\
& \Log\circ\Phi_{n-i+1}^{-1}\circ\cdots\circ\Log\circ\Phi_n^{-1}(\gamma_{n,\mp}^0(t))  & ~~~\text{if }& i \text{ is odd}.
\end{aligned}
\right.
\end{equation}
In particular, $\gamma_{n-i,\pm}^i\subset \overline{\mho'}$ for every $0\leqslant i\leqslant n$. Define
\begin{equation}
\Gamma_{n-i,\pm}^i(t):=\Phi_{n-i}^{-1}(\gamma_{n-i,\pm}^i(t)), \text{ where } t\in[0,+\infty).
\end{equation}
Then $\Gamma_{n-i,+}^i\cup\{0\}$ and $\Gamma_{n-i,-}^i\cup\{0\}$ are Jordan arcs, and $\Gamma_{n-i,+}^i\cup\Gamma_{n-i,-}^i\cup\{0\}$ is a Jordan curve\footnote{As before we use the fact that $\lim_{\im\zeta\to+\infty}\Phi_{n-i}^{-1}(\zeta)=0$, where $\zeta\in\Phi_{n-i}(\MP_{n-i})$.}. In particular, we have $\Gamma_{0,+}^n\cup\Gamma_{0,-}^n\cup\{0\}=\partial K_n$ and two sequences of continuous curves $\gamma_{0,\pm}^n:[0,+\infty)\to\overline{\mho'}$, where $n\in\N$. In the following we prove that $\gamma_{0,\pm}^n(t)$ and $\Gamma_{0,\pm}^n(t)$ converge uniformly on $[0,+\infty)$ as $n\to\infty$.

\medskip
\textbf{Step 2}: We first estimate the distance between $\gamma_{n-1,\pm}^0(t)$ and $\gamma_{n-1,\pm}^1(t)$ for all $n\geqslant 1$ and $t\in[0,+\infty)$. Let $t_n\in(1,+\infty)$ be the unique parameter such that
\begin{equation}
\im\gamma_{n,\pm}^0(t_n)=\varphi_n(\im \gamma_{n-1,\pm}^0(t_n))-e^{-7\pi/2}-\tfrac{7}{4}=\tfrac{1}{\alpha_n}.
\end{equation}
Then we have $\tfrac{1}{2\pi}\log\tfrac{1}{\alpha_n}<\im \gamma_{n-1,\pm}^0(t_n)<\tfrac{1}{2\pi}\log\tfrac{1}{\alpha_n}+2\alpha_n$. By definition, we have
\begin{equation}
\begin{split}
|\gamma_{n-1,\pm}^0(t)-\gamma_{n-1,\pm}^1(t)|
=&~|\gamma_{n-1,\pm}^0(t)-\Log\circ\Phi_n^{-1}(\gamma_{n,\mp}^0(t))|\\
\leqslant &~1+|\im\gamma_{n-1,\pm}^0(t)-\im\Log\circ\Phi_n^{-1}(\gamma_{n,\mp}^0(t))|.
\end{split}
\end{equation}
If $t\geqslant t_n$, then $\im\gamma_{n,\pm}^0(t)\geqslant\tfrac{1}{\alpha_n}$. By \eqref{equ-varphi}, \eqref{equ-gamma-n-plus} and Lemma \ref{lema-key-estimate-lp}(a), we have
\begin{equation}
\begin{split}
&~|\im\gamma_{n-1,\pm}^0(t)-\im\Log\circ\Phi_n^{-1}(\gamma_{n,\mp}^0(t))|\\
\leqslant &~D_3+\big|\im\gamma_{n-1,\pm}^0(t)-\alpha_n\im\gamma_{n,\mp}^0(t)-\tfrac{1}{2\pi}\log\tfrac{1}{\alpha_n}\big|\\
\leqslant &~D_3+1+\alpha_n(e^{-7\pi/2}+\tfrac{7}{4})<D_3+2.
\end{split}
\end{equation}
If $t< t_n$, then $\im\gamma_{n,\pm}^0(t)<\tfrac{1}{\alpha_n}$. By \eqref{equ-varphi}, \eqref{equ-gamma-n-plus} and Lemma \ref{lema-key-estimate-lp}(b), there exist two universal constants $C_1$, $C_2\geqslant 1$ such that
\begin{equation}
\begin{split}
&~|\im\gamma_{n-1,\pm}^0(t)-\im\Log\circ\Phi_n^{-1}(\gamma_{n,\mp}^0(t))|\\
\leqslant &~D_3+\big|\im\gamma_{n-1,\pm}^0(t)-\tfrac{1}{2\pi}\log(1+|\gamma_{n,\mp}^0(t)|)\big|\\
\leqslant &~D_3+C_1+\big|\im\gamma_{n-1,\pm}^0(t)-\tfrac{1}{2\pi}\log(1+|\im\gamma_{n,\mp}^0(t)|)\big|
\leqslant D_3+C_1+C_2.
\end{split}
\end{equation}
Therefore, for all $n\geqslant 1$ and $t\in[0,+\infty)$, we have
\begin{equation}\label{equ:gamma-n-1}
|\gamma_{n-1,\pm}^0(t)-\gamma_{n-1,\pm}^1(t)|\leqslant D_3+C_1+C_2+1.
\end{equation}

\textbf{Step 3}:
Let $\rho_\mho(\zeta)|\dd \zeta|$ and $\rho_n(z)|\dd z|$ be the hyperbolic metrics of $\mho$ and $\MQ_n$ respectively.
Note that $\gamma_{n-1,\pm}^0$, $\gamma_{n-1,\pm}^1\subset\overline{\mho'}$ and $B_{1/4}(\overline{\mho'})\subset\mho$. By \eqref{equ:gamma-n-1}, there exists $C_3>0$ such that the hyperbolic distance between $\gamma_{n-1,\pm}^0$ and $\gamma_{n-1,\pm}^1$ satisfies
\begin{equation}\label{equ:dist-rho}
\dist_{\rho_\mho}(\gamma_{n-1,\pm}^0(t),\gamma_{n-1,\pm}^1(t))\leqslant C_3 \text{\quad for any } n\geqslant 1 \text{ and } t\in[0,+\infty).
\end{equation}
According to Corollary \ref{cor-anti-holo}, for $1\leqslant i\leqslant n$, each map $\Log\circ\Phi_i^{-1}:(\mho,\rho_\mho)\to (\mho,\rho_\mho)$ can be decomposed as:
\begin{equation}
\begin{split}
\Log\circ\Phi_i^{-1}:(\mho,\rho_\mho)
&~\xlongrightarrow{\Phi_i^{-1}} (\MQ_i,\rho_i) \xlongrightarrow{\Log} (\Log(\MQ_i),\tilde{\rho}_i) \\
&~\overset{inc.}{\hookrightarrow} (B_{1/4}(\Log(\MQ_i)),\hat{\rho}_i) \overset{inc.}{\hookrightarrow} (\mho,\rho_\mho),
\end{split}
\end{equation}
where $\tilde{\rho}_i$ and $\hat{\rho}_i$ are hyperbolic metrics of $\Log(\MQ_i)$ and $B_{1/4}(\Log(\MQ_i))$ respectively.
Since $\diam(\re(\Log(\MQ_i)))\leqslant 1$, by Lemma \ref{lema-uni-con-prep}, the inclusion map
\begin{equation}
(\Log(\MQ_i),\tilde{\rho}_i) \overset{inc.}{\hookrightarrow} (B_{1/4}(\Log(\MQ_i)),\hat{\rho}_i)
\end{equation}
is uniformly contracting with respect to their hyperbolic metrics. Since $\Phi_i^{-1}$, $\Log$ and the second inclusion map do not expand the hyperbolic metrics, it follows that $\Log\circ\Phi_i^{-1}:(\mho,\rho_\mho)\to (\mho,\rho_\mho)$ is uniformly contracting.

\medskip
By the definition of $\gamma_{0,\pm}^n$, there exists a constant $0<\nu<1$ such that
\begin{equation}
\dist_{\rho_\mho}(\gamma_{0,\pm}^{n-1}(t),\gamma_{0,\pm}^n(t))\leqslant C_3\cdot\nu^{n-1}, \text{ where }n\geqslant 1 \text{ and }t\in[0,+\infty).
\end{equation}
This implies that the hyperbolic distance between $\Gamma_{0,\pm}^{n-1}(t)$ and $\Gamma_{0,\pm}^n(t)$ in $\MQ_0=\Phi_0^{-1}(\mho)$ satisfies
\begin{equation}
\dist_{\rho_0}(\Gamma_{0,\pm}^{n-1}(t),\Gamma_{0,\pm}^n(t))\leqslant C_3\cdot\nu^{n-1}, \text{ where }n\geqslant 1 \text{ and }t\in[0,+\infty).
\end{equation}
Let $\check{\MQ}_0:=B_1(\MQ_0)$ and $\check{\rho}_0(z)|\dd z|$ be the hyperbolic metric of $\check{\MQ}_0$. Then the Euclidean and hyperbolic metrics (with respect to $\check{\rho}_0$) are comparable on $\MQ_0$. According to Schwarz-Pick's lemma, we have $\check{\rho}_0(z)<\rho_0(z)$ for all $z\in\MQ_0$. Therefore, there exists a constant $C_4>0$ such that the distance in the Euclidean metric satisfies
\begin{equation}
|\Gamma_{0,\pm}^{n-1}(t)-\Gamma_{0,\pm}^n(t)|\leqslant C_4\cdot\nu^{n-1}, \text{ where }n\geqslant 1 \text{ and }t\in[0,+\infty).
\end{equation}
Therefore, the following convergence is uniform for $t\in [0,+\infty)$:
\begin{equation}
\Gamma_{0,\pm}^\infty(t):=\lim_{n\to\infty} \Gamma_{0,\pm}^n(t).
\end{equation}
Note that $1\in\mho$ and $\Log\circ\Phi_n^{-1}(1)=1$. By the uniformly contracting of $\Log\circ\Phi_i^{-1}:(\mho,\rho_\mho)\to (\mho,\rho_\mho)$ for all $1\leqslant i\leqslant n$, we have
\begin{equation}
\begin{split}
\lim_{n\to\infty}\Gamma_{0,\pm}^n(0)
=&~\lim_{n\to\infty}\Phi_0^{-1}\circ\Log\circ\Phi_1^{-1}\circ\cdots\circ\Log\circ\Phi_n^{-1}(1-\tfrac{7}{4}\ii) \\
=&~\Phi_0^{-1}(1)=-\tfrac{4}{27}.
\end{split}
\end{equation}

Since $\gamma_{n-1,\pm}^0 \subset\overline{\mho'}$ and $B_{1/4}(\overline{\mho'})\subset\mho$, there exists a constant $C_3'>0$ such that
\begin{equation}
\dist_{\rho_\mho}(\gamma_{n-1,+}^0(t),\gamma_{n-1,-}^0(t))\leqslant C_3' \text{\quad for any } n\geqslant 1 \text{ and } t\in[0,+\infty).
\end{equation}
By a similar argument as above, we have
\begin{equation}
\Gamma_{0,+}^\infty(t)=\Gamma_{0,-}^\infty(t), \text{\quad where } t\in [0,+\infty).
\end{equation}
Note that $\Gamma$ is the intersection of the nested sequence $(K_n)_{n\geqslant 0}$, where $K_n$ is the bounded component of $\C\setminus(\Gamma_{0,+}^n\cup\Gamma_{0,-}^n\cup\{0\})$ for all $n\geqslant 0$. Therefore, $\Gamma=\Gamma_{0,+}^\infty=\Gamma_{0,-}^\infty$ and $\Gamma\cup\{0\}$ is a Jordan arc connecting $-4/27$ with $0$.
\end{proof}

\subsection{Dynamical behavior of the points on the arcs}

Let $\phi_0:=\textup{id}$. For each $n\geqslant 1$, we denote
\begin{equation}
\phi_n:=\Expo\circ\Phi_{n-1}\circ\cdots\circ\Expo\circ\Phi_0.
\end{equation}
Let $\Gamma$ be the Jordan arc defined in \eqref{equ-Gamma}. By the proof of Lemma \ref{lem-Jordan-arc}, $\phi_n$ can be defined on $\Gamma_0:=\Gamma$ since
\begin{equation}
\Gamma_n:=\phi_n(\Gamma_0)\subset \MQ_n'=\Phi_n^{-1}(\mho'), \text{ where }n\geqslant 1.
\end{equation}
Note that the restriction of $\Expo\circ\Phi_{n-1}$ on $\Gamma_{n-1}$ is a homeomorphism. Hence each $\Gamma_n\cup\{0\}$ is also a Jordan arc connecting $-\tfrac{4}{27}$ with $0$ in the dynamical plane of $f_n$. For each $n\geqslant 1$, the map $\phi_n:\Gamma_0\to \Gamma_n$ can be extended homeomorphically to $\phi_n:\Gamma_0\cup\{0\}\to \Gamma_n\cup\{0\}$ such that $\phi_n(-\tfrac{4}{27})=-\tfrac{4}{27}$ and $\phi_n(0)=0$.
Moreover,
\begin{equation}\label{equ-gamma-n}
\gamma_n:=\Phi_n(\Gamma_n)
\end{equation}
is an unbounded arc in $\mho'$ with the initial point $1$.

\begin{defi}
For $n\geqslant 1$, we define
\begin{equation}\label{equ-s-alpha-n}
s_{\alpha_n}:=\Phi_n\circ \Expo:\gamma_{n-1}\to\gamma_n.
\end{equation}
Then $s_{\alpha_n}$ is a homeomorphism with $s_{\alpha_n}(1)=1$.
\end{defi}

In the following, we assume that $\alpha=\alpha_0\in\MB_N$, where $\MB_N$ is the set of high type Brjuno numbers defined in \eqref{equ:Brjuno}.
Let $\MB(\alpha_n)$ be the Brjuno sum defined in \eqref{equ:Brj-Yoccoz-n}.

\begin{defi}
For $n\geqslant 0$, we define
\begin{equation}\label{equ-height-new}
\widetilde{\MB}(\alpha_n):=\frac{\MB(\alpha_n)}{2\pi}+ M,
\end{equation}
where $M\geqslant 1$ is a constant which will be determined in a moment.
\end{defi}

\begin{lem}\label{lema-all-greater}
There exists a constant $M_0>1$ such that if $M\geqslant M_0$, for $\zeta\in\gamma_{n-1}$ with $\im\zeta\geqslant\widetilde{\MB}(\alpha_n)$, then $\im\,s_{\alpha_n}(\zeta)\geqslant\widetilde{\MB}(\alpha_{n+1})$, where $n\geqslant 1$.
\end{lem}

\begin{proof}
Let $D_4>0$ be the constant introduced in Lemma \ref{lema-key-esti-inverse}. If $M\geqslant D_4$, then
\begin{equation}
\widetilde{\MB}(\alpha_n)=\frac{\MB(\alpha_n)}{2\pi}+ M>\frac{1}{2\pi}\log\frac{1}{\alpha_n}+D_4.
\end{equation}
By Lemma \ref{lema-key-esti-inverse}(a), if $M\geqslant 2D_5$ and $\im\zeta\geqslant\widetilde{\MB}(\alpha_n)$, then
\begin{equation}
\begin{split}
\im\,s_{\alpha_n}(\zeta)
&~\geqslant \frac{1}{\alpha_n}\left(\im\zeta-\frac{1}{2\pi}\log\frac{1}{\alpha_n}-D_5\right)
\geqslant \frac{1}{\alpha_n}\left(\widetilde{\MB}(\alpha_n)-\frac{1}{2\pi}\log\frac{1}{\alpha_n}-D_5\right) \\
&~=\widetilde{\MB}(\alpha_{n+1})+\frac{1}{\alpha_n}\big((1-\alpha_n) M- D_5\big)\geqslant \widetilde{\MB}(\alpha_{n+1}).
\end{split}
\end{equation}
Then the lemma follows by setting $M_0:=\max\{D_4,2D_5\}$.
\end{proof}

Since $\alpha\in\MB_N$, every $f_0\in\IS_\alpha\cup\{Q_\alpha\}$ has a Siegel disk $\Delta_0$ centered at the origin.
Let $D_7> 1$ be the universal constant in Lemma \ref{lema-fixed-disk}. In the following we fix
\begin{equation}\label{equ:M}
M\geqslant \max\left\{M_0,\frac{1}{2\pi}\log\frac{27 D_7}{4} \right\}.
\end{equation}
Let $\Gamma_0\cup\{0\}$ be the Jordan arc connecting the critical value $\cv=-\frac{4}{27}$ with $0$ corresponding to $f_0$ (see Lemma \ref{lem-Jordan-arc}). For a given point $z_0\in\Gamma_0$, let $(\zeta_n)_{n\geqslant 0}$ be the sequence defined by
\begin{equation}\label{equ-sequence}
\zeta_0:=\Phi_0(z_0)\in\gamma_0 \text{\quad and\quad} \zeta_n:=s_{\alpha_n}(\zeta_{n-1})\in\gamma_n \text{\quad for } n\geqslant 1.
\end{equation}

\begin{lem}\label{lema-eventually-above}
If $z_0\in\Gamma_0\cap\Delta_0$, then there exists $n_0\geqslant 0$ such that $\im \zeta_n\geqslant \widetilde{\MB}(\alpha_{n+1})$ for all $n\geqslant n_0$.
\end{lem}

\begin{proof}
Let $z_0\in\Gamma_0\cap\Delta_0$. By Lemma \ref{lema-fixed-disk}, for every $n\in\N$, the inner radius of the Siegel disk of $f_n$ is $c_ne^{-\MB(\alpha_n)}$, where $c_n\in[1/D_7,D_7]$. Let $\mho$ be the half-infinite strip defined in \eqref{equ-E}. By the definition of near-parabolic renormalization $f_{n+1}=\MMR f_n$, there exists $\widetilde{\zeta}_n\in \overline{\mho'}$ such that $\Expo(\widetilde{\zeta}_n)\in\partial\Delta_{n+1}$ and (see \eqref{equ-y-widehat})
\begin{equation}\label{equ-y-n-hat-2}
\im \widetilde{\zeta}_n=\frac{1}{2\pi}\MB(\alpha_{n+1})-\frac{1}{2\pi}\log\frac{27 c_{n+1}}{4}.
\end{equation}

Assume there exists a subsequence $(n_j)_{j\geqslant 1}$ such that $\im \zeta_{n_j}<  \widetilde{\MB}(\alpha_{n_j+1})=\frac{1}{2\pi}\MB(\alpha_{n_j+1})+M$.
If $\im\zeta_{n_j}\leqslant \im \widetilde{\zeta}_{n_j}$, there exists $\zeta'_{n_j}\in\Phi_{n_j}(\partial\Delta_{n_j}\cap\MP_{n_j})\cap \overline{\mho'}$ with
$\im\zeta_{n_j}'=\im\zeta_{n_j}$ such that
\begin{equation}\label{equ:Euc-1}
|\zeta_{n_j}-\zeta_{n_j}'|\leqslant 1.
\end{equation}
If $\im\zeta_{n_j}> \im \widetilde{\zeta}_{n_j}$, we have
\begin{equation}
\frac{1}{2\pi}\MB(\alpha_{n_j+1})-\frac{1}{2\pi}\log\frac{27 c_{n_j+1}}{4}<\im \zeta_{n_j}<  \frac{1}{2\pi}\MB(\alpha_{n_j+1})+M
\end{equation}
and hence
\begin{equation}\label{equ:Euc-2}
|\zeta_{n_j}-\widetilde{\zeta}_{n_j}|^2 \leqslant 1+\Big(M+\frac{1}{2\pi}\log\frac{27 D_7}{4}\Big)^2.
\end{equation}

By \eqref{equ:Euc-1} and \eqref{equ:Euc-2}, for each $\zeta_{n_j}$ with $j\geqslant 1$, one can find a point ($\zeta_{n_j}'$ or $\widetilde{\zeta}_{n_j}$) in $\Phi_{n_j}(\partial\Delta_{n_j}\cap\MP_{n_j})\cap \overline{\mho'}$ such that the hyperbolic distance with respect to $\rho_{\mho}$ between them are uniformly bounded above. By a similar argument to Proposition \ref{prop-tend-to-bdy} based on Lemma \ref{lema-uni-con-prep}, we conclude that $\zeta_0\in \Phi_0(\partial\Delta_0\cap\MP_0)\cap \overline{\mho'}$ and $z_0\in\partial\Delta_0$, which violates our assumption that $z_0\in\Delta_0$. Therefore, there exists $n_0\geqslant 0$ such that $\im \zeta_n\geqslant \widetilde{\MB}(\alpha_{n+1})$ for all $n\geqslant n_0$.
\end{proof}

\begin{lem}\label{lema-in-disk}
$\Gamma_0\cap\partial\Delta_0$ is a singleton. In particular, $\Gamma_0\setminus\{\cv\}\subset\Delta_0$ if and only if $\cv\in\partial\Delta_0$.
\end{lem}

\begin{proof}
Since $\Gamma_0\cup\{0\}$ is a Jodan arc connecting $\cv=-\frac{4}{27}$ with $0$, there exists a homeomorphism $\beta:[0,1]\to\Gamma_0\cup\{0\}$ such that $\beta(0)=\cv$ and $\beta(1)=0$. Assume that $\Gamma_0\cap\partial\Delta_0$ is not a singleton. Then there exist $0\leqslant t_1<t_2<1$ such that
\begin{itemize}
\item $\beta(t_i)\in\partial\Delta_0$ for $i=1,2$; and
\item $\beta([0,t_1])\cap\Delta_0=\emptyset$ and $\beta((t_2,1])\subset\Delta_0$.
\end{itemize}
Let $\Gamma_0':=\beta([t_1,t_2])$ be a subarc of $\Gamma_0$. Then we have the following two cases.

\medskip
(1) Assume $\Gamma_0'\subset\partial\Delta_0$. There exists $z_0\in\Gamma_0'$ such that $f_0^{\circ q_n}(z_0)\in\Gamma_0'$ for some big integer $n$ since the restriction of $f_0$ on $\partial\Delta_0$ is conjugate to the rigid rotation. Denote $\Gamma_n':=\Expo\circ\Phi_{n-1}\circ\cdots\circ\Expo\circ\Phi_0(\Gamma_0')$. Then $\Gamma_n'$ is a Jordan arc contained in $\Gamma_n\subset\MQ_n'$. By Lemma \ref{lema-Cheraghi-2}(a), $\Gamma_n'$ and hence $\Gamma_n$ contains a point $z_n$ and $f_n(z_n)$, which is impossible.

\medskip
(2) Assume $\Gamma_0'\not\subset\partial\Delta_0$. Since $\Phi_n(\Gamma_n)\subset\mho'$, it follows that $f_n(\Gamma_n)$ is well-defined and contained in $\MP_n$. Thus by Lemma \ref{lema-Cheraghi-2}, $\Gamma_0$ (and hence $\Gamma_0'$) can be iterated infinitely many times by $f_0$.
Let $W\neq\Delta_0$ be any bounded component of $\C\setminus(\partial\Delta_0\cup\Gamma_0')$. Since $\partial\Delta_0\cup\Gamma_0'$ and $W$ can be iterated infinitely many times by $f_0$, it follows from the maximum modulus principle that $W$ is contained in the Fatou set of $f_0$. Since $\partial W\cap\partial\Delta_0$ contains a subarc of $\partial\Delta_0$, it follows that $W$ is contained in $\Delta_0$, which is a contradiction.
This finishes the proof that $\Gamma_0\cap\partial\Delta_0$ is a singleton.

\medskip

From $\Gamma_0\setminus\{\cv\}\subset\Delta_0$ we obtain $\cv\in\partial\Delta_0$ immediately. If $\cv\in\partial\Delta_0$, since $\Gamma_0$ is a Jordan arc and $\Gamma_0\cap\partial\Delta_0$ is a singleton, we conclude that $\Gamma_0\setminus\{\cv\}\subset\Delta_0$.
\end{proof}

\subsection{A new class of irrational numbers}\label{subsec-new-class}

For $n\geqslant 1$, let $s_{\alpha_n}:\gamma_{n-1}\to\gamma_n$ be the homeomorphism defined in \eqref{equ-s-alpha-n}.
In the following, we use $\Gamma_\alpha$ (resp. $\gamma_\alpha$) to denote $\Gamma_0$ (resp. $\gamma_0=\Phi_0(\Gamma_0)$) when we want to emphasize the dependence on $\alpha=\alpha_0\in\HT_N$.

\begin{defi}
Let $\widetilde{\MH}_N$ be a subset of $\MB_N$ defined as
\begin{equation}
\widetilde{\MH}_N:=
\left\{\alpha\in\MB_N
\left|
\begin{array}{l}
\forall\,\zeta\in\gamma_\alpha\setminus\{1\}, \,\exists\,n\geqslant 1 \text{ such that}\\
\im s_{\alpha_n}\circ\cdots\circ s_{\alpha_1}(\zeta)\geqslant\widetilde{\MB}(\alpha_{n+1})
\end{array}
\right.
\right\}.
\end{equation}
\end{defi}

In the next section we show that $\widetilde{\MH}_N$ is independent of the choice of $f_0\in\IS_\alpha\cup\{Q_\alpha\}$ by proving that $\widetilde{\MH}_N$ coincides with the set of high type Herman numbers.

\begin{prop}\label{prop-equi-alpha}
The critical value $\cv=-\frac{4}{27}$ is contained in $\partial\Delta_0$ if and only if $\alpha\in\widetilde{\MH}_N$.
\end{prop}

\begin{proof}
For each $\zeta\in\gamma_\alpha\setminus\{1\}$ and $n\geqslant 1$, we denote
\begin{equation}
\zeta_n:=s_{\alpha_n}\circ\cdots\circ s_{\alpha_1}(\zeta).
\end{equation}
Suppose $\alpha\in\widetilde{\MH}_N$. Then there exists $n\geqslant 1$ such that $\im\zeta_n\geqslant\widetilde{\MB}(\alpha_{n+1})$. By \eqref{equ-y-n-hat-2} and the choice of $M$ in \eqref{equ:M}, we have $\Phi_n^{-1}(\zeta_n)\in\Delta_n$ and hence $\Phi_0^{-1}(\zeta)\in\Delta_0$. Therefore, $\Gamma_\alpha\setminus\{\cv\}=\Phi_0^{-1}(\gamma_\alpha\setminus\{1\})$ is contained in $\Delta_0$ and $\cv\in\partial\Delta_0$.

Suppose $\alpha\in\MB_N$ and $\cv\in\partial\Delta_0$. By Lemma \ref{lema-in-disk}, we have $\Phi_0^{-1}(\zeta)\in\Delta_0\cap\Gamma_\alpha$. According to Lemma \ref{lema-eventually-above}, there exists an integer $n\geqslant 1$ so that $\im \zeta_n\geqslant\widetilde{\MB}(\alpha_{n+1})$. This implies that $\alpha\in\widetilde{\MH}_N$.
\end{proof}

\section{Optimality of Herman condition}

Herman condition is not easy to verify in general. Yoccoz gave this condition an arithmetic characterization so that one can check easily whether an irrational number is of Herman type. In this section, we first recall Yoccoz's characterization and then prove that under the high type condition, an irrational number is of Herman type if and only if it belongs to the set $\widetilde{\MH}_N$ defined in \S\ref{subsec-new-class}.

\subsection{Yoccoz's characterization on $\MH$}\label{subsec-Yoc}

For $\alpha\in(0,1)$ and $x\in\R$, define
\begin{equation}
r_\alpha(x):=
\left\{
\begin{aligned}
& \frac{1}{\alpha}\Big(x-\log\frac{1}{\alpha}+1\Big)  & ~~~\text{if}\quad x\geqslant \log\frac{1}{\alpha},\\
& e^x & ~~~\text{if}\quad x< \log\frac{1}{\alpha}.
\end{aligned}
\right.
\end{equation}
The map $r_\alpha$ is of class $C^1$ on $\R$, satisfying $r_\alpha(\log\frac{1}{\alpha})=r_\alpha'(\log\frac{1}{\alpha})=\frac{1}{\alpha}$, $x+1\leqslant r_\alpha(x)\leqslant e^x$ for all $x\in\R$, and $r_\alpha'(x)\geqslant 1$ for all $x\geqslant 0$.

\medskip
For an irrational number $\alpha\in(0,1)$, we use $(\alpha_n)_{n\geqslant 0}$ to denote the sequence of irrationals defined as in \eqref{equ-gauss}.
Let $\MB(\alpha)$ be the Brjuno sum of $\alpha$ (see \eqref{equ-Brjuno-sum-Yoccoz}). A Brjuno number $\alpha$ is a Herman number (or belongs to Herman type) if every orientation-preserving analytic circle diffeomorphism of rotation number $\alpha$ is analytically conjugate to a rigid rotation.
Let $\MH$ be the set of all Herman numbers.

\begin{thm}[{\cite[\S 2.5]{Yoc02}}]\label{thm-Yoccoz}
Herman condition has the following arithmetic characterization:
\begin{equation}
\MH=\big\{\alpha\in\MB: \forall\, m\geqslant 0, \,\exists\, n>m \text{ such that } r_{\alpha_{n-1}}\circ\cdots\circ r_{\alpha_m}(0)\geqslant \MB(\alpha_n)\big\}.
\end{equation}
\end{thm}

\subsection{Two conditions are equivalent}\label{subsec-equi-irrat}

In this subsection, we prove that the set of Herman numbers is equal to $\widetilde{\MH}_N$ defined in \S\ref{subsec-new-class} under the high type condition.

\begin{lem}[{\cite[Lemma 4.9]{Yoc02}}]\label{lema-Yoccoz}
Let $\alpha$ be irrational and $x\geqslant 0$. Then $\alpha\not\in\MH$ if and only if there exist $m$ and an infinite set $I=I(m,x,\alpha)\subset\N$ such that, for all $k\in I$, we have
\begin{equation}
r_{\alpha_{m+k-1}}\circ\cdots\circ r_{\alpha_m}(x)<\log\tfrac{1}{\alpha_{m+k}}.
\end{equation}
\end{lem}

Let $D_4$ and $D_5>1$ be the constants introduced in Lemma \ref{lema-key-esti-inverse}.

\begin{defi}
For $\alpha\in(0,1)$ and $y\in\R$, we define
\begin{equation}\label{equ-s-alpha-star}
\overline{s}_\alpha(y):=
\left\{
\begin{aligned}
& \frac{1}{\alpha}\Big(y-\frac{1}{2\pi}\log\frac{1}{\alpha}+D_5\Big)  & ~~~\text{if}\quad y\geqslant \frac{1}{2\pi}\log\frac{1}{\alpha}+D_4,\\
& e^{D_5}\, e^{2\pi y} & ~~~\text{if}\quad y< \frac{1}{2\pi}\log\frac{1}{\alpha}+D_4.
\end{aligned}
\right.
\end{equation}
\end{defi}

Let $\gamma_\alpha=\gamma_{\alpha_0}$ be the unbounded arc defined in \eqref{equ-gamma-n} and $s_{\alpha_n}:=\Phi_n\circ \Expo:\gamma_{n-1}\to\gamma_n$ the map defined in \eqref{equ-s-alpha-n}.
By Lemma \ref{lema-key-esti-inverse} and the definition of $\overline{s}_\alpha$, we have the following immediate result.

\begin{lem}\label{lema-less-than}
For each $\alpha\in\MB_N$ and $\zeta\in\gamma_\alpha$, we have
\begin{equation}
\im s_\alpha(\zeta)\leqslant \overline{s}_\alpha(\im \zeta).
\end{equation}
\end{lem}

Define $\MH_N:=\MH\cap\MB_N$.

\begin{lem}\label{lem-subset-alpha}
We have $\widetilde{\MH}_N\subset\MH_N$.
\end{lem}

\begin{proof}
Assume by contradiction that $\alpha\in\widetilde{\MH}_N\setminus\MH_N$. Define
\begin{equation}\label{equ:C-0}
C_0:=8\pi e^{D_5+2\pi D_4}.
\end{equation}
By Lemma \ref{lema-Yoccoz}, for the number $2C_0$, there exist $m\geqslant 1$ and an infinite subset $I=I(m,2C_0,\alpha)$ of $\N$ such that for all $k\in I$, we have
\begin{equation}\label{equ-not-Herman}
r_{\alpha_{m+k-1}}\circ\cdots\circ r_{\alpha_m}(2C_0)<\log\tfrac{1}{\alpha_{m+k}}.
\end{equation}
Denote $x_{m-1}:=2C_0$ and $y_{m-1}:=1$. For $k\geqslant 1$, we define
\begin{equation}
x_{m+k-1}:=r_{\alpha_{m+k-1}}\circ\cdots\circ r_{\alpha_m}(2C_0) \text{\quad and\quad} y_{m+k-1}:=\overline{s}_{\alpha_{m+k-1}}\circ\cdots\circ \overline{s}_{\alpha_m}(1),
\end{equation}
where $\overline{s}_{\alpha_n}$ is the map defined in \eqref{equ-s-alpha-star}. We claim that
\begin{equation}\label{equ-key-inequality}
x_{m+k-1}\geqslant 2\pi y_{m+k-1}+C_0 \text{\quad for all } k\geqslant 0.
\end{equation}

Assume that \eqref{equ-key-inequality} holds temporarily. Since $\gamma_\alpha$ is an arc starting at the point $1$ and finally going up to the infinity, there exists $\zeta\in\gamma_{\alpha_{m-1}}$ so that $\im \zeta=1$. For $k\geqslant 1$, we denote
\begin{equation}
\zeta_{m+k-1}:=s_{\alpha_{m+k-1}}\circ\cdots\circ s_{\alpha_m}(\zeta),
\end{equation}
where each $s_{\alpha_n}$ is defined in \eqref{equ-s-alpha-n}. By Lemma \ref{lema-less-than}, we have $y_{m+k-1}\geqslant \im \zeta_{m+k-1}$ for all $k\geqslant 1$.

Since $\alpha\in\widetilde{\MH}_N$, by the definition of $\widetilde{\MH}_N$ and Lemma \ref{lema-all-greater}, there exists an integer $k_0\geqslant 1$ such that for all $k\geqslant k_0$, one has
\begin{equation}
y_{m+k-1}\geqslant \im \zeta_{m+k-1}\geqslant \widetilde{\MB}(\alpha_{m+k})=\frac{\MB(\alpha_{m+k})}{2\pi}+M>\frac{1}{2\pi}\log\frac{1}{\alpha_{m+k}}+M.
\end{equation}
On the other hand, since $\alpha\not\in\MH_N$, by \eqref{equ-not-Herman} there exists $k\in I$ with $k\geqslant k_0$ such that $x_{m+k-1}<\log\tfrac{1}{\alpha_{m+k}}$. This is a contradiction since by \eqref{equ-key-inequality} we have $x_{m+k-1}\geqslant 2\pi y_{m+k-1}+C_0>\log\tfrac{1}{\alpha_{m+k}}$. Hence it suffices to prove the claim \eqref{equ-key-inequality}.

\medskip
Obviously, \eqref{equ-key-inequality} is true when $k=0$ since $C_0\geqslant 2\pi$. Suppose $x_{m+k-1}\geqslant 2\pi y_{m+k-1}+C_0$ for some $k\geqslant 0$. It suffices to obtain $x_{m+k}\geqslant 2\pi y_{m+k}+C_0$. The arguments are divided into following three cases.

\medskip
\textbf{Case I}: Suppose $x_{m+k-1}<\log\frac{1}{\alpha_{m+k}}$ and $y_{m+k-1}<\frac{1}{2\pi}\log\frac{1}{\alpha_{m+k}}+D_4$. By \eqref{equ:C-0}, we have $C_0>2 (D_5+\log(2\pi))$ and hence $e^{y+C_0}>e^{y+D_5+\log(2\pi)}+C_0$ for any $y\geqslant 1$. Therefore,
\begin{equation}
\begin{split}
x_{m+k}
=&~e^{x_{m+k-1}}\geqslant e^{2\pi y_{m+k-1}+C_0}> e^{2\pi y_{m+k-1}+D_5+\log(2\pi)}+C_0\\
=&~2\pi\,\overline{s}_{\alpha_{m+k}}(y_{m+k-1})+C_0=2\pi y_{m+k}+C_0.
\end{split}
\end{equation}

\textbf{Case II}: Suppose $x_{m+k-1}\geqslant\log\frac{1}{\alpha_{m+k}}$ and $y_{m+k-1}\geqslant\frac{1}{2\pi}\log\frac{1}{\alpha_{m+k}}+D_4$.
Then
\begin{equation}
\begin{split}
x_{m+k}&~=\frac{1}{\alpha_{m+k}}\Big(x_{m+k-1}-\log\frac{1}{\alpha_{m+k}}+1\Big)\\
&~\geqslant \frac{2\pi}{\alpha_{m+k}}\Big(y_{m+k-1}-\frac{1}{2\pi}\log\frac{1}{\alpha_{m+k}}+D_5\Big) +\frac{1}{\alpha_{m+k}}(C_0+1-2\pi D_5)\\
&~\geqslant 2\pi y_{m+k}+2(C_0+1-2\pi D_5)> 2\pi y_{m+k}+C_0.
\end{split}
\end{equation}

\textbf{Case III}: Suppose $x_{m+k-1}\geqslant\log\frac{1}{\alpha_{m+k}}$ and $y_{m+k-1}<\frac{1}{2\pi}\log\frac{1}{\alpha_{m+k}}+D_4$. We consider the following two subcases:

\medskip
Subcase (i): Suppose $2\pi y_{m+k-1}<\log\frac{1}{\alpha_{m+k}}-\frac{C_0}{4}$. Note that
\begin{equation}
x_{m+k}=\frac{1}{\alpha_{m+k}}\left(x_{m+k-1}-\log\frac{1}{\alpha_{m+k}}+1\right)\geqslant \frac{1}{\alpha_{m+k}}.
\end{equation}
Since $x_{m-1}=2C_0$, we have $x_{m+k}\geqslant \max\big\{2C_0,\frac{1}{\alpha_{m+k}}\big\}$. By \eqref{equ:C-0}, we have $C_0>4D_5+4\log(4\pi)$ and hence $2\pi e^{D_5-C_0/4}<1/2$. Then
\begin{equation}
\begin{split}
x_{m+k}\geqslant \max\left\{2C_0,\frac{1}{\alpha_{m+k}}\right\}
\geqslant &~ \frac{2\pi e^{D_5-C_0/4}}{\alpha_{m+k}}+C_0 \\
\geqslant &~ 2\pi e^{D_5} e^{2\pi y_{m+k-1}}+C_0=2\pi y_{m+k}+C_0.
\end{split}
\end{equation}

\medskip
Subcase (ii): Suppose $\log\frac{1}{\alpha_{m+k}}-\frac{C_0}{4}\leqslant 2\pi y_{m+k-1}<\log\frac{1}{\alpha_{m+k}}+2\pi D_4$. Then
\begin{equation}\label{equ-compare-1}
\begin{split}
&~\alpha_{m+k}\big(x_{m+k}-(2\pi y_{m+k}+C_0)\big) \\
=&~ x_{m+k-1}-\log\tfrac{1}{\alpha_{m+k}}+1-\alpha_{m+k}(2\pi e^{D_5} e^{2\pi y_{m+k-1}}+C_0) \\
\geqslant &~ 2\pi y_{m+k-1}+C_0+1-\log\tfrac{1}{\alpha_{m+k}}-2\pi \alpha_{m+k} e^{D_5} e^{2\pi y_{m+k-1}}-C_0\alpha_{m+k}.
\end{split}
\end{equation}
For $\alpha\in(0,1/2]$, we consider the following continuous function:
\begin{equation}
h(t):=t+C_0+1-\log\tfrac{1}{\alpha}-2\pi \alpha e^{D_5} e^t-C_0\alpha, \quad \text{where } t\in\R.
\end{equation}
Then $h'(t)=1-2\pi \alpha e^{D_5} e^t$. Hence $h$ is increasing on $(-\infty,\log\frac{1}{\alpha}-D_5-\log(2\pi)]$ and decreasing on $[\log\frac{1}{\alpha}-D_5-\log(2\pi),+\infty)$. By \eqref{equ:C-0} and a direct calculation, we have
\begin{equation}\label{equ-value-ends}
\begin{split}
h(\log\tfrac{1}{\alpha}-\tfrac{C_0}{4})=&~(\tfrac{3}{4}-\alpha)C_0+1-2\pi e^{D_5-C_0/4}>0,\text{ and } \\
h(\log\tfrac{1}{\alpha}+2\pi D_4)=&~(1-\alpha)C_0+2\pi D_4+1-2\pi e^{D_5+2\pi D_4}>0.
\end{split}
\end{equation}
By \eqref{equ-compare-1} and \eqref{equ-value-ends}, we have $x_{m+k}>2\pi y_{m+k}+C_0$.
This finishes the proof of the claim \eqref{equ-key-inequality} and the lemma holds.
\end{proof}

Let $D_3>0$ be the constant introduced in Lemma \ref{lema-key-estimate-lp}.

\begin{defi}
For $\alpha\in(0,1)$ and $y\in\R$, we define
\begin{equation}\label{equ-s-alpha-star-1}
\underline{s}_{\,\alpha}(y):=
\left\{
\begin{aligned}
& \frac{1}{\alpha}\Big(y-\frac{1}{2\pi}\log\frac{1}{\alpha}-D_3\Big)  & ~~~\text{if}\quad y\geqslant \frac{1}{2\pi}\log\frac{1}{\alpha}+D_3+1,\\
& e^{-D_5}\, e^{2\pi y}-3 & ~~~\text{if}\quad y< \frac{1}{2\pi}\log\frac{1}{\alpha}+D_3+1.
\end{aligned}
\right.
\end{equation}
\end{defi}

\begin{lem}\label{lema-less-than-1}
For each $\alpha\in\MB_N$ and $\zeta\in\gamma_\alpha$, we have
\begin{equation}
\underline{s}_{\,\alpha}(\im \zeta)\leqslant \im s_\alpha(\zeta).
\end{equation}
\end{lem}

\begin{proof}
It follows from the proof of Lemma \ref{lema-key-esti-inverse} that $D_4=D_3+1$. Moreover, we choose $D_5=D_3$ in the proof of Lemma \ref{lema-key-esti-inverse}(a). Then this lemma follows immediately from Lemma \ref{lema-key-esti-inverse} and the definition of $\underline{s}_{\,\alpha}$.
\end{proof}

\begin{lem}\label{lem-subset-alpha-yy}
We have $\MH_N\subset\widetilde{\MH}_N$.
\end{lem}

\begin{proof}
The proof is similar to that of Lemma \ref{lem-subset-alpha}. Suppose $\alpha\in\MH_N\setminus\widetilde{\MH}_N$ by contradiction. Since $\alpha\not\in\widetilde{\MH}_N$, by the definition of $\widetilde{\MH}_N$, there exist a point $\zeta\in\gamma_\alpha\setminus\{1\}$ and an infinite sequence $(n_k)_{k\in\N}$ such that
\begin{equation}\label{equ-im-s-alpha-n}
\im \zeta_{n_k}<\widetilde{\MB}(\alpha_{n_k+1}),
\end{equation}
where
\begin{equation}
\zeta_n:=s_{\alpha_n}\circ\cdots\circ s_{\alpha_1}(\zeta) \text{\quad for all } n\in\N.
\end{equation}
By the uniform contraction with respect to the hyperbolic metric as in the proof of Proposition \ref{prop-tend-to-bdy} and Lemma \ref{lem-Jordan-arc}, there exists an integer $m\geqslant 1$ such that
\begin{equation}
\zeta_{m-1}\in\gamma_{m-1} \text{\quad and\quad}
\im \zeta_{m-1}\geqslant 2C_0,
\end{equation}
where $C_0>2M$ is a large number and $M\geqslant 1$ is introduced in the definition of $\widetilde{\MB}(\alpha_n)$.
Then by \eqref{equ-im-s-alpha-n} there exists an infinite subset $I'=I'(m,\zeta,\alpha)$ of $\N$ such that for all $k\in I'$, we have
\begin{equation}\label{equ-not-in-H}
\im \zeta_{m+k-1}<\widetilde{\MB}(\alpha_{m+k}).
\end{equation}

Since $\alpha\in\MH_N$, by Theorem \ref{thm-Yoccoz}, there exists $k_0=k_0(m)\geqslant 1$ such that $r_{\alpha_{m+k_0-1}}\circ\cdots\circ r_{\alpha_m}(0)\geqslant \MB(\alpha_{m+k_0})$. A direct calculation shows that for all $k\geqslant k_0$, one has
\begin{equation}\label{equ-not-Herman-yy}
r_{\alpha_{m+k-1}}\circ\cdots\circ r_{\alpha_m}(0)\geqslant \MB(\alpha_{m+k}).
\end{equation}
Denote $x_{m-1}:=0$ and $y_{m-1}:=2C_0$. For $k\geqslant 1$, we define
\begin{equation}
x_{m+k-1}:=r_{\alpha_{m+k-1}}\circ\cdots\circ r_{\alpha_m}(0) \text{\quad and\quad} y_{m+k-1}:=\underline{s}_{\,\alpha_{m+k-1}}\circ\cdots\circ \underline{s}_{\,\alpha_m}(2C_0),
\end{equation}
where $\underline{s}_{\,\alpha_n}$ is the map defined in \eqref{equ-s-alpha-star-1}. We claim that if $C_0$ is large enough, then
\begin{equation}\label{equ-key-inequality-yy}
2\pi y_{m+k-1}\geqslant x_{m+k-1}+C_0 \text{\quad for all } k\geqslant 0.
\end{equation}

Assume that \eqref{equ-key-inequality-yy} holds temporarily. By Lemma \ref{lema-less-than-1}, we have $y_{m+k-1}\leqslant \im \zeta_{m+k-1}$ for all $k\geqslant 1$. By \eqref{equ-not-in-H}, there exists an integer $k\in I'$ with $k\geqslant k_0$ such that
\begin{equation}
y_{m+k-1}\leqslant \im \zeta_{m+k-1}< \widetilde{\MB}(\alpha_{m+k})=\frac{\MB(\alpha_{m+k})}{2\pi}+M.
\end{equation}
On the other hand, by \eqref{equ-not-Herman-yy}, we have $x_{m+k-1}\geqslant\MB(\alpha_{m+k})$.
However, by \eqref{equ-key-inequality-yy} we have $x_{m+k-1}\le 2\pi y_{m+k-1}-C_0<\MB(\alpha_{m+k})$, which is a contradiction.
Hence it suffices to prove the claim \eqref{equ-key-inequality-yy}.

Obviously, \eqref{equ-key-inequality-yy} is true when $k=0$. Suppose $2\pi y_{m+k-1}\geqslant x_{m+k-1}+C_0$ for some $k\geqslant 0$.
Then one can divide the arguments into three cases as in Lemma \ref{lem-subset-alpha} to obtain $2\pi y_{m+k}\geqslant x_{m+k}+C_0$.
We omit the details since the rest proof is completely the same.
\end{proof}

\begin{rmk}
In fact, if $\alpha\in\MH_N$, then according to \cite{Ghy84} and \cite{Her85}, the boundary of the Siegel disk of each $f\in\IS_\alpha\cup\{Q_\alpha\}$ contains the unique critical value $-\frac{4}{27}$. This implies that $\alpha\in\widetilde{\MH}_N$ by Proposition \ref{prop-equi-alpha}. Therefore in this way we also obtain $\MH_N\subset\widetilde{\MH}_N$.
\end{rmk}

\begin{proof}[{Proof of the second part of the Main Theorem}]
Let $\alpha\in\HT_N$ be an irrational number of sufficiently high type. By Lemmas \ref{lem-subset-alpha} and \ref{lem-subset-alpha-yy}, $\alpha\in\MH_N$ if and only if $\alpha\in\widetilde{\MH}_N$. By Proposition \ref{prop-equi-alpha}, $\alpha\in\widetilde{\MH}_N$ if and only if $\cv=f(\cp_f)\in\partial\Delta_f$, where $\Delta_f$ is the Siegel disk of $f\in\IS_\alpha\cup \{Q_\alpha\}$ and $\cp_f$ is the unique critical point of $f$. Therefore, $\alpha\in\MH_N$ if and only if $\cp_f\in\partial\Delta_f$.
\end{proof}

\appendix
\section{Some calculations in Fatou coordinate planes}\label{sec-arc-straight}

In this appendix we give the proof of Lemma \ref{lema-height} based on some estimates in \cite{IS08}. Let $0<\alpha<1/2$. Define
\begin{equation}
Y:=\Big\{w=x+ y\ii\in\C:-\frac{1}{2\pi\alpha}\Big(\arccos\frac{\sqrt{3}}{2e^{2\pi\alpha y}}-\frac{\pi}{6}\Big)<x<\frac{2}{3\alpha} \text{ and } y>1\Big\}
\end{equation}
and $R:=\tfrac{4}{27}e^{3\pi}$ (see Figure \ref{Fig_Y-w}).

\begin{figure}[!htpb]
  \setlength{\unitlength}{1mm}
  \centering
  \includegraphics[width=120mm]{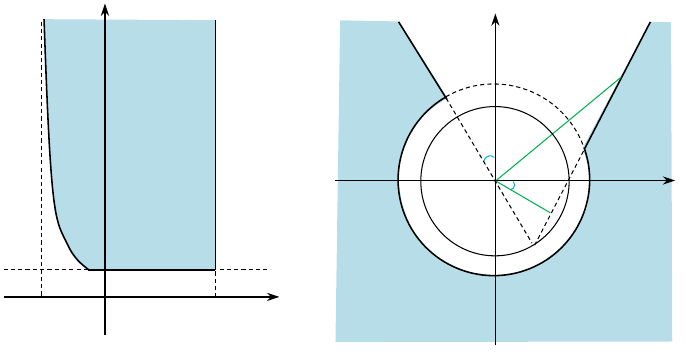}
  \put(-98,40){$Y$}
  \put(-84,6){$\tfrac{2}{3\alpha}$}
  \put(-116,6){$-\tfrac{1}{6\alpha}$}
  \put(-101,6){$0$}
  \put(-101,12){$1$}
  \put(-36,27){$0$}
  \put(-20,27){$1$}
  \put(-16,27){$e^{2\pi\alpha}$}
  \put(-10,48){$\xi$}
  \put(-16,6){$e^{-2\pi\ii\alpha w}$}
  \put(-87,18){$w$}
  \put(-29,28.3){\tiny{$\pi/6$}}
  \put(-36,36.7){\small{$\frac{\pi}{6}$}}
  \caption{The domain $Y$ and its image under $w\mapsto e^{-2\pi\ii\alpha w}$.}
  \label{Fig_Y-w}
\end{figure}

\begin{lem}\label{lema-calcu}
There exists $\varepsilon'>0$ such that for all $f\in\IS_\alpha$ with $\alpha\in(0,\varepsilon']$,
\begin{equation}
\tau_f(Y)\subset\D(0,R)\setminus [0,R),
\end{equation}
where $\tau_f:\C\to\EC\setminus\{0,\sigma_f\}$ is the universal covering defined in \eqref{equ-tau}.
\end{lem}

\begin{proof}
By a direct calculation, we have
\begin{equation}
\begin{split}
\{e^{-2\pi\ii\alpha w}: w\in Y\}
= &~\Big\{\xi\in\C: |\xi|>e^{2\pi\alpha} \text{ and } -\frac{4\pi}{3}<\arg\xi<\arccos\frac{\sqrt{3}}{2|\xi|}-\frac{\pi}{6}\Big\}\\
= &~\C\setminus\Big(\overline{\D}(0,e^{2\pi\alpha})\cup\Big\{\xi\in\C:\frac{\pi}{3}\leqslant\arg\Big(\xi-\frac{1-\sqrt{3}\ii}{2}\Big)\leqslant\frac{2\pi}{3}\Big\}\Big).
\end{split}
\end{equation}
Since $4\pi\alpha/(3R)<e^{2\pi\alpha}-1$, we have (see Figure \ref{Fig_Y-w})
\begin{equation}
e^{-2\pi\ii\alpha w}\in
\C\setminus\Big(\overline{\D}\big(1,\frac{4\pi\alpha}{3R}\big)
\cup\Big\{\xi\in\C:\frac{\pi}{3}\leqslant\arg\big(\xi-1\big)\leqslant\frac{2\pi}{3}\Big\}\Big).
\end{equation}
This implies that
\begin{equation}\label{equ-calcu-1}
\frac{1}{1-e^{-2\pi\ii\alpha w}}\in\D\Big(0,\frac{3R}{4\pi\alpha}\Big)\setminus\Big\{\xi\in\C:\frac{\pi}{3}\leqslant\arg\xi\leqslant\frac{2\pi}{3}\Big\}.
\end{equation}

Note that $\arcsin x\leqslant \tfrac{\pi}{3}x$ for $0\leqslant x\leqslant 1/2$. By \cite[Main Theorem 1(a)]{IS08}, $|f_0''(0)-4.91|\leqslant 1.14$ for all $f_0\in\IS_0$. Hence $|\arg f_0''(0)|<\arcsin\tfrac{1}{3}\leqslant\tfrac{\pi}{9}$ and
\begin{equation}
-\frac{4\pi\ii\alpha}{f_0''(0)}\in
\Big\{z\in\C:\frac{4\pi\alpha}{7}< |z|< \frac{8\pi\alpha}{7}\text{ and }\frac{25\pi}{18}<\arg z<\frac{29\pi}{18}\Big\}.
\end{equation}
By \eqref{equ-sigma-f} and the pre-compactness of $\IS_\alpha$, there exists a small $\varepsilon'>0$ such that for all $f\in\IS_\alpha$ with $\alpha\in(0,\varepsilon']$, then
\begin{equation}\label{equ-tau-est-2}
\sigma_f\in\Big\{z\in\C:\frac{\pi\alpha}{2}< |z|< \frac{4\pi\alpha}{3}\text{ and }\frac{4\pi}{3}<\arg z<\frac{5\pi}{3}\Big\}.
\end{equation}
By \eqref{equ-calcu-1} and \eqref{equ-tau-est-2} we have
\begin{equation}
\tau_f(w)=\frac{\sigma_f}{1-e^{-2\pi\ii\alpha w}}\in\D(0,R)\setminus [0,R).
\end{equation}
The proof is complete.
\end{proof}

For each $C\geqslant 1$, we define a subset of $\mho$ (see \eqref{equ-E}):
\begin{equation}\label{equ-mho-1}
\mho_1(C):=\{\zeta\in\C:1/4<\re\zeta< 7/4 \text{ and } \im\zeta\geqslant C\}.
\end{equation}

\begin{lem}\label{lema-calcu-1}
There exist $C> 1$ and $\varepsilon''>0$ such that for all $f\in\IS_\alpha$ with $\alpha\in(0,\varepsilon'']$, we have
\begin{equation}
L_f^{-1}(\overline{\mho_1(C)})\subset Y.
\end{equation}
where $L_f:\widetilde{\MP}_f\to\C$ is the univalent map defined in \eqref{equ-L-f}.
\end{lem}

\begin{proof}
Let $D_2>0$ be introduced in Proposition \ref{prop-Cheraghi-L-f}. For $y>0$, we define
\begin{equation}
\varphi_1(y):=\log(2+\sqrt{y^2+(7/4)^2}).
\end{equation}
There exists a constant $C>0$ depending only on $D_2$ such that if $y\geqslant C$, then
\begin{equation}\label{equ:y-2D-2}
y-2D_2 \varphi_1(y)>1
\end{equation}
and
\begin{equation}\label{equ-y-arc}
\frac{y}{2\pi}\Big(\arccos\frac{\sqrt{3}}{2e^{2\pi}}-\frac{\pi}{6}\Big)-D_2\varphi_1(y)>0.
\end{equation}
Let $0<\alpha\leqslant 1/C$. By Proposition \ref{prop-Cheraghi-L-f}, we have $L_f^{-1}(\overline{\mho_1(C)})\subset X_1\cup X_2\cup X_3$, where
\begin{equation}
\begin{split}
X_1&~=\{x+y\ii:-D_2\log(1+\tfrac{1}{\alpha})\leqslant x\leqslant D_2\log(1+\tfrac{1}{\alpha})+\tfrac{7}{4} \text{ and }y\geqslant \tfrac{1}{\alpha}\},\\
X_2&~=\{x+y\ii:-D_2 \varphi_1(y)\leqslant x\leqslant D_2 \varphi_1(y)+\tfrac{7}{4} \text{ and }y\in[C,\tfrac{1}{\alpha}]\}
\end{split}
\end{equation}
and
\begin{equation}
X_3=\{x+y\ii:-D_2 \varphi_1(C)\leqslant x\leqslant D_2 \varphi_1(C)+\tfrac{7}{4} \text{ and }y\in[C-D_2\varphi_1(C),C]\}.
\end{equation}

For $y>0$, we define a continuous function
\begin{equation}
\phi(y):=\frac{1}{2\pi\alpha}\Big(\arccos\frac{\sqrt{3}}{2e^{2\pi\alpha y}}-\frac{\pi}{6}\Big).
\end{equation}
Note that $\alpha\log(1+1/\alpha)$ is uniformly bounded above for $0<\alpha<1$. There exists a constant $\kappa_1>0$ depending only on $D_2$ such that if $\alpha\in(0, \kappa_1]$, then for $y\geqslant 1/\alpha$,
\begin{equation}
\phi(y)-D_2\log\Big(1+\frac{1}{\alpha}\Big)\geqslant \frac{1}{2\pi\alpha}\Big(\arccos\frac{\sqrt{3}}{2e^{2\pi}}-\frac{\pi}{6}\Big)-D_2\log\Big(1+\frac{1}{\alpha}\Big)>0.
\end{equation}
For $y\in[C,1/\alpha]$, we denote $t=2\pi\alpha y\in[2\pi\alpha C,2\pi]$. Then
\begin{equation}
\phi(y)-D_2\varphi_1(y)=y\psi(t)-D_2\varphi_1(y),
\end{equation}
where
\begin{equation}\label{equ-psi-t}
\psi(t):=\frac{1}{t}\Big(\arccos\frac{\sqrt{3}}{2e^t}-\frac{\pi}{6}\Big).
\end{equation}
A direct calculation\footnote{Note that $\psi(t)=\frac{1}{t}\int_{0}^{t}\big(\frac{4}{3}e^{2s}-1\big)^{-1/2}\textup{d}s$ can be seen as the average of the integral of $\widetilde{\psi}(s)=\big(\frac{4}{3}e^{2s}-1\big)^{-1/2}$ in the interval $(0,t)$. Since $s\mapsto\widetilde{\psi}(s)$ is strictly decreasing in $(0,+\infty)$, we conclude that $t\mapsto \psi(t)$ is also.} shows that $\psi(t)$ is decreasing on $(0,2\pi]$. By \eqref{equ-y-arc} we have
\begin{equation}
\phi(y)-D_2\varphi_1(y)\geqslant \frac{y}{2\pi}\Big(\arccos\frac{\sqrt{3}}{2e^{2\pi}}-\frac{\pi}{6}\Big)-D_2\varphi_1(y)>0.
\end{equation}
Finally, let $y\in[C-D_2\varphi_1(C),C]$ and we still denote $t=2\pi\alpha y$. A direct calculation shows that $\lim_{t\to 0^+}\psi(t)=\sqrt{3}$, where $\psi$ is defined in \eqref{equ-psi-t}.
By \eqref{equ:y-2D-2}, there exists a constant $\kappa_2>0$ depending only on $D_2$ such that if $\alpha\in(0, \kappa_2]$, then for $y\in[C-D_2\varphi_1(C),C]$ we have
\begin{equation}
\phi(y)-D_2\varphi_1(C)\geqslant y-D_2\varphi_1(C)\geqslant (C-D_2\varphi_1(C))-D_2\varphi_1(C)>1.
\end{equation}

Let $\kappa_3>0$ be a constant depending only on $D_2$ such that $D_2\varphi_1(\tfrac{1}{\alpha})+\tfrac{7}{4}<\tfrac{2}{3\alpha}$ for all $\alpha\in(0,\kappa_3]$. The proof is finished if we set $\varepsilon'':=\min\{1/C,\kappa_1,\kappa_2,\kappa_3\}$.
\end{proof}

\begin{proof}[{Proof of Lemma \ref{lema-height}}]
For $f_0\in\IS_0$, one can define $\MC_{f_0}$ and $\MC_{f_0}^\sharp$ as in \eqref{defi-C-f-alpha} similarly (Replacing $\MP_f$ and $\Phi_f$ there by $\MP_{attr,f_0}$ and $\Phi_{attr,f_0}$). We first show that \eqref{equ-M-c-subset} holds for $f_0\in\IS_0$ and then use an argument of continuity.

The Main Theorem 1 in \cite{IS08} was proved by transferring the parabolic fixed point $0$ of $f_0\in\IS_0$ to $\infty$ and a class corresponding to $\IS_0$ was defined (see \cite[\S 5.A]{IS08}):
\begin{equation}
\IS_0^Q:=
\left\{F=Q\circ\varphi^{-1}
\left|
\begin{array}{l}
\varphi:\EC\setminus E\to \EC\setminus\{0\} \text{ is univalent}, \\
\varphi(\infty)=\infty \text{ and } \varphi'(\infty)=1
\end{array}
\right.
\right\},
\end{equation}
where $E$ is the ellipse defined in \eqref{ellipse} and $Q(z)=z(1+\frac{1}{z})^6/(1-\frac{1}{z})^4$ is a parabolic map. Each map in this class has a parabolic fixed point at $\infty$, a critical point at $\cp_F:=\varphi(5+2\sqrt{6})$ and a critical value at $\cv_Q=27$ which is independent of $F$.

By \cite[Lemma 5.14(a)]{IS08}, $P$ and $Q$ are related by $Q=\psi_0^{-1}\circ P\circ \psi_1$, where $\psi_1(z)=-4z/(1+z)^2$ is defined in \eqref{U-and-psi-1} and $\psi_0(z)=-4/z$. By \cite[Proposition 5.3(c)]{IS08}, there exists a one-to-one correspondence between $\IS_0$ and $\IS_0^Q$. For $F\in\IS_0^Q$, one has natural definitions of the attracting petal $\MP_{attr,F}$, repelling petal $\MP_{rep,F}$, attracting Fatou coordinate $\Phi_{attr,F}$ and repelling Fatou coordinate $\Phi_{rep,F}$ etc based on the definitions relating to $f_0\in\IS_0$ in \S\ref{subsec-IS-class}. For example, the attracting Fatou coordinate of $F$ is defined as $\Phi_{attr,F}(z)=\Phi_{attr,f_0}\circ\psi_0(z)$.

\medskip
For $f_0\in\IS_0$, we define a topological triangle
\begin{equation}
\MQ_{f_0}:=\{z\in\MP_{attr,f_0}:\Phi_{attr,f_0}(z)\in \mho\}.
\end{equation}
In order to prove \eqref{equ-M-c-subset}, it is convenient to work in the corresponding dynamical plane of $F=\psi_0^{-1}\circ f_0\circ\psi_0\in\IS_0^Q$. Define
\begin{equation}
D_{0,F}:=\{z\in\MP_{attr,F}:0<\re\Phi_{attr,F}(z)< 1 \text{~and~} \im\Phi_{attr,F}(z)>-2\}
\end{equation}
and $D_{1,F}:=F(D_{0,F})$. By \cite[Proposition 5.7(e)]{IS08}, for $z\in \overline{D}_{0,F}$ we have
\begin{equation}
|z|\geqslant 0.05>27\,e^{-3\pi} \text{\quad and\quad} z\not\in \R_-.
\end{equation}
By \cite[Proposition 5.6(b)]{IS08}, for $z\in \overline{D}_{1,F}$ we have
\begin{equation}
|z|\geqslant \tfrac{25}{\sqrt{3}}\sin\tfrac{\pi}{3}=\tfrac{25}{2}>27\,e^{-3\pi} \text{\quad and\quad} z\not\in \R_-.
\end{equation}
Let $R=\tfrac{4}{27}e^{3\pi}$. We have
\begin{equation}\label{equ-omega-new}
\overline{D}_{0,F}\cup \overline{D}_{1,F}\subset \psi_0^{-1}\big(\D(0,R)\setminus [0,R)\big)=\C\setminus\big(\overline{\D}(0,27\,e^{-3\pi})\cup\R^-\big).
\end{equation}
By the definition of $\MQ_{f_0}$, we have
\begin{equation}
\psi_0^{-1}(\MQ_{f_0})=\{z\in\MP_{attr,F}:1/4<\re\Phi_{attr,F}(z)< 7/4\ \text{~and~} \im\Phi_{attr,F}(z)> -2\}.
\end{equation}
Therefore, by \eqref{equ-omega-new} we have $\psi_0^{-1}(\overline{\MQ}_{f_0}\setminus\{0\})\subset \overline{D}_{0,F}\cup \overline{D}_{1,F}$.
This implies that
\begin{equation}\label{equ:MQ-f0}
\overline{\MQ}_{f_0}\setminus\{0\}\subset \D(0,R)\setminus [0,R) \text{\quad for all } f_0\in\IS_0.
\end{equation}

Let $C>1$ be the constant introduced in Lemma \ref{lema-calcu-1} and $\mho_1=\mho_1(C)$ be defined in \eqref{equ-mho-1}. By Lemmas \ref{lema-calcu} and \ref{lema-calcu-1}, for every $f\in\IS_\alpha$ with $0<\alpha\leqslant \min\{\varepsilon',\varepsilon''\}$, we have
\begin{equation}
\Phi_f^{-1}(\overline{\mho}_1)=\tau_f\circ L_f^{-1}(\overline{\mho}_1)\subset \D(0,R)\setminus [0,R).
\end{equation}
Define
\begin{equation}
\mho_2:=\overline{\mho\setminus\mho_1}=\{\zeta\in\C:1/4\leqslant\re\zeta\leqslant 7/4 \text{ and } -2\leqslant\im\zeta\leqslant C\}.
\end{equation}
By \eqref{equ:MQ-f0}, the continuity of the Fatou coordinates in Proposition \ref{prop-BC-prop-12}(d) (see also \cite[Proposition 3.2.2]{Shi00a}) and the pre-compactness of $\IS_0$, there exists a constant $0<\varepsilon_4'\leqslant \min\{\varepsilon',\varepsilon''\}$ such that for all $f\in\IS_\alpha$ with $\alpha\in(0,\varepsilon_4']$, we have $\Phi_f^{-1}(\mho_2)\subset \D(0,R)\setminus [0,R)$ and hence $\overline{\MQ}_f\setminus\{0\}=\Phi_f^{-1}(\overline{\mho})\subset\D(0,R)\setminus [0,R)$.
\end{proof}

\bibliographystyle{amsalpha}
\bibliography{E:/Latex-model/Ref1}

\providecommand{\bysame}{\leavevmode\hbox to3em{\hrulefill}\thinspace}
\providecommand{\MR}{\relax\ifhmode\unskip\space\fi MR }
\providecommand{\MRhref}[2]{%
  \href{http://www.ams.org/mathscinet-getitem?mr=#1}{#2}
}
\providecommand{\href}[2]{#2}
\begin{thebibliography}{Rog92b}

\bibitem[ABC04]{ABC04}
A.~Avila, X.~Buff, and A.~Ch{\'{e}}ritat, \emph{Siegel disks with smooth
  boundaries}, Acta Math. \textbf{193} (2004), no.~1, 1--30.

\bibitem[AC18]{AC18}
A.~Avila and D.~Cheraghi, \emph{Statistical properties of quadratic polynomials
  with a neutral fixed point}, J. Eur. Math. Soc. (JEMS) \textbf{20} (2018),
  no.~8, 2005--2062.

\bibitem[AL22]{AL22}
A.~Avila and M.~Lyubich, \emph{Lebesgue measure of {F}eigenbaum {J}ulia sets},
  Ann. of Math. (2) \textbf{195} (2022), no.~1, 1--88.

\bibitem[BC07]{BC07}
X.~Buff and A.~Ch{\'{e}}ritat, \emph{How regular can the boundary of a
  quadratic {S}iegel disk be?}, Proc. Amer. Math. Soc. \textbf{135} (2007),
  no.~4, 1073--1080.

\bibitem[BC12]{BC12}
\bysame, \emph{Quadratic {J}ulia sets with positive area}, Ann. of Math. (2)
  \textbf{176} (2012), no.~2, 673--746.

\bibitem[BCR09]{BCR09}
X.~Buff, A.~Ch{\'{e}}ritat, and L.~Rempe, \emph{Arithmetical hedgehogs},
  manuscript, 2009.

\bibitem[BF18]{BF18}
A.~M. Benini and N.~Fagella, \emph{Singular values and bounded {S}iegel disks},
  Math. Proc. Cambridge Philos. Soc. \textbf{165} (2018), no.~2, 249--265.

\bibitem[Brj71]{Brj71}
A.~D. Brjuno, \emph{Analytic form of differential equations. {I}, {II}}, Trudy
  Moskov. Mat. Ob\v{s}\v{c}. \textbf{25} (1971), 119--262; ibid. \textbf{26}
  (1972), 199--239.

\bibitem[CC15]{CC15}
D.~Cheraghi and A.~Ch\'{e}ritat, \emph{A proof of the {M}armi-{M}oussa-{Y}occoz
  conjecture for rotation numbers of high type}, Invent. Math. \textbf{202}
  (2015), no.~2, 677--742.

\bibitem[CDY20]{CDY20}
D.~Cheraghi, A.~DeZotti, and F.~Yang, \emph{Dimension paradox of irrationally
  indifferent attractors}, arXiv: 2003.12340, 2020.

\bibitem[CE18]{CE18}
A.~Ch{\'{e}}ritat and A.~L. Epstein, \emph{Bounded type {S}iegel disks of
  finite type maps with few singular values}, Sci. China Math. \textbf{61}
  (2018), no.~12, 2139--2156.

\bibitem[Ch{\'{e}}06]{Che06}
A.~Ch{\'{e}}ritat, \emph{Ghys-like models for {L}avaurs and simple entire
  maps}, Conform. Geom. Dyn. \textbf{10} (2006), 227--256.

\bibitem[Ch{\'{e}}11]{Che11}
\bysame, \emph{Relatively compact {S}iegel disks with non-locally connected
  boundaries}, Math. Ann. \textbf{349} (2011), no.~3, 529--542.

\bibitem[Che13]{Che13}
D.~Cheraghi, \emph{Typical orbits of quadratic polynomials with a neutral fixed
  point: {B}rjuno type}, Comm. Math. Phys. \textbf{322} (2013), no.~3,
  999--1035.

\bibitem[Che19]{Che19}
\bysame, \emph{Typical orbits of quadratic polynomials with a neutral fixed
  point: non-{B}rjuno type}, Ann. Sci. \'{E}c. Norm. Sup\'{e}r. (4) \textbf{52}
  (2019), no.~1, 59--138.

\bibitem[Che22a]{Che22b}
\bysame, \emph{Topology of irrationally indifferent attractors}, arXiv:
  1706.02678v3, 2022.

\bibitem[Ch{\'{e}}22b]{Che22}
A.~Ch{\'{e}}ritat, \emph{Near parabolic renormalization for unicritical
  holomorphic maps}, Arnold Math. J. \textbf{8} (2022), no.~2, 169--270.

\bibitem[CR16]{CR16}
A.~Ch{\'{e}}ritat and P.~Roesch, \emph{Herman's condition and {S}iegel disks of
  bi-critical polynomials}, Comm. Math. Phys. \textbf{344} (2016), no.~2,
  397--426.

\bibitem[CS15]{CS15}
D.~Cheraghi and M.~Shishikura, \emph{Satellite renormalization of quadratic
  polynomials}, arXiv: 1509.07843, 2015.

\bibitem[DL22]{DL22}
D.~Dudko and M.~Lyubich, \emph{Uniform a priori bounds for neutral
  renormalization}, arXiv: 2210.09280, 2022.

\bibitem[Dou83]{Dou83}
A.~Douady, \emph{Syst\`emes dynamiques holomorphes}, Bourbaki seminar, {V}ol.
  1982/83, Ast\'{e}risque, vol. 105, Soc. Math. France, Paris, 1983,
  pp.~39--63.

\bibitem[Dou87]{Dou87}
\bysame, \emph{Disques de {S}iegel et anneaux de {H}erman}, Bourbaki seminar,
  {V}ol. 1986/87, Ast\'{e}risque, no. 152-153, Soc. Math. France, Paris, 1987,
  pp.~151--172.

\bibitem[Gey01]{Gey01}
L.~Geyer, \emph{Siegel discs, {H}erman rings and the {A}rnold family}, Trans.
  Amer. Math. Soc. \textbf{353} (2001), no.~9, 3661--3683.

\bibitem[Ghy84]{Ghy84}
{\'{E}}.~Ghys, \emph{Transformations holomorphes au voisinage d'une courbe de
  {J}ordan}, C. R. Acad. Sci. Paris S\'{e}r. I Math. \textbf{298} (1984),
  no.~16, 385--388.

\bibitem[G{\'{S}}03]{GS03}
J.~Graczyk and G.~{\'{S}}wi{\c{a}}tek, \emph{Siegel disks with critical points
  in their boundaries}, Duke Math. J. \textbf{119} (2003), no.~1, 189--196.

\bibitem[Her79]{Her79}
M.~R. Herman, \emph{Sur la conjugaison diff\'{e}rentiable des
  diff\'{e}omorphismes du cercle \`a des rotations}, Inst. Hautes \'{E}tudes
  Sci. Publ. Math. (1979), no.~49, 5--233.

\bibitem[Her85]{Her85}
\bysame, \emph{Are there critical points on the boundaries of singular
  domains?}, Comm. Math. Phys. \textbf{99} (1985), no.~4, 593--612.

\bibitem[Her86]{Her86}
\bysame, \emph{Conjugaison quasi-symm\'{e}trique des diff\'{e}omorphismes du
  cercle \`{a} des rotations et applications aux disques singuliers de
  {S}iegel}, manuscript, 1986.

\bibitem[Her87]{Her87}
\bysame, \emph{Conjugaison quasi-symm\'{e}trique des hom\'{e}omorphismes
  analytiques du cercle \`{a} des rotations}, preliminary manuscript, 1987.

\bibitem[IS08]{IS08}
H.~Inou and M.~Shishikura, \emph{The renormalization for parabolic fixed points
  and their perturbation},
  \href{https://www.math.kyoto-u.ac.jp/~mitsu/pararenorm/}{https://www.math.kyoto-u.ac.jp/$\sim$mitsu/pararenorm/},
  preprint, 2008.

\bibitem[KZ09]{KZ09}
L.~Keen and G.~Zhang, \emph{Bounded-type {S}iegel disks of a one-dimensional
  family of entire functions}, Ergodic Theory Dynam. Systems \textbf{29}
  (2009), no.~1, 137--164.

\bibitem[McM98]{McM98b}
C.~T. McMullen, \emph{Self-similarity of {S}iegel disks and {H}ausdorff
  dimension of {J}ulia sets}, Acta Math. \textbf{180} (1998), no.~2, 247--292.

\bibitem[Mil06]{Mil06}
J.~Milnor, \emph{Dynamics in one complex variable}, third ed., Annals of
  Mathematics Studies, vol. 160, Princeton University Press, Princeton, NJ,
  2006.

\bibitem[P{\'{e}}r92]{Per92}
R.~P{\'{e}}rez{-Marco}, \emph{Solution compl\`ete au probl\`eme de {S}iegel de
  lin\'{e}arisation d'une application holomorphe au voisinage d'un point fixe
  (d'apr\`es {J}.-{C}. {Y}occoz)}, Ast\'{e}risque (1992), no.~206, 273--310.

\bibitem[P{\'{e}}r97]{Per97}
\bysame, \emph{Fixed points and circle maps}, Acta Math. \textbf{179} (1997),
  no.~2, 243--294.

\bibitem[Pom75]{Pom75}
C.~Pommerenke, \emph{Univalent functions}, Vandenhoeck \& Ruprecht,
  G\"{o}ttingen, 1975.

\bibitem[PZ04]{PZ04}
C.~L. Petersen and S.~Zakeri, \emph{On the {J}ulia set of a typical quadratic
  polynomial with a {S}iegel disk}, Ann. of Math. (2) \textbf{159} (2004),
  no.~1, 1--52.

\bibitem[Rem04]{Rem04}
L.~Rempe, \emph{On a question of {H}erman, {B}aker and {R}ippon concerning
  {S}iegel disks}, Bull. London Math. Soc. \textbf{36} (2004), no.~4, 516--518.

\bibitem[Rem08]{Rem08}
\bysame, \emph{Siegel disks and periodic rays of entire functions}, J. Reine
  Angew. Math. \textbf{624} (2008), 81--102.

\bibitem[Rip94]{Rip94}
P.~J. Rippon, \emph{On the boundaries of certain {S}iegel discs}, C. R. Acad.
  Sci. Paris S\'{e}r. I Math. \textbf{319} (1994), no.~8, 821--826.

\bibitem[Rog92a]{Rog92a}
J.~T. Rogers, \emph{Is the boundary of a {S}iegel disk a {J}ordan curve?},
  Bull. Amer. Math. Soc. (N.S.) \textbf{27} (1992), no.~2, 284--287.

\bibitem[Rog92b]{Rog92b}
\bysame, \emph{Singularities in the boundaries of local {S}iegel disks},
  Ergodic Theory Dynam. Systems \textbf{12} (1992), no.~4, 803--821.

\bibitem[Rog98]{Rog98}
\bysame, \emph{Diophantine conditions imply critical points on the boundaries
  of {S}iegel disks of polynomials}, Comm. Math. Phys. \textbf{195} (1998),
  no.~1, 175--193.

\bibitem[Shi98]{Shi98}
M.~Shishikura, \emph{The {H}ausdorff dimension of the boundary of the
  {M}andelbrot set and {J}ulia sets}, Ann. of Math. (2) \textbf{147} (1998),
  no.~2, 225--267.

\bibitem[Shi00]{Shi00a}
\bysame, \emph{Bifurcation of parabolic fixed points}, The {M}andelbrot set,
  theme and variations, London Math. Soc. Lecture Note Ser., vol. 274,
  Cambridge Univ. Press, Cambridge, 2000, pp.~325--363.

\bibitem[Shi01]{Shi01}
\bysame, \emph{Herman's theorem on quasisymmetric linearization of analytic
  circle homemorphisms}, manuscript, 2001.

\bibitem[Sie42]{Sie42}
C.~L. Siegel, \emph{Iteration of analytic functions}, Ann. of Math. (2)
  \textbf{43} (1942), 607--612.

\bibitem[Yam08]{Yam08}
M.~Yampolsky, \emph{Siegel disks and renormalization fixed points}, Holomorphic
  dynamics and renormalization, Fields Inst. Commun., vol.~53, Amer. Math.
  Soc., Providence, RI, 2008, pp.~377--393.

\bibitem[Yan13]{Yan13}
F.~Yang, \emph{On the dynamics of a family of entire functions}, Acta Math.
  Sin. (Engl. Ser.) \textbf{29} (2013), no.~11, 2047--2072.

\bibitem[Yan21]{Yan21p}
\bysame, \emph{Parabolic and near-parabolic renormalizations for local degree
  three}, arXiv: 1510.00043v4, 2021.

\bibitem[Yoc88]{Yoc88}
J.-C. Yoccoz, \emph{Lin\'{e}arisation des germes de diff\'{e}omorphismes
  holomorphes de {$({\bf C}, 0)$}}, C. R. Acad. Sci. Paris S\'{e}r. I Math.
  \textbf{306} (1988), no.~1, 55--58.

\bibitem[Yoc95]{Yoc95}
\bysame, \emph{Th\'{e}or\`eme de {S}iegel, nombres de {B}runo et polyn\^{o}mes
  quadratiques}, Ast\'{e}risque (1995), no.~231, 3--88.

\bibitem[Yoc02]{Yoc02}
\bysame, \emph{Analytic linearization of circle diffeomorphisms}, Dynamical
  systems and small divisors ({C}etraro, 1998), Lecture Notes in Math., vol.
  1784, Springer, Berlin, 2002, pp.~125--173.

\bibitem[YZ01]{YZ01}
M.~Yampolsky and S.~Zakeri, \emph{Mating {S}iegel quadratic polynomials}, J.
  Amer. Math. Soc. \textbf{14} (2001), no.~1, 25--78.

\bibitem[Zak99]{Zak99}
S.~Zakeri, \emph{Dynamics of cubic {S}iegel polynomials}, Comm. Math. Phys.
  \textbf{206} (1999), no.~1, 185--233.

\bibitem[Zak10]{Zak10}
\bysame, \emph{On {S}iegel disks of a class of entire maps}, Duke Math. J.
  \textbf{152} (2010), no.~3, 481--532.

\bibitem[ZFS20]{ZFS20}
S.~Zhang, J.~Fu, and X.~Shi, \emph{Bounded type {S}iegel disks of a family of
  sine families}, J. Math. Anal. Appl. \textbf{488} (2020), no.~1, 124041, 12
  pp.

\bibitem[Zha05]{Zha05}
G.~Zhang, \emph{On the dynamics of {$e^{2\pi i\theta}\sin(z)$}}, Illinois J.
  Math. \textbf{49} (2005), no.~4, 1171--1179.

\bibitem[Zha11]{Zha11}
\bysame, \emph{All bounded type {S}iegel disks of rational maps are
  quasi-disks}, Invent. Math. \textbf{185} (2011), no.~2, 421--466.

\bibitem[Zha14]{Zha14}
\bysame, \emph{Polynomial {S}iegel disks are typically {J}ordan domains},
  arXiv: 1208.1881v3, 2014.

\bibitem[Zha16]{Zha16}
\bysame, \emph{On {PZ} type {S}iegel disks of the sine family}, Ergodic Theory
  Dynam. Systems \textbf{36} (2016), no.~3, 973--1006.

\end{thebibliography}

\end{document}